\documentclass[11pt]{amsart}
\usepackage{amssymb}
\usepackage{amsthm}
\usepackage{amsmath}
\usepackage[dvips]{epsfig}
\usepackage{graphicx}
\usepackage{hyperref}
\usepackage{color}
\usepackage{url}
\usepackage{tikz}
\usepackage{caption}
\usepackage{subcaption}
\usepackage{epstopdf}
\usepackage[arrow, matrix, curve]{xy}
\usepackage{marginnote}

\usetikzlibrary{decorations.markings}

\makeatletter\@addtoreset {equation}{section}\makeatother

\newtheorem{theorem}{Theorem}[section]

\newtheorem{lemma}{Lemma}[section]
\theoremstyle{remark}
\newtheorem{remark}{Remark}[section]
\theoremstyle{definition}

\theoremstyle{corollary}
\newtheorem{corollary}{Corollary}[section]

{\begin{trivlist} \item[]{\em Proof }}%
{\hspace*{\fill}$\rule{.3\baselineskip}{.35\baselineskip}$\end{trivlist}}

\begin{document}

\title[Convergence of Petviashvili's method]{\bf Convergence of Petviashvili's method near periodic waves
in the fractional Korteweg--de Vries equation}

\author{Uyen Le}
\address[U. Le]{Department of Mathematics and Statistics, McMaster University,
Hamilton, Ontario, Canada, L8S 4K1}
\email{leu@mcmaster.ca}

\author{Dmitry E. Pelinovsky}
\address[D. Pelinovsky]{Department of Mathematics and Statistics, McMaster University,
Hamilton, Ontario, Canada, L8S 4K1}
\email{dmpeli@math.mcmaster.ca}


\date{\today}
\maketitle

\begin{abstract}
Petviashvili's method has been successfully used for approximating of solitary waves in nonlinear evolution equations. It was
discovered empirically that the method may fail for approximating of periodic waves. We consider the case study of the fractional
Korteweg--de Vries equation and explain divergence of Petviashvili's method from unstable
eigenvalues of the generalized eigenvalue problem. We also show that a simple modification
of the iterative method after the mean value shift results in the unconditional convergence
of Petviashvili's method. The results are illustrated numerically for
the classical Korteweg--de Vries and Benjamin--Ono equations.
\end{abstract}

% journal choices: (1) SIMA, (2) J Comp. Appl Math., (3) Nonlinear Analysis A

\tableofcontents

\section{Introduction}

A robust iterative method for approximating solitary waves was proposed by V.I. Petviashvili in 1976 \cite{Pet}.
Since then, it has become a popular numerical toolbox \cite{yang} with many recent generalizations in \cite{LY07,LY08}
and in \cite{AD14,AD15,AD16}.

In the context of Euler equations for water waves, Petviashvili's iterative method turns out
to be very useful for computing the solitary gravity waves \cite{DC1,DC2}.
However, it has been found empirically that the iterative algorithm does not converge for periodic waves,
hence suitable generalizations were proposed in the case of infinite \cite{DLK} and finite \cite{DC3} depths.
The work \cite{DLK} explores the generalization of Petviashvili's method for non-power nonlinearities proposed originally in
\cite{LY07}. The work of \cite{DC3} relies on an iteration-dependent shift of the field variable to enforce positivity of the periodic wave,
after which the classical Petviashvili method can be employed.

In a setting of fractional Korteweg-de Vries (KdV) and extended Boussinesq equations,
another modification of the Petviashvili method was proposed in \cite{AD17,D18},
where an iteration-independent shift of the field variable was computed from the underlying equation.
Numerical results in \cite{AD17} illustrated convergence of the Petviashvili method
for the periodic waves after the shift.

The main purpose of this work is to explain analytically the failure of the classical Petviashvili method
for approximating of periodic waves and to prove convergence of the same method
after a suitable shift of the field variable. We consider the toy problem
given by the fractional KdV equation with a quadratic nonlinearity,
which is a simplified model arising from the Euler equations in the shallow limit  \cite{BBM}.
The fractional KdV equation is taken in the normalized form
\begin{equation}
\label{kdv}
u_t + 2 u u_x + (D_{\alpha} u)_x = 0,
\end{equation}
where $D_{\alpha}$ is a fractional derivative operator defined by its Fourier symbol
$$
\widehat{D_{\alpha} u}(\xi) = -|\xi|^{\alpha} \hat{u}(\xi), \quad \xi \in \mathbb{R}.
$$
The case $\alpha = 2$ corresponds to the classical KdV equation, whereas the case $\alpha = 1$
corresponds to the integrable BO (Benjamin--Ono) equation. Henceforth, we assume that $\alpha > 0$.

Global existence in the fractional KdV equation (\ref{kdv}) for the initial data in
the energy space $H^{\alpha/2}$ was proven in \cite{SautPilod} for $\alpha > 1/2$
and for $\alpha = 1/2$ and small data. More recently,
local existence for the initial data in $H^s$ was shown for $\alpha > 0$ and $s > 3/2 - 5 \alpha/4$
in \cite{Molinet}.

Existence and stability of periodic waves in the fractional KdV equation (\ref{kdv}) were analyzed
by using perturbative \cite{J} and variational \cite{C,BC,JH1} methods.
For the classical KdV and BO equations, stability of periodic waves was also
proven in \cite{NataliPava}. These results, especially perturbation expansions
in the limit of small wave amplitudes, are also useful
in our analysis of convergence of iterative methods near the periodic waves.

Periodic traveling waves are solutions of the fractional KdV equation (\ref{kdv}) in the form
$u(x,t) = \psi(x - ct)$, where $\psi$ is a periodic function in its argument
and $c$ is the speed parameter for the wave travelling to the right.
Without loss of generality, due to scaling and translation invariance
of the fractional KdV equation (\ref{kdv}), we scale the period of $\psi$ to $2\pi$ and translate $\psi$
to become an even function of its argument. Due to the Galilean invariance,
integration of the nonlinear equation for $\psi$
is performed with zero integration constant. All together,
the wave profile $\psi$ is a $2\pi$-periodic even solution
to the following boundary-value problem:
\begin{equation}
\label{ode-new}
(c - D_\alpha ) \psi =  \psi^2, \quad \psi \in H^{\alpha}_{\rm per}(-\pi,\pi).
\end{equation}
We say that the periodic wave has a {\em single-lobe} profile if there exist only one maximum and minimum
of $\psi$ on the period. For uniqueness of solutions, we place the maximum of $\psi$ at $x = 0$
and the minimum of $\psi$ at $x = \pm \pi$.

In addition to the waves travelling to the right, the fractional KdV equation (\ref{kdv}) has
also periodic traveling waves in the form $u(x,t) = \phi(x + ct)$, where $\phi$ is a $2\pi$-periodic even solution
to the following boundary-value problem:
\begin{equation}
\label{ode}
(c + D_{\alpha}) \phi + \phi^2 = 0, \quad \phi \in H^{\alpha}_{\rm per}(-\pi,\pi).
\end{equation}
A very simple formula connects the right-propagating waves with the left-propagating waves:
\begin{equation}
\label{phi-psi}
\phi(x) = -c + \psi(x).
\end{equation}
The wave profile $\phi$ is a solution to the boundary-value problem (\ref{ode}) with some $c > 0$
if and only if $\psi$ is a solution to the boundary-value problem (\ref{ode-new}) with the same $c > 0$.
Section \ref{sec-2} collects together some results on existence of solutions
to the boundary-value problems (\ref{ode-new}) and (\ref{ode}).

\begin{remark}
Although most of the previous works (see, e.g., \cite{NataliPava,C,BC,JH1,J}) are devoted
the right-propagating waves with profile $\psi$, there are no apriori reasons to prefer these
waves over the left-propagating waves with profile $\phi$. Perturbative expansions
for waves of small amplitudes are more easily developed for the left-propagating waves
with profile $\phi$ since they arise in the local bifurcation theory from linearization of the
zero equilibrium (see Theorem \ref{lemma-small} below). On the other hand, the proof of positivity
of the wave profile $\psi$ is developed easier from the boundary-value problem (\ref{ode-new}) (see Theorem \ref{lemma-positive} below).
\end{remark}

Let us now explain how Petviashvili's iterative methods can be employed in order to approximate solutions
to the boundary-value problems (\ref{ode-new}) and (\ref{ode}) numerically.
In fact, the most interesting interplay between convergent and divergent iterations
arises in the context of the boundary-value problem (\ref{ode}).

Suppose that $\phi \in H^{\alpha}_{\rm per}(-\pi,\pi)$ is a solution to the boundary-value problem (\ref{ode})
for some $c > 0$. For uniqueness of solutions, we always denote by $\phi$ the {\em single-lobe periodic solution}
in the sense of the definition above. The classical Petviashvili method for approximating of $\phi$ is defined as follows.

Consider $\mathcal{L}_{c,\alpha} := -c - D_{\alpha}$ as a linear operator in $L^2_{\rm per}(-\pi,\pi)$ with
the domain $H^{\alpha}_{\rm per}(-\pi,\pi)$ and define the Petviashvili quotient:
\begin{equation}
\label{rule}
M(w) := \frac{\langle \mathcal{L}_{c,\alpha} w, w \rangle}{\langle w^2,w \rangle}, \quad w \in H^{\alpha}_{\rm per}(-\pi,\pi),
\end{equation}
where $\langle \cdot, \cdot \rangle$ denotes the standard inner product in $L^2_{\rm per}(-\pi,\pi)$.
For $c \notin \{ 1, 2^{\alpha}, 3^{\alpha}, \dots \}$, for which the linear operator
$\mathcal{L}_{c,\alpha} : H^{\alpha}_{\rm per}(-\pi,\pi) \to L^2_{\rm per}(-\pi,\pi)$ is invertible,
and for any suitable initial guess $w_0 \in H^{\alpha}_{\rm per}(-\pi,\pi)$,
define a sequence $\{ w_n \}_{n \in \mathbb{N}}$ in $H^{\alpha}_{\rm per}(-\pi,\pi)$ by the iterative rule:
\begin{equation}
\label{T}
w_{n+1} = T_{c,\alpha}(w_n) := [M(w_n)]^2 \mathcal{L}_{c,\alpha}^{-1} (w_n^2), \quad n \in \mathbb{N}.
\end{equation}
Here we have selected the quadratic exponent of $M(w_n)$ so that $T_{c,\alpha}(w)$ is a homogeneous power function in $w$
of degree {\em zero}. This ensures the fastest convergence rate of the iterative method (\ref{T})
near a solution of the nonlinear equation (\ref{ode}) \cite{PelSt}.

As is well understood since the first proof of convergence in \cite{PelSt} (see also follow-up works in \cite{AD15,AD16,ChugPel-KdV,DS,LY07}),
convergence of the iterative method is analyzed from contraction of the linearized operator
at the fixed point $\phi \in H^{\alpha}_{\rm per}(-\pi,\pi)$
of $T_{c,\alpha}$. By Lemma 1.2 in \cite{PelSt}, the set of fixed points of $T_{c,\alpha}$ coincides with the set of
solutions to the boundary-value problem (\ref{ode}). Contraction of the corresponding linearized operator
is defined by the spectrum of the generalized eigenvalue problem
\begin{equation}
\label{gEp}
\mathcal{H}_{c,\alpha} v = \lambda \mathcal{L}_{c,\alpha} v, \quad v \in H^{\alpha}_{\rm per}(-\pi,\pi),
\end{equation}
where
\begin{equation}
\label{Jacobian}
\mathcal{H}_{c,\alpha} := -c - D_{\alpha} - 2 \phi
\end{equation}
is the associated linearized operator in $L^2_{\rm per}(-\pi,\pi)$ with the domain $H^{\alpha}_{\rm per}(-\pi,\pi)$.
Note that $\mathcal{H}_{c,\alpha}$ is the Jacobian operator for the boundary-value
problem (\ref{ode}), which also plays the crucial role in the stability analysis
of the travelling periodic waves \cite{NataliPava,JH1,J}.

Section \ref{sec-3} presents the main result on convergence of the iterative method (\ref{T}).
Here and in what follows, the following critical values of $\alpha$ are important:
\begin{equation}
\label{alpha}
\alpha_0 := \frac{\log 3}{\log 2} - 1, \qquad
\alpha_1 := \frac{\log 5}{\log 2} - 1,
\end{equation}
where $1/2 < \alpha_0 < 1 < \alpha_1 < 2$. The proof of the main result
is achieved by the count of unstable eigenvalues
in the generalized eigenvalue problem (\ref{gEp}) and by perturbative arguments.

\begin{theorem}
\label{theorem-main-1}
For every $c > 1$ and $\alpha \in (\alpha_0,2]$, there exists a unique
single-lobe solution $\phi \in H^{\alpha}_{\rm per}(-\pi,\pi)$ to the boundary-value
problem (\ref{ode}). If $c \gtrsim 1$, this unique solution is an unstable fixed point
of the iterative method (\ref{T}) for $\alpha \in (\alpha_0,\alpha_1)$ and
an asymptotically stable fixed point (up to a translation) for $\alpha \in (\alpha_1,2]$.
If $c > 2^{\alpha}$, this unique solution is an unstable fixed-point of the iterative method (\ref{T})
for $\alpha \in (\alpha_0,2]$.
\end{theorem}

\begin{remark}
Notation $c \gtrsim 1$ implies that there is $c_0 > 1$ near $1$ such that
the statement holds for every $c \in (1,c_0)$. The
unique solution to the boundary-value problem (\ref{ode})
exists also for $\alpha < \alpha_0$ but is located for
$c \lesssim 1$.
\end{remark}

\begin{remark}
The constraint $\alpha \leq 2$ is necessary to apply results of \cite{JH1} 
on existence of single-lobe solution $\phi$ and the non-degeneracy of the kernel of 
$\mathcal{H}_{c,\alpha}$ at $\phi$. The periodic wave $\phi$ may develop oscillations 
for $\alpha > 2$ and sufficiently large $c$, in which case methods of \cite{JH1} are not applicable. 
\end{remark}

\begin{remark}
Theorem \ref{theorem-main-1} implies that the iterative method (\ref{T}) diverges
from $\phi$ for the classical BO equation with $\alpha = 1$.
Although the iterative method (\ref{T}) converges to $\phi$ for the classical
KdV equation with $\alpha = 2$ for $c \gtrsim 1$, we show numerically that
it diverges from $\phi$ for $c \gtrapprox 2.3$. Instabilities of the iterative
method (\ref{T}) are explained by the unstable eigenvalues of the generalized
eigenvalue problem (\ref{gEp}).
\end{remark}

As is suggested by Theorem \ref{theorem-main-1}, the iterative method (\ref{T}) is
unsuccessful in approximating the solution $\phi$ to the boundary-value problem
(\ref{ode}). On the other hand, we can develop a similar method for the solution
$\psi$ of the equivalent boundary-value problem (\ref{ode-new}),
which is related to $\phi$ by the transformation (\ref{phi-psi}).
By setting $\tilde{\mathcal{L}}_{c,\alpha} := c - D_{\alpha}$, we denote
\begin{equation}
\label{rule-new}
\tilde{M}(w) := \frac{\langle \tilde{\mathcal{L}}_{c,\alpha} w, w \rangle}{\langle w^2,w \rangle}, \quad w \in H^{\alpha}_{\rm per}(-\pi,\pi),
\end{equation}
and define a sequence $\{ w_n \}_{n \in \mathbb{N}}$ in $H^{\alpha}_{\rm per}(-\pi,\pi)$
for any suitable initial guess $w_0 \in H^{\alpha}_{\rm per}(-\pi,\pi)$ by the iterative rule:
\begin{equation}
\label{T-new}
w_{n+1} = \tilde{T}_{c,\alpha}(w_n) := \left[ \tilde{M}(w_n) \right]^2 \tilde{\mathcal{L}}_{c,\alpha}^{-1} (w_n^2), \quad n \in \mathbb{N}.
\end{equation}
Contraction of the linearized operator
of the iterative rule (\ref{T-new}) is defined by the spectrum of the generalized eigenvalue problem
\begin{equation}
\label{gEp-new}
\tilde{\mathcal{H}}_{c,\alpha} v = \lambda \tilde{\mathcal{L}}_{c,\alpha} v, \quad v \in H^{\alpha}_{\rm per}(-\pi,\pi),
\end{equation}
where the new Jacobian operator for the boundary-value problem (\ref{ode-new})
is identical to the Jacobian operator (\ref{Jacobian}) of the boundary-value problem (\ref{ode}):
\begin{equation}
\label{Jacobian-new}
\tilde{\mathcal{H}}_{c,\alpha} := c - D_{\alpha} - 2 \psi = -c - D_{\alpha} - 2 (-c + \psi) = \mathcal{H}_c,
\end{equation}
where the transformation (\ref{phi-psi}) has been used.

Section \ref{sec-4} presents the main result on convergence of the iterative method (\ref{T-new}).
In addition to the count of unstable eigenvalues in the generalized eigenvalue problem (\ref{gEp-new}),
we use here positivity of the wave profile $\psi$ for solutions to the boundary-value problem (\ref{ode-new}).

\begin{theorem}
\label{theorem-main-2}
For every $c > 1$ and $\alpha \in (\alpha_0,2]$, there exists a unique
single-lobe solution $\psi \in H^{\alpha}_{\rm per}(-\pi,\pi)$ to the boundary-value
problem (\ref{ode-new}) such that $\psi(x) > 0$ for every $x \in [-\pi,\pi]$.
This unique solution is an asymptotically stable (up to a translation) fixed point of the iterative method (\ref{T-new})
for every $c > 1$ and $\alpha \in (\alpha_0,2]$.
\end{theorem}

\begin{remark}
The unconditional convergence of the iterative method (\ref{T-new}) compared to the iterative method (\ref{T})
has a well-known physical interpretation. The phase velocity of the linear waves of the fractional KdV equation (\ref{kdv})
on the zero background is strictly negative, hence the travelling wave $u(x,t) = \phi(x+ct)$ propagating to the left
is in resonance with the linear waves. On the other hand, the travelling wave on the constant background $b := -c < 0$
propagates to the right and avoids resonances with the linear waves on the background $b < 0$,
which still have negative phase velocity.
\end{remark}

\begin{remark}
The new iterative method (\ref{T-new}) can be considered as a modification of
the classical Petviashvili method (\ref{T}) after the shift of the field variable 
proposed in \cite{AD17}. The modified algorithm consists of three steps.
In the first step, the constant value $b$ is found from the constant
solution of the stationary problem (\ref{ode}). Solving $c b + b^2 = 0$ for
nonzero $b$ yields $b = -c$. In the second step, the change of variables $\phi = b + \psi$ transforms
the original problem (\ref{ode}) to the new problem (\ref{ode-new}), which is confirmed
from the transformation formula (\ref{phi-psi}) since $b = -c$.
Finally, the third step is the iterative method for the transformed problem
(\ref{ode-new}), which is defined by the new iterative operator $\tilde{T}_{c,\alpha}$ in (\ref{T-new}).
\end{remark}

\begin{remark}
In the case of solitary waves, the boundary-value problem (\ref{ode}) for $\phi$ and $c > 0$ admits no
solutions and the iterative method (\ref{T}) cannot be defined since $\mathcal{L}_{c,\alpha}$ is not invertible
in $L^2(\mathbb{R})$ for $c > 0$. On the other hand, the boundary-value problem (\ref{ode-new}) for $\psi$ and $c > 0$
admits solitary wave solutions and the iterative method (\ref{T-new}) is well-defined to approximate
this solution, as shown numerically in \cite{D18}.
\end{remark}

\section{Periodic waves of the fractional KdV equation}
\label{sec-2}

We collect together some results on existence of periodic wave solutions to the boundary-value
problems (\ref{ode-new}) and (\ref{ode}). Some of the previous results have been improved and 
we specify explicitly where the improvement has been made.
Section \ref{sec-21} presents results on the small-amplitude limit of the periodic waves with profile $\phi$.
Sections \ref{sec-22} and \ref{sec-23} collects together explicit expressions
for the periodic waves in the classical KdV and BO equations, respectively.
Section \ref{sec-24} gives results on the positivity of the wave profile $\psi$.

\subsection{Small-amplitude limit of the periodic waves}
\label{sec-21}

The following result reports on existence of the periodic wave $\phi$ of the boundary-value
problem (\ref{ode}) in the small-amplitude limit.
The small-amplitude periodic waves bifurcate from the constant zero solution
to the boundary-value problem (\ref{ode}).
The construction of the small-amplitude periodic waves is nearly identical to Lemma 2.1 in \cite{J} subject to the following two changes.
First, the constant of integration is set to zero thanks to the Galilean invariance, while
in \cite{J} the constant was carried as an additional (redundant) parameter of the problem.
Second, the speed $c$ is used as the main parameter of the periodic solution while the period is set to $2\pi$,
whereas in \cite{J} $c$ was set to $1$ and the period was taken as the main parameter of the periodic solution.

Although the formal computations of the periodic waves in the small-amplitude limit
hold for every $\alpha > 0$, the justification of the perturbative
expansions requires $\alpha > 1/2$, for which $H^{\alpha}_{\rm per}(-\pi,\pi)$
is a Banach algebra with respect to multiplication with a continuous embedding into
$L^{\infty}_{\rm per}(-\pi,\pi)$. A typical justification of the perturbative expansions
is based on the method of Lyapunov--Schmidt reductions which requires smoothness of the nonlinear mappings. This smoothness
is guaranteed in $H^{\alpha}_{\rm per}(-\pi,\pi)$ with $\alpha > 1/2$.
Since refinement to $\alpha \in (0,1/2)$
is not important for the subject of our work, we leave the restriction $\alpha > 1/2$
in the same way as it was used in Theorem A.1 in \cite{J}.

\begin{theorem}
\label{lemma-small}
For every $c \gtrsim 1$ and $\alpha > \alpha_0$, there exists a unique single-lobe
solution $\phi$ of the boundary-value problem (\ref{ode}) with the global maximum at $x = 0$.
The wave profile $\phi$ and the wave speed $c$ are real-analytic functions of the wave amplitude $a$
satisfying the following Stokes expansions:
\begin{equation}
\label{wave-expansion}
\phi_{a,\alpha}(x) = a \cos(x) + a^2 \phi_2(x) + a^3 \phi_3(x) + a^4 \phi_4(x) + \mathcal{O}(a^5),
\end{equation}
and
\begin{equation}
\label{speed-expansion}
c_{a,\alpha} = 1 + c_2 a^2 + c_4 a^4 + \mathcal{O}(a^6),
\end{equation}
where the $\alpha$-dependent corrections terms $\{\phi_2,\phi_3,\phi_4\}$ and $\{ c_2, c_4 \}$
are defined in (\ref{correction1})--(\ref{correction5}) below.
\end{theorem}

\begin{proof}
We give algorithmic computations of the higher-order coefficients to the periodic wave
by using the classical Stokes expansions:
$$
\phi(x) = \sum_{k=1}^{\infty} a^k \phi_k(x), \quad c = 1 + \sum_{k=1}^{\infty} c_{2k} a^{2k}.
$$
The correction terms satisfy recursively,
\begin{eqnarray*}
\left\{ \begin{array}{l}
\mathcal{O}(a) \; : \quad (1+D_{\alpha}) \phi_1 = 0, \\
\mathcal{O}(a^2) : \quad (1+D_{\alpha}) \phi_2 + \phi_1^2 = 0, \\
\mathcal{O}(a^3) : \quad (1+D_{\alpha}) \phi_3 + c_2 \phi_1 + 2 \phi_1 \phi_2 = 0,\\
\mathcal{O}(a^4) : \quad (1+D_{\alpha}) \phi_4 + c_2 \phi_2 + 2 \phi_1 \phi_3 + \phi_2^2 = 0,\\
\mathcal{O}(a^5) : \quad (1+D_{\alpha}) \phi_5 + c_2 \phi_3 + c_4 \phi_1 + 2 \phi_1 \phi_4 + 2 \phi_2 \phi_3 = 0,
\end{array} \right.
\end{eqnarray*}
For the single-lobe wave profile $\phi$ with the global maximum at $x = 0$,
we select uniquely $\phi_1(x) = \cos(x)$ since ${\rm Ker}_{\rm even}(1+D_{\alpha}) = {\rm span}\{ \cos(\cdot) \}$
in the space of even functions in $L^2_{\rm per}(-\pi,\pi)$.
In order to select uniquely all other corrections to the Stokes expansion (\ref{wave-expansion}),
we require the corrections terms $\{ \phi_k \}_{k \geq 2}$ to be orthogonal to $\phi_1$ in $L^2_{\rm per}(-\pi,\pi)$.

Solving the inhomogeneous equation at $\mathcal{O}(a^2)$ yields the exact solution in $H^{\alpha}_{\rm per}(-\pi,\pi)$:
\begin{equation}
\label{correction1}
\phi_2(x) = -\frac{1}{2} + \frac{1}{2(2^{\alpha} - 1)} \cos(2x).
\end{equation}
The inhomogeneous equation at $\mathcal{O}(a^3)$ admits a solution $\phi_3 \in H^{\alpha}_{\rm per}(-\pi,\pi)$ if and only
if the right-hand side is orthogonal to $\phi_1$, which selects uniquely the correction $c_2$ by
\begin{equation}
\label{correction2}
c_2 = 1 - \frac{1}{2(2^{\alpha} - 1)}.
\end{equation}
After the resonant term is removed, the inhomogeneous equation at $\mathcal{O}(a^3)$ yields the exact solution in $H^{\alpha}_{\rm per}(-\pi,\pi)$:
\begin{equation}
\label{correction3}
\phi_3(x) = \frac{1}{2 (2^{\alpha}-1) (3^{\alpha}-1)} \cos(3x).
\end{equation}
By continuing the algorithm, we find the exact solution of the inhomogeneous equation at $\mathcal{O}(a^4)$ in $H^{\alpha}_{\rm per}(-\pi,\pi)$:
\begin{eqnarray}
\nonumber
\phi_4(x) & = & \frac{1}{4} - \frac{1}{4 (2^{\alpha}-1)} - \frac{1}{8 (2^{\alpha} -1)^2}
+ \frac{1}{4 (2^{\alpha} - 1)^2} \left[ \frac{2}{3^{\alpha} - 1} - \frac{1}{2^{\alpha} -1} \right] \cos(2x) \\
& \phantom{t} &
+ \frac{1}{8(2^{\alpha} - 1)(4^{\alpha}- 1)} \left[ \frac{4}{3^{\alpha}-1} + \frac{1}{2^{\alpha}-1} \right] \cos(4x).
\label{correction4}
\end{eqnarray}
Finally, the inhomogeneous equation at $\mathcal{O}(a^5)$ admits a solution $\phi_5 \in H^{\alpha}_{\rm per}(-\pi,\pi)$ if and only
if the right-hand side is orthogonal to $\phi_1$, which selects uniquely the correction $c_4$ by
\begin{equation}
\label{correction5}
c_4 = -\frac{1}{2} + \frac{1}{2(2^{\alpha} - 1)} + \frac{1}{4(2^{\alpha}-1)^2} + \frac{1}{4(2^{\alpha}-1)^3}- \frac{3}{4(2^{\alpha}-1)^2 (3^{\alpha}-1)}.
\end{equation}
Note that $c_2 > 0$ if $\alpha > \alpha_0 := \log 3/\log 2 - 1$, which implies
that the small-amplitude periodic wave with profile $\phi$ exists in the boundary-value problem (\ref{ode}) 
for $c \gtrsim 1$ and $\alpha > \alpha_0$. The periodic wave has a global maximum at $x = 0$ for small $a$ 
since $x = 0$ is the only maximum of $\phi_1(x) = \cos(x)$ and $\phi'(0) = 0$ with $\phi''(0) = -a + \mathcal{O}(a^2) < 0$. 

Justification of the existence, uniqueness, and analyticity of the Stokes expansions (\ref{wave-expansion}) and (\ref{speed-expansion})
is performed with the method of Lyapunov--Schmidt reductions for $\alpha > 1/2$, see Lemma 2.1 and Theorem A.1 in \cite{J}.
Since $\alpha_0 > 1/2$, the justification procedure applies for every $\alpha > \alpha_0$.
\end{proof}

\begin{remark}
\label{remark-c-2}
If $\alpha < \alpha_0$, then $c_2 < 0$ so that the small-amplitude periodic wave exists for $c \lesssim 1$. 
The critical value $\alpha_0$ can also be seen in the expansion of
the wave period $T$ (for fixed $c = 1$) with respect to the wave amplitude $a$ in Lemma 2.1 of \cite{J}.
\end{remark}

\begin{remark}
Variational results on existence of finite-amplitude periodic waves in the boundary-value problem (\ref{ode-new})
are obtained in Proposition 2.1 of \cite{JH1} in the energy space $H^{\alpha/2}_{\rm per}(-\pi,\pi)$
for $\alpha \in (1/3,2]$. It is shown that there exists a local minimizer of
energy for fixed momentum and mass for every $c > 0$, however, it is overlooked in \cite{JH1}
that the local minimizer may coincide with the nonzero constant solution $\psi_c(x) = c$ for all $x \in [-\pi,\pi]$
to the same boundary-value problem (\ref{ode-new}), see also Theorem \ref{lemma-positive}.
\end{remark}

For further reference, we prove the following technical result. For notational convenience, we
omit parameters $a$ and $\alpha$ when we refer to the periodic wave profile $\phi$ which solves
the boundary-value problem (\ref{ode}) for some $c \gtrsim 1$.

\begin{lemma}
\label{proposition-negative}
For every $c \gtrsim 1$, the periodic wave $\phi$ defined in Theorem \ref{lemma-small} satisfies
\begin{equation}
\int_{-\pi}^{\pi} \phi^3 dx  \left. \begin{array}{l} < 0, \quad \alpha > \alpha_0, \\
> 0, \quad \alpha < \alpha_0, \end{array} \right\}
\label{wave-negative}
\end{equation}
and
\begin{equation}
\int_{-\pi}^{\pi} \phi (\phi')^2 dx \left. \begin{array}{l} < 0, \quad \alpha > \alpha_1, \\
> 0, \quad \alpha < \alpha_1, \end{array} \right\}
\label{wave-assumption}
\end{equation}
where $\alpha_0$ and $\alpha_1$ are given by (\ref{alpha}).
\end{lemma}

\begin{proof}
By using Stokes expansions (\ref{wave-expansion}), we compute
$$
\int_{-\pi}^{\pi} \phi^3 dx = \frac{3\pi a^4}{4 (2^{\alpha} - 1)} (3 - 2^{\alpha + 1}) + \mathcal{O}(a^6)
$$
and
$$
\int_{-\pi}^{\pi} \phi (\phi')^2 dx = \frac{\pi a^4}{4 (2^{\alpha} - 1)} (5 - 2^{\alpha + 1}) + \mathcal{O}(a^6),
$$
from which (\ref{wave-negative}) and (\ref{wave-assumption}) follows thanks to the definition (\ref{alpha}).
\end{proof}

Since the Fourier basis $\{ e^{inx} \}_{n \in \mathbb{Z}}$ in $L^2_{\rm per}(-\pi,\pi)$
diagonalizes $\mathcal{L}_{c,\alpha}$, we obtain the spectrum of
$\mathcal{L}_{c,\alpha}$ in $L^2_{\rm per}(-\pi,\pi)$ for every $c \in \mathbb{R}$ and $\alpha > 0$:
\begin{equation}
\label{spectrum-L}
\sigma(\mathcal{L}_{c,\alpha}) = \{ -c + |n|^{\alpha}, \;\; n \in \mathbb{Z} \}.
\end{equation}
The following lemma clarifies the number and multiplicity of negative and zero eigenvalues
of the Jacobian operator $\mathcal{H}_{c,\alpha}$ in $L^2_{\rm per}(-\pi,\pi)$,
where the expression for $\mathcal{H}_{c,\alpha}$ is given by (\ref{Jacobian}).

\begin{lemma}
\label{lemma-2}
For every $c \gtrsim 1$ and $\alpha > \alpha_0$, $\sigma(\mathcal{H}_{c,\alpha})$
in $L^2_{\rm per}(-\pi,\pi)$ consists of one simple negative eigenvalue, a simple zero eigenvalue,
and a countable sequence of positive eigenvalues bounded away from zero.
\end{lemma}

\begin{proof}
Note that $\sigma(\mathcal{H}_{c,\alpha})$  in $L^2_{\rm per}(-\pi,\pi)$ is purely discrete for every $c > 1$,
thanks to the compactness of $[-\pi,\pi]$ and boundedness of $\phi \in L^{\infty}_{\rm per}(-\pi,\pi)$.
For $c = 1$, $\mathcal{H}_{c=1,\alpha}$ coincides with $\mathcal{L}_{c=1,\alpha}$, hence
it follows from (\ref{spectrum-L}) that $\sigma(\mathcal{H}_{c=1,\alpha})$
has a simple negative eigenvalue, a double zero eigenvalue, and a countable sequence of positive eigenvalues bounded away from zero.

Since $\mathcal{H}_{c,\alpha} - \mathcal{L}_{c,\alpha} = - 2 \phi$ is a bounded perturbation
and $(\phi,c)$ depend analytically on $a$,
the analytic perturbation theory (Theorem VII.1.7 in \cite{Kato}) guarantees continuity of eigenvalues
for $c \gtrsim 1$ close to their limiting values as $c \to 1$. Therefore, the proof is achieved if we can show
that the double zero eigenvalue of $\mathcal{H}_{c,\alpha}$ in $L^2_{\rm per}(-\pi,\pi)$ splits
as $c \gtrsim 1$ into a simple zero eigenvalue and a simple positive eigenvalue.

Since ${\rm Ker}(\mathcal{H}_{c=1,\alpha}) = {\rm span}\{\cos(\cdot),\sin(\cdot)\}$ and
$\mathcal{H}_{c,\alpha} \phi' = 0$ for every $c > 1$ with odd $\phi$, the zero eigenvalue
associated with the subspace ${\rm Ker}_{\rm odd}(\mathcal{H}_{c=1,\alpha}) = {\rm span}\{\sin(\cdot)\}$ persists for $c > 1$.
It remains to check the shift of the zero eigenvalue associated with the subspace
${\rm Ker}_{\rm even}(\mathcal{H}_{c=1,\alpha}) = {\rm span}\{\cos(\cdot)\}$. Hence, we expand $\mathcal{H}_{c,\alpha}$
in powers of $a$ by using (\ref{wave-expansion}):
\begin{equation}
\label{H-expansions}
\mathcal{H}_{c,\alpha} = -1 - D_{\alpha} - 2a \cos(x)-  \frac{a^2}{2^\alpha-1} \left[\cos(2x) - \frac{1}{2} \right] + \mathcal{O}(a^3)
\end{equation}
and look for solutions $(\lambda,v) \in \mathbb{R} \times H^{\alpha}_{\rm per}(-\pi,\pi)$
of the eigenvalue problem $\mathcal{H}_{c,\alpha} v = \lambda v$ near $(\lambda,v) = (0,\cos(\cdot))$ by using the expansions
\begin{eqnarray*}
\left\{ \begin{array}{l} v(x) = \cos(x) + a v_1(x) + a^2 v_2(x) + \mathcal{O}(a^3), \\
\lambda = a \lambda_1 + a^2 \lambda_2 + \mathcal{O}(a^3).\end{array} \right.
\end{eqnarray*}
The correction terms in $H^{\alpha}_{\rm per}(-\pi,\pi)$ satisfy recursively,
\begin{eqnarray*}
\left\{ \begin{array}{l}
\mathcal{O}(a) : \quad (1+D_{\alpha}) v_1 + 1 + \cos(2x) + \lambda_1 \cos(x) = 0, \\
\mathcal{O}(a^2) : \quad (1+D_{\alpha}) v_2 + 2 \cos(x) v_1 + \frac{1}{2^{\alpha}-1} \left[ \cos(2x) - \frac{1}{2} \right] \cos(x) +
\lambda_2 \cos(x) = 0.
\end{array} \right.
\end{eqnarray*}
In order to determine them uniquely, we impose orthogonality conditions of $\{ v_k \}_{k \geq 1}$ to
$\cos(\cdot)$ in $L^2_{\rm per}(-\pi,\pi)$. The linear inhomogeneous equation at $\mathcal{O}(a)$ admits a solution
$v_1 \in H^{\alpha}_{\rm per}(-\pi,\pi)$ if and only if $\lambda_1 = 0$, after which
the solution is found explicitly:
$$
v_1(x) = \frac{1}{2^\alpha-1}\cos(2x) - 1.
$$
The linear inhomogeneous equation at $\mathcal{O}(a^2)$ admits a solution
$v_2 \in H^{\alpha}_{\rm per}(-\pi,\pi)$ if and only if $\lambda_2 = 2 c_2$, where $c_2$ is defined by (\ref{correction2}).
Since $c_2 > 0$ if $\alpha > \alpha_0$, the small positive eigenvalue $\lambda = 2 c_2 a^2 + \mathcal{O}(a^3)$ bifurcates
from the zero eigenvalue as $c \gtrsim 1$. Functional-analytic setup for justification
of perturbative expansions can be found in \cite{J} (see also \cite{JH2}) for $\alpha > 1/2$,
which is met since $\alpha_0 > 1/2$.
\end{proof}

\begin{remark}
By using variational methods, it was shown in Proposition 3.1 and Lemma 3.3 of \cite{JH1} that
${\rm ker}(\mathcal{H}_{c,\alpha}) = {\rm span}\{\phi'\}$ is one-dimensional,
the zero eigenvalue is the lowest eigenvalue in the subspace of odd functions in $L^2_{\rm per}(-\pi,\pi)$,
and $\sigma(\mathcal{H}_{c,\alpha})$ has either one or two negative eigenvalues
for every $c > 1$ and $\alpha \in (1/3,2]$. By Lemma \ref{lemma-2},
$\sigma(\mathcal{H}_{c,\alpha})$ has only one simple negative eigenvalue for $\alpha > \alpha_0$
if $c \gtrsim 1$.
\end{remark}

The following lemma gives the {\em isospectrality} result for the linearized operator $\mathcal{H}_{c,\alpha}$
for all $c > 1$.

\begin{lemma}
\label{lemma-11}
For every $c > 1$ and $\alpha \in (\alpha_0,2]$, $\sigma(\mathcal{H}_{c,\alpha})$
in $L^2_{\rm per}(-\pi,\pi)$ consists of one simple negative eigenvalue,
a simple zero eigenvalue, and a countable sequence of positive eigenvalues
bounded away from zero.
\end{lemma}

\begin{proof}
By Proposition 2.1 of \cite{JH1}, the single-lobe solution $\psi$ of the boundary-value problem
(\ref{ode-new}) exists for every $c > 1$ and $\alpha \in (1/3,2]$ and the solution
is a $C^1$ function of $c$ for $c > 1$. This result is extended to the single-lobe solution $\phi$ of the boundary-value
problem (\ref{ode}) thanks to the transformation (\ref{phi-psi}), where it is uniquely identified
with the small-amplitude periodic wave in Theorem \ref{lemma-small}.

By Proposition 3.1 of \cite{JH1}, the kernel of $\mathcal{H}_{c,\alpha}$
at the single-lobe solution $\phi \in H^{\alpha}_{\rm per}(-\pi,\pi)$
is simple with ${\rm ker}(\mathcal{H}_{c,\alpha}) = {\rm span}\{\phi'\}$
for every $c > 1$ and $\alpha \in (1/3,2]$.
The number of negative eigenvalues of $\mathcal{H}_{c,\alpha}$ may change
in the parameter continuations in $c$ if and only if the eigenvalues pass through zero.
By Lemma \ref{lemma-2}, $\sigma(\mathcal{H}_{c,\alpha})$ at the single-lobe solution $\phi$
in $L^2_{\rm per}(-\pi,\pi)$ consists of one simple negative eigenvalue, a simple zero eigenvalue,
and a countable sequence of positive eigenvalues bounded away from zero for $c \gtrsim 1$
if $\alpha > \alpha_0$. By the continuity argument and Proposition 3.1 of \cite{JH1},
the same remains true for $\sigma(\mathcal{H}_{c,\alpha})$ for every $c > 1$ 
and $\alpha \in (\alpha_0,2]$.
\end{proof}

\begin{remark}
For the KdV case with $\alpha = 2$, a different homotopy argument for the proof of isospectrality
of $\sigma(\mathcal{H}_{c,\alpha})$ can be developed, see, e.g., \cite{J0}, based on the classical results on
the non-degeneracy of the energy-to-period function in \cite{Schaaf} and \cite{GV}.
For the BO case with $\alpha = 1$, explicit computations based on complex analysis techniques
were developed much earlier in \cite{AT}.
\end{remark}

\subsection{Periodic waves in the KdV equation}
\label{sec-22}

For the KdV equation (see, e.g., Proposition 4.1 in \cite{HLP}),
the solution $\phi$ to the boundary-value problem (\ref{ode}) with $\alpha = 2$ is given by
\begin{equation}
\label{wave-explicit}
\phi(x) = \frac{2 K(k)^2}{\pi^2} \left[  1 - 2k^2 - \sqrt{1 - k^2 + k^4} + 3 k^2 {\rm cn}^2\left(\frac{K(k)}{\pi} x; k\right) \right]
\end{equation}
where ${\rm cn}$ is the Jacobi elliptic function, $K(k)$ is a complete elliptic integral of the first kind,
and $k \in (0,1)$ is the elliptic modulus that parameterizes the wave speed $c$ given by
\begin{equation}
\label{speed-explicit}
c = \frac{4 K(k)^2}{\pi^2} \sqrt{1 - k^2 + k^4}.
\end{equation}
The small-amplitude expansions (\ref{wave-expansion})--(\ref{speed-expansion}) is recovered
from (\ref{wave-explicit})--(\ref{speed-explicit}) with the wave amplitude
$a := 3k^2/4 + \mathcal{O}(k^4)$ as $k \to 0$.

We prove that the map $(0,1) \ni k \mapsto c \in (1,\infty)$ is strictly increasing,
hence the explicit solution (\ref{wave-explicit})--(\ref{speed-explicit}) exists for every $c > 1$ (see also \cite{NataliPava}).
We also extend the inequalities (\ref{wave-negative}) and (\ref{wave-assumption}) with $\alpha = 2$
for every $c > 1$.

\begin{lemma}
\label{lemma-kdv}
The map $(0,1) \ni k \mapsto c \in (1,\infty)$ for the solution (\ref{wave-explicit})--(\ref{speed-explicit})
is strictly increasing. In addition, for every $c > 1$, we have
\begin{equation}
\label{wave-negative-kdv}
\int_{-\pi}^{\pi} \phi^3 dx < 0, \quad \int_{-\pi}^{\pi} \phi (\phi')^2 dx < 0.
\end{equation}
\end{lemma}

\begin{proof}
We have $\phi = 0$ and $c = 1$ at $k = 0$. Thanks to the smoothness of $\phi$ and $c$ in $k$,
it holds from (\ref{speed-explicit}) by explicit differentiation:
$$
\frac{\pi^2 \sqrt{1-k^2+k^4}}{4 K(k)} \frac{dc}{dk} = 2(1-k^2+k^4) \frac{d K(k)}{dk} - k (1 - 2k^2) K(k).
$$
By using the differential relation,
$$
\frac{d K(k)}{dk} = \frac{E(k) - (1-k^2) K(k)}{k (1-k^2)},
$$
the previous expression can be reduced to the form
$$
\frac{\pi^2 k (1-k^2) \sqrt{1-k^2+k^4}}{4 K(k)} \frac{dc}{dk} = 2(1-k^2+k^4) E(k) - (2-3k^2 + k^4) K(k) =: I(k),
$$
where $E(k)$ is a complete elliptic integral of the second kind and $I(k)$ is introduced for convenience.
Note that $I(0) = 0$. We claim that the map $(0,1) \ni k \mapsto I$ is strictly increasing.
Indeed, by using the differential relation
$$
\frac{d E(k)}{dk} = \frac{E(k) - K(k)}{k},
$$
we obtain after straightforward computations
$$
\frac{dI(k)}{dk} = 5 k \left[ (1-k^2) K(k) - (1-2k^2) E(k) \right] > 0,
$$
where the last inequality follows from the fact that $K(k) > E(k)$ for every $k \in (0,1)$.
Since $I(0) = 0$, we have $I(k) > 0$ for every $k \in (0,1)$, which implies that
$\frac{dc}{dk} > 0$ for every $k \in (0,1)$.

Let us now prove the inequalities (\ref{wave-negative-kdv}) for every $c > 1$.
Since $\phi$ and $c$ are smooth in $k$, we differentiate the nonlinear equation
in the boundary-value problem (\ref{ode}) with $\alpha = 2$ in $k$ and obtain
\begin{equation*}
\left[ c + D_{\alpha = 2} + 2 \phi \right] \frac{\partial \phi}{\partial k} + \frac{dc}{dk} \phi = 0.
\end{equation*}
Multiplying this equation by $\phi$ and integrating on $[-\pi,\pi]$
imply that
\begin{equation*}
\int_{-\pi}^{\pi} \phi^2 \frac{\partial \phi}{\partial k} dx = - \frac{dc}{dk} \int_{-\pi}^{\pi} \phi^2 dx,
\end{equation*}
where we have used the facts that $D_{\alpha = 2}$ is self-adjoint in $L^2_{\rm per}(-\pi,\pi)$ and
$\phi, \partial_a \phi \in H^{\alpha = 2}_{\rm per}(-\pi,\pi)$.
Since $\frac{dc}{dk} > 0$ for every $k \in (0,1)$, the map $k \mapsto \int_{-\pi}^{\pi} \phi^3 dx$
is strictly decreasing with $\int_{-\pi}^{\pi} \phi^3 dx = 0$ at $k = 0$.
Therefore, $\int_{-\pi}^{\pi} \phi^3 dx < 0$ for $k \in (0,1)$ by the continuity argument in $k$.

Finally, the inequality $\int_{-\pi}^{\pi} \phi (\phi')^2 dx < 0$ for every $c > 1$
follows from the boundary-value problem (\ref{ode}) with $\alpha = 2$:
$$
\int_{-\pi}^{\pi} \phi (\phi')^2 dx = -\frac{1}{c} \left[ \int_{-\pi}^{\pi} (\phi')^2 \phi'' dx +
\int_{-\pi}^{\pi} \phi^2 (\phi')^2 dx \right],
$$
where the first term in the right-hand side is zero thanks to the smoothness of $\phi$.
\end{proof}

\subsection{Periodic waves in the BO equation}
\label{sec-23}

For the BO equation (see, e.g., \cite{Matsuno}),
the solution $\phi$ to the boundary-value problem (\ref{ode}) with $\alpha = 1$ is given by
\begin{equation}
\label{Sol-Ben-Ono}
\phi(x) = \frac{\cosh \gamma \cos x - 1}{\sinh \gamma (\cosh \gamma - \cos x)}, \quad c = \coth \gamma.
\end{equation}
The small-amplitude expansions (\ref{wave-expansion})--(\ref{speed-expansion}) is recovered
from (\ref{Sol-Ben-Ono}) with the wave amplitude $a := 2 e^{-\gamma} + \mathcal{O}(e^{-3\gamma})$
as $\gamma \to \infty$. It follows from the simple expression $c = \coth \gamma$ that
the map $(0,\infty) \ni \gamma \mapsto c \in (1,\infty)$ is strictly decreasing,
hence the explicit solution (\ref{Sol-Ben-Ono}) exists for every $c > 1$.

Let us now show that the inequalities (\ref{wave-negative}) and (\ref{wave-assumption}) with $\alpha = 1$
holds for every $c > 1$, that is,
\begin{equation}
\label{wave-negative-bo}
\int_{-\pi}^{\pi} \phi^3 dx < 0, \quad \int_{-\pi}^{\pi} \phi (\phi')^2 dx > 0.
\end{equation}
Indeed, by using the explicit formula (\ref{Sol-Ben-Ono}) and symbolic computations with Wolfram's MATHEMATICA,
we obtain
$$
\int_{-\pi}^{\pi} \phi^3 dx = -\pi (c-1)^2 (2c + 1)
$$
and
$$
\int_{-\pi}^{\pi} \phi (\phi')^2 dx = \frac{1}{4} \pi (c^2-1)^2,
$$
from which the inequalities (\ref{wave-negative-bo}) hold for every $c > 1$.

\subsection{Positivity of the periodic waves}
\label{sec-24}

The following result states that the single-lobe wave profile $\psi$ in the boundary-value problem (\ref{ode-new})
for every $c > 1$ and $\alpha \in (\alpha_0,2]$ is positive and satisfies $\psi(x) \geq \psi(\pm \pi) > 0$ for every $x \in [-\pi,\pi]$.
The result has not appeared in the literature, e.g. a remark in the proof of Proposition 2.1 in \cite{JH1} states that
a periodic solution need not be positive everywhere. On the other hand, positivity of the Fourier
coefficients in the Fourier series for the periodic wave $\psi$ is proven in Theorem 3.5 of \cite{C}
for every $\alpha > 1/2$ and for sufficiently large periods (which is equivalent to $c > 1$ at 
the $2\pi$-period).

Our proof has similarity to the work of \cite{Torres} on the second-order differential equations. 
However, the existence of constant solutions is eliminated in \cite{Torres} by the space-dependent 
coefficients in the boundary-value problem. For the problem (\ref{ode-new}), we have to use 
the Leray--Schauder index to single out single-lobe periodic solutions from the constant solutions. 

\begin{theorem}
\label{lemma-positive}
For every $c > 1$ and $\alpha \in (\alpha_0,2]$, there exists a unique single-lobe
solution $\psi$ of the boundary-value problem (\ref{ode-new}) such that
$\psi(x) > 0$ for every $x \in [-\pi,\pi]$.
\end{theorem}

\begin{proof}
For $c \gtrsim 1$, the assertion of the lemma follows from Theorem \ref{lemma-small} thanks
to the transformation (\ref{phi-psi}) and smallness of $a$ in the Stokes expansion (\ref{wave-expansion}).
In order to prove the same for every $c > 1$, we introduce the Green function $G_{c,\alpha} \in L^2_{\rm per}(-\pi,\pi)$
for the positive operator $(c - D_{\alpha})$ from solution $\varphi(x) = \int_{-\pi}^{\pi} G_{c,\alpha}(x-s) h(s) ds$
of the linear inhomogeneous equation
\begin{equation}
\label{G-definition}
(c - D_{\alpha}) \varphi = h, \quad h \in L^2_{\rm per}(-\pi,\pi).
\end{equation}
By Fourier series, the solution for $G$ is available in the Fourier series form:
\begin{equation}
\label{G-fourier}
G_{c,\alpha}(x) = \frac{1}{2\pi} \sum_{n \in \mathbb{Z}} \frac{e^{inx}}{c + |n|^{\alpha}},
\end{equation}
from which it follows that $G_{c,\alpha} \in L^2_{\rm per}(-\pi,\pi)$ if $\alpha > 1/2$ but
$G_{c,\alpha}(0) = \infty$ if $\alpha \leq 1$. It is proven in \cite{Nieto} for $\alpha \in (0,1)$
(and the proof is extended for $\alpha \in [1,2]$, see \cite{Bai})
that there is a positive $(c,\alpha)$-dependent constant $m_{c,\alpha}$ such that
\begin{equation}
\label{G-positive}
G_{c,\alpha}(x) \geq m_{c,\alpha}, \quad x \in [-\pi,\pi].
\end{equation}
In addition, for $\alpha > 1/2$, $M_{c,\alpha} := \| G_{c,\alpha} \|_{L^2_{\rm per}}$
for a positive $(c,\alpha)$-dependent constant $M_{c,\alpha}$.

Let us consider a positive cone in the space of $L^2_{\rm per}(-\pi,\pi)$-functions defined by
\begin{equation}
P_{c,\alpha} := \left\{ \psi \in L^2_{\rm per}(-\pi,\pi) : \quad
\psi(x) \geq \frac{m_{c,\alpha}}{M_{c,\alpha}} \| \psi \|_{L^2_{\rm per}}, \;\; x \in [-\pi,\pi] \right\}.
\label{cone}
\end{equation}
Define the following nonlinear operator $A_{c,\alpha}(\psi) : L^2_{\rm per}(-\pi,\pi) \mapsto L^2_{\rm per}(-\pi,\pi)$
for any $c > 0$:
\begin{equation}
A_{c,\alpha}(\psi) := (c-D_{\alpha})^{-1} \psi^2 \quad \Rightarrow \quad
A_{c,\alpha}(\psi)(x) = \int_{-\pi}^{\pi} G_{c,\alpha}(x-s) \psi(s)^2 ds.
\label{fixed-point}
\end{equation}

The operator $A_{c,\alpha}$ is bounded and continuous in $L^2_{\rm per}(-\pi,\pi)$ thanks
to the generalized Young inequality:
\begin{equation}
\label{A-estimate1}
\| A_{c,\alpha}(\psi) \|_{L^2_{\rm per}} \leq \| G_{c,\alpha} \|_{L^2_{\rm per}} \| \psi^2 \|_{L^1_{\rm per}}
\leq M_{c,\alpha} \| \psi \|_{L^2_{\rm per}}^2.
\end{equation}
Moreover, $A_{c,\alpha}$ is compact because it is the limit of compact operators $A^{(N)}_{c,\alpha}$
given by the first $2N+1$ Fourier coefficients. Indeed, we have
\begin{eqnarray*}
\| A_{c,\alpha}(\psi) - A_{c,\alpha}^{(N)}(\psi) \|^2_{L^2_{\rm per}}
& = & \frac{1}{2\pi} \sum_{|n| > N} \frac{|(\psi^2)_n|^2}{(c+|n|^{\alpha})^2} \leq
\frac{1}{2\pi} \| (\psi^2)_n \|_{\ell^{\infty}}^2 \sum_{|n| > N} \frac{1}{(c+|n|^{\alpha})^2} \\
& \leq & \frac{1}{2\pi} \|\psi^2\|_{L^1_{\rm per}}^2 \sum_{|n| > N} \frac{1}{(c+|n|^{\alpha})^2}
= \frac{1}{2\pi} \|\psi\|_{L^2_{\rm per}}^4 \sum_{|n| > N} \frac{1}{(c+|n|^{\alpha})^2},
\end{eqnarray*}
where the numerical series converges for every $\alpha > 1/2$. Therefore,
for every $\psi \in L^2_{\rm per}(-\pi,\pi)$,
$$
\lim_{N \to \infty} \| A_{c,\alpha}(\psi) - A_{c,\alpha}^{(N)}(\psi) \|_{L^2_{\rm per}} = 0,
$$
so that $A_{c,\alpha}$ maps bounded sets in $L^2_{\rm per}(-\pi,\pi)$ to pre-compact sets in $L^2_{\rm per}(-\pi,\pi)$.

By using positivity of the Green function in (\ref{G-positive}), we confirm that
the operator $A_{c,\alpha}(\psi)$ is closed in $P_{c,\alpha} \subset L^2_{\rm per}(-\pi,\pi)$:
\begin{equation}
\label{A-estimate2}
A_{c,\alpha}(\psi)(x) \geq m_{c,\alpha} \| \psi \|_{L^2_{\rm per}}^2 \geq \frac{m_{c,\alpha}}{M_{c,\alpha}} \| A_{c,\alpha}(\psi) \|_{L^2_{\rm per}}.
\end{equation}
A fixed point $\psi$ of $A_{c,\alpha}(\psi)$ in $P_{c,\alpha} \subset L^2_{\rm per}(-\pi,\pi)$
corresponds to the positive function $\psi$ such that $\psi(x) > 0$ for every $x \in [-\pi,\pi]$.

Let $B_r := \{ \psi \in L^2_{\rm per}(-\pi,\pi) :\;\; \| \psi \|_{L^2_{\rm per}} < r\}$ be a ball
of radius $r$ in $L^2_{\rm per}(-\pi,\pi)$. The existence of a fixed point of $A_{c,\alpha}(\psi)$
in $P_{c,\alpha} \cap (\bar{B}_{r_+} \backslash B_{r_-})$ for some $0 < r_- < r_+ < \infty$
follows from Krasnoselskii's fixed-point theorem (see, e.g., Corollary 20.1 in \cite{Var-Book})
if there exist $r_-$ and $r_+$ such that
\begin{equation}
\label{cone-condition-1}
\| A_{c,\alpha}(\psi) \|_{L^2_{\rm per}} < \| \psi \|_{L^2_{\rm per}}, \quad \psi \in P_{c,\alpha} \cap \partial B_{r_-}
\end{equation}
and
\begin{equation}
\label{cone-condition-2}
\| A_{c,\alpha}(\psi) \|_{L^2_{\rm per}} > \| \psi \|_{L^2_{\rm per}}, \quad \psi \in P_{c,\alpha} \cap \partial B_{r_+}.
\end{equation}
Bound (\ref{cone-condition-1}) follows from (\ref{A-estimate1}) with $M_{c,\alpha} r_- < 1$.
Bound (\ref{cone-condition-2}) follows from (\ref{A-estimate2}) with $\sqrt{2\pi} m_{c,\alpha} r_+ > 1$,
hence the two radii satisfy the constraints
\begin{equation}
\label{r-minus-r-plus}
0 < r_- < \frac{1}{M_{c,\alpha}} \leq \frac{1}{\sqrt{2\pi} m_{c,\alpha}} < r_+ < \infty,
\end{equation}
where $\sqrt{2\pi} m_{c,\alpha} \leq M_{c,\alpha}$ follows from (\ref{G-positive}). 
Hence, there exists a fixed point of $A_{c,\alpha}(\psi)$ in $P_{c,\alpha} \cap (\bar{B}_{r_+} \backslash B_{r_-})$.

We use bootstrapping arguments similar to those used in the proof of Proposition 2.1 in \cite{JH1}
and show that the fixed point of $A_{c,\alpha}$ in $L^2_{\rm per}(-\pi,\pi)$ also exists in $H^{\alpha}_{\rm per}(-\pi,\pi)$,
hence $\psi$ is a positive solution of the boundary-value problem (\ref{ode-new}). Indeed,
if $\psi \in L^4_{\rm per}(-\pi,\pi)$, then $\psi \in H^{\alpha}_{\rm per}(-\pi,\pi)$ thanks to the estimate:
$$
\| D_{\alpha} \psi \|_{L^2_{\rm per}} = \| D_{\alpha} (c - D_{\alpha})^{-1} \psi^2 \|_{L^2_{\rm per}}
\leq \| \psi^2 \|_{L^2_{\rm per}} = \| \psi \|_{L^4_{\rm per}}^2.
$$
In order to show that $\psi \in L^4_{\rm per}(-\pi,\pi)$, we use the generalized Young and
H\"{o}lder inequalities:
\begin{eqnarray}
\label{ineq-1}
\| \psi \|_{L^r_{\rm per}} & \leq & \| G \|_{L^p_{\rm per}} \| \psi^2 \|_{L^q_{\rm per}}, \quad
\quad \quad \quad 1 + \frac{1}{r} = \frac{1}{p} + \frac{1}{q}, \quad p,q,r \geq 1,\\
\label{ineq-2}
& \leq & \| G \|_{L^p_{\rm per}} \| \psi \|_{L^{sq}_{\rm per}} \| \psi \|_{L^{sq/(s-1)}_{\rm per}}, \quad s \geq 1.
\end{eqnarray}
By using the Hausdorff--Young inequality
$$
\| G \|_{L^p_{\rm per}} \leq C_p \| (c + |n|^{\alpha})^{-1} \|_{\ell^{p/(p-1)}}, \quad p \geq 2,
$$
we can see that $\| G \|_{L^p_{\rm per}} < \infty$ if $\alpha p/(p-1) > 1$.
If $\alpha \geq 1$, then $G \in L^p_{\rm per}(-\pi,\pi)$ for every $p \in [2,\infty)$.
Applying (\ref{ineq-1}) with $r = p$ and $q = 1$,
we have $\psi \in L^p_{\rm per}(-\pi,\pi)$ for every $p \in [2,\infty)$.

If $\alpha \in (\alpha_0,1)$, we set $p_0 = 1/(1-\alpha_0) > 2$
and obtain with the same argument that $G, \psi \in L^{p_0}_{\rm per}(-\pi,\pi)$.
Then, using bound (\ref{ineq-2}) with $sq = 2$ and $s q/(s-1) = p_0$,
that is, with $s = 1 + 2/p_0$ and $q = 2p_0/(2+p_0)$,
we obtain $\psi \in L^r_{\rm per}(-\pi,\pi)$ with $r = 2p_0/(4-p_0) > p_0$ (because $p_0 > 2$).
Iterating bound (\ref{ineq-2}) with $sq = 2$ and $s q/(s-1) = r$,
we obtain a bigger value for $r = p_0/(3-p_0) > 2p_0/(4-p_0)$, hence by further iterations,
we get $\psi \in L^p_{\rm per}(-\pi,\pi)$ for every $p \in [2,\infty)$ including $p = 4$.

The fixed point $\psi \in P_{c,\alpha} \cap (\bar{B}_{r_+} \backslash B_{r_-})$ for $r_- < r_+$ satisfying
(\ref{r-minus-r-plus}) exists for every $c > 0$. However, the constant periodic
solution 
\begin{equation}
\label{wave-constant}
\psi_c(x) = c, \quad x \in [-\pi,\pi]
\end{equation}
is a fixed point of $A_{c,\alpha}$ in $P_{c,\alpha} \cap (\bar{B}_{r_+} \backslash B_{r_-})$ for every $c > 0$ and $\alpha > 0$.
Indeed, $A_{c,\alpha}(\psi_c) = \psi_c$ for every
$\alpha > 0$ and $\psi_c \in P_{c,\alpha} \cap (\bar{B}_{r_+} \backslash B_{r_-})$ for every $c > 0$
thanks to the condition $\sqrt{2\pi} m_{c,\alpha} \leq M_{c,\alpha}$. In order to be able to claim
that there exists a non-trivial fixed point $\psi \in P_{c,\alpha} \cap (\bar{B}_{r_+} \backslash B_{r_-})$
for $c > 1$ in addition to the constant fixed point $\psi_c$, we look at the Leray--Schauder index of
the fixed point in the subspace of even functions in $L^2_{\rm per}(-\pi,\pi)$, defined
as $(-1)^N$, where $N$ is the number of unstable eigenvalues of $A_{c,\alpha}'(\psi)$ outside the unit disk
with the account of their multiplicities.

For the fixed point $\psi_c$ in (\ref{wave-constant}), we have $A_{c,\alpha}'(\psi_c) = 2c (c - D_{\alpha})^{-1}$,
hence there exists $N = K + 1$ unstable eigenvalues of $A_{c,\alpha}'(\psi_c)$
outside the unit disk for every $c \in (K^{\alpha},(K+1)^{\alpha})$, where $K \in \mathbb{N}$.
Therefore, the index of $\psi_c$ changes sign every time $c$ crosses 
values in the set $\{ K^{\alpha} \}_{K \in \mathbb{N}}$,
as is shown on Figure \ref{fig-bif}.
On the other hand, for $K = 1$, $c = 1$ is a bifurcation value by Theorem \ref{lemma-small}
and two non-trivial fixed points $\psi \in P_{c,\alpha} \cap (\bar{B}_{r_+} \backslash B_{r_-})$
bifurcate for $c \gtrsim 1$ if $\alpha > \alpha_0$, one is single-lobe with maximum at $x = 0$ and
the other one is single-lobe with minimum at $x = 0$, both are strictly positive.
For the non-trivial fixed points $\psi$, we have 
$$
A_{c,\alpha}'(\psi) = 2 (c - D_{\alpha})^{-1} \psi =
{\rm Id} - (c-D_{\alpha})^{-1} \tilde{\mathcal{H}}_{c,\alpha}, 
$$
where it follows from
positivity of $\psi$ that $A_{c,\alpha}'(\psi) \geq 0$. By Lemma \ref{lemma-11} for $c > 1$ and $\alpha \in (\alpha_0,2]$,
$\tilde{\mathcal{H}}_{c,\alpha} = \mathcal{H}_{c,\alpha}$ has only one simple negative eigenvalue,
hence there exists $N = 1$ unstable eigenvalues of $A_{c,\alpha}'(\psi)$.
Therefore, the pair of non-trivial fixed points $\psi \in P_{c,\alpha} \cap (\bar{B}_{r_+} \backslash B_{r_-})$
is distinct from the constant fixed point $\psi_c$ for every $c > 1$,
as is shown on Figure \ref{fig-bif}.

\begin{figure}[h]
\includegraphics[width=0.8\linewidth]{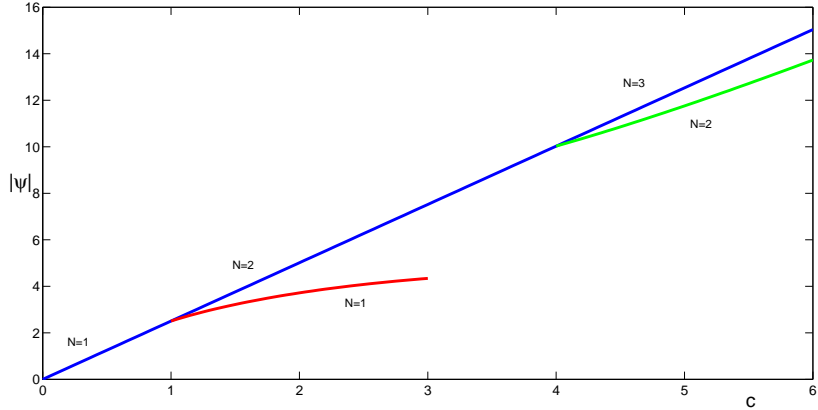}
\caption{Schematic representation of the constant fixed point $\psi_c$
and pairs of non-trivial fixed points on the $(c,\| \psi \|_{L^2_{\rm per}})$ plane
for $\alpha = 2$.}
\label{fig-bif}
\end{figure}

The pair of non-trivial fixed points for the single-lobe solution 
remains inside $P_{c,\alpha} \cap (\bar{B}_{r_+} \backslash B_{r_-})$
in continuation of the solution family in $c$ for a fixed $\alpha \in (\alpha_0,2]$,
thanks to the conditions (\ref{cone-condition-1}),(\ref{cone-condition-2}), and (\ref{r-minus-r-plus}).
Their indices also remain invariant with respect to $c$ thanks to Lemma \ref{lemma-11}.
Therefore, these fixed points cannot coalesce with any other fixed points
of $A_{c,\alpha}$ in $P_{c,\alpha} \cap (\bar{B}_{r_+} \backslash B_{r_-})$. 
By continuity, these fixed points coincide with the single-lobe solutions, 
existence of which is proven in Proposition 2.1 in \cite{JH1}. 
\end{proof}

\begin{remark}
At every bifurcation point $c = K^{\alpha}$ with $K \geq 2$,
a pair of additional fixed points of $A_{c,\alpha}$ bifurcates in $P_{c,\alpha} \cap (\bar{B}_{r_+} \backslash B_{r_-})$,
as is shown on Figure \ref{fig-bif} for $K = 2$ and $\alpha = 2$. These fixed points are not single-lobe solutions for $K \geq 2$
but instead these are concatenations of the single-lobe solutions with $K$ periods on $[-\pi,\pi]$.
\end{remark}

\begin{remark}
\label{remark-isospectrality}
Theorem 4.1 in \cite{NataliPava} states that $\tilde{\mathcal{H}}_{c,\alpha} = \mathcal{H}_{c,\alpha}$ in (\ref{Jacobian-new})
has only one simple negative eigenvalue and a simple zero eigenvalue if $\psi$ and its Fourier transform are strictly positive.
These properties have been verified in \cite{NataliPava} for the integrable cases $\alpha = 2$ and $\alpha = 1$,
for which the exact solutions (\ref{wave-KdV}) and (\ref{wave-BO}) are available.
With Theorem 3.5 in \cite{C} and Theorem \ref{lemma-positive} above, Theorem 4.1 in \cite{NataliPava}
can be applied to the periodic waves for every $c > 1$ and $\alpha \in (\alpha_0,2]$.
This argument gives an alternative proof of Lemma \ref{lemma-11}.
\end{remark}

Let us illustrate positivity of $\psi$ for the classical cases $\alpha = 2$ and $\alpha = 1$.
For the KdV equation with the solution (\ref{wave-explicit}) and (\ref{speed-explicit}), we
use $\psi(x) = c + \phi(x)$ and obtain
\begin{equation}
\label{wave-KdV}
\psi(x) = \frac{2 K(k)^2}{\pi^2} \left[  1 - 2k^2 + \sqrt{1 - k^2 + k^4} + 3 k^2 {\rm cn}^2\left(\frac{K(k)}{\pi} x; k\right) \right],
\end{equation}
from which $\psi(x) \geq \psi(\pm \pi) > 0$ holds for every $x \in [-\pi,\pi]$ and every $k \in (0,1)$.
Indeed, if $\alpha = 2$, the boundary-value problem (\ref{ode-new}) can be formulated as a planar Hamiltonian
system on the phase plane $(\psi,\psi')$ and a set of closed orbits for periodic solutions is located
on the phase plane between the saddle point $(0,0)$ and the center point $(c,0)$, hence, $\psi(x) > 0$ for every $x \in [-\pi,\pi]$.

For the BO equation with the solution (\ref{Sol-Ben-Ono}), we use $\psi(x) = c + \phi(x)$ and obtain
\begin{equation}
\label{wave-BO}
\psi(x) = \frac{\sinh \gamma}{\cosh \gamma - \cos x},
\end{equation}
from which $\psi(x) \geq \psi(\pm \pi) = \tanh \gamma > 0$ holds for every $x \in [-\pi,\pi]$ and every $\gamma \in (0,\infty)$.

\section{Proof of Theorem \ref{theorem-main-1}}
\label{sec-3}

In what follows, we always use $\phi$ to denote the single-lobe periodic wave, which is even with a maximum at
$x = 0$ and minimum at $x = \pm \pi$. We always assume that
\begin{equation}
\label{assumption-on-phi}
\int_{-\pi}^{\pi} \phi^3 dx \neq 0 \quad \mbox{\rm and} \quad
\int_{-\pi}^{\pi} \phi (\phi')^2 dx \neq 0.
\end{equation}
Recall that although $\phi \in H^{\alpha}_{\rm per}(-\pi,\pi)$,
it is extended to $\phi \in H^{\infty}_{\rm per}(-\pi,\pi)$ by
bootstrapping arguments similar to those used in the proof of Theorem \ref{lemma-positive}.

Linearizing $T_{c,\alpha}$ at $\phi$ with $w_n = \phi + \omega_n$, where $\omega_n \in H^{\alpha}_{\rm per}(-\pi,\pi)$,
yields the linearized iterative rule:
\begin{equation}
\label{lin-T}
\omega_{n+1} = - \frac{2 \langle \mathcal{L}_{c,\alpha} \phi, \omega_n \rangle}{\langle \mathcal{L}_{c,\alpha} \phi, \phi \rangle} \phi +
\mathcal{L}_{c,\alpha}^{-1} (2 \phi \omega_n), \quad n \in \mathbb{N}.
\end{equation}
Since $\mathcal{L}_{c,\alpha}^{-1} (\phi^2) = \phi$ and $\mathcal{L}_{c,\alpha}^{-1} (2 \phi \phi') = \phi'$,
the linearized iterative rule (\ref{lin-T}) is invariant in the constrained space
\begin{equation}
\label{constrained-space}
L^2_c := \left\{ \omega \in L^2_{\rm per}(-\pi,\pi) : \quad \langle \phi^2, \omega \rangle = \langle \phi \phi', \omega \rangle = 0 \right\}.
\end{equation}
To satisfy the two constraints, one can expand $\omega_n = a_n \phi + b_n \phi' + \beta_n$
with $\beta_n \in H^{\alpha}_{\rm per}(-\pi,\pi) \cap L^2_c$ and derive from (\ref{lin-T}):
\begin{equation}
\label{iteration-linear}
a_{n+1} = 0, \quad b_{n+1} = b_n, \quad \beta_{n+1} = \mathcal{L}_T \beta_n,
\end{equation}
where
\begin{equation}
\label{operator-L-T}
\mathcal{L}_T := \mathcal{L}_{c,\alpha}^{-1} (2 \phi \cdot) = {\rm Id} - \mathcal{L}_{c,\alpha}^{-1} \mathcal{H}_{c,\alpha} : \quad
H^{\alpha}_{\rm per}(-\pi,\pi) \cap L^2_c \mapsto H^{\alpha}_{\rm per}(-\pi,\pi) \cap L^2_c
\end{equation}
is the linearized iterative operator with $\mathcal{H}_{c,\alpha}$ given by (\ref{Jacobian}).
The following two results provide sufficient conditions for divergence or convergence of the iterative method (\ref{T}).

\begin{theorem}
\label{theorem-divergence-classic}
Assume $\int_{-\pi}^{\pi} \phi^3 dx \neq 0$.
There exists $w_0 \in H^{\alpha}_{\rm per}(-\pi,\pi)$ near $\phi \in H^{\alpha}_{\rm per}(-\pi,\pi)$
such that the iterative method (\ref{T}) diverges from $\phi$ if $\sigma(\mathcal{L}_T)$ in $L^2_c$
includes at least one eigenvalue outside the unit disk.
\end{theorem}

\begin{proof}
If $\sigma(\mathcal{L}_T)$ in $L^2_c$ admits at least one eigenvalue
outside the unit disk, the corresponding eigenfunction of $\mathcal{L}_T$ 
defines a direction in $H^{\alpha}_{\rm per}(-\pi,\pi)$ along which
the sequence $\{ w_n \}_{n \in \mathbb{N}}$ diverges from the fixed point $\phi$, 
as follows from the unstable manifold theorem. 
\end{proof}

\begin{theorem}
\label{theorem-convergence-classic}
Assume $\int_{-\pi}^{\pi} \phi^3 dx \neq 0$ and $\int_{-\pi}^{\pi} \phi (\phi')^2 dx \neq 0$.
There exists a small $\epsilon_0 > 0$ such that for every $w_0 \in H^{\alpha}_{\rm per}(-\pi,\pi)$
satisfying
\begin{equation}
\label{estimate-0}
\epsilon := \| w_0 - \phi \|_{H^{\alpha}_{\rm per}} \leq \epsilon_0,
\end{equation}
there exist $b_*$ satisfying $|b_*| \leq C \epsilon$ for some $\epsilon$-independent $C > 0$
such that the iterative method (\ref{T}) converges to $\phi(\cdot - b_*)$ if
$\sigma(\mathcal{L}_T)$ in $L^2_c$ is located inside the unit disk.
\end{theorem}

\begin{proof}
Let us first assume that $w_0 \in H^{\alpha}_{\rm per}(-\pi,\pi)$ is even, in which case the assertion
is true with $b_* = 0$. Since $\mathcal{L}_{c,\alpha}$ maps
even functions to even functions, the sequence of functions $\{ w_n \}_{n \in \mathbb{N}}$ in $H^{\alpha}_{\rm per}(-\pi,\pi)$
generated by (\ref{T}) is even. Therefore, the linearization $w_n = \phi + \omega_n$ and the decomposition
$\omega_n = a_n \phi + b_n \phi' + \beta_n$ yields $b_n = 0$ for every $n \geq 0$. The linear iterative formula (\ref{iteration-linear})
yields $a_n = 0$ for every $n \geq 1$ even if $a_0 \neq 0$.
The linearized operator $\mathcal{L}_T$ given by (\ref{operator-L-T}) is a strict contraction 
if $\sigma(\mathcal{L}_T)$ in $L^2_c$ is located inside the unit disk.
Convergence of the sequence to $\phi$ follows by Banach's fixed-point theorem (Theorem 1.A in \cite{Zeidler}).

Let us now relax the condition that the initial guess $w_0 \in H^{\alpha}_{\rm per}(-\pi,\pi)$ is even.
In order to control the projection $b_n$ in the decomposition $\omega_n = a_n \phi + b_n \phi' + \beta_n$,
we need to use tools of the modulation theory for periodic waves, see, e.g., Section 5 in \cite{GalPel2015}.
Instead of defining $b_n$ by $\omega_n = a_n \phi + b_n \phi' + \beta_n$, we define $b_n \in \mathbb{R}$
by using the decomposition
\begin{equation}
\label{dec-w}
w_n(x) = \phi(x-b_n) + \omega_n(x-b_n)
\end{equation}
and the orthogonality condition
\begin{equation}
\label{constraint-w}
\langle \phi \phi', \omega_n \rangle = 0.
\end{equation}
By a standard application  of the implicit function theorem, see, e.g., Lemma 6.1 in \cite{GalPel2015}, for every
$w_n \in H^{\alpha}_{\rm per}(-\pi,\pi)$ satisfying
\begin{equation}
\label{arg-w}
\epsilon_n := \inf_{b \in [-\pi,\pi]} \| w_n - \phi(\cdot - b) \|_{H^{\alpha}_{\rm per}} \leq \epsilon_0,
\end{equation}
the decomposition (\ref{dec-w})--(\ref{constraint-w}) is unique under the assumption $\int_{-\pi}^{\pi} \phi (\phi')^2 dx \neq 0$
with uniquely defined $b_n$ near the argument of the infimum
in (\ref{arg-w}) and uniquely defined $\omega_n$ satisfying
\begin{equation}
\label{estimate-1}
\| \omega_n \|_{H^{\alpha}_{\rm per}} \leq C_0 \epsilon_n
\end{equation}
for some $\epsilon_n$-independent constant $C_0 > 0$.

Substituting the decomposition (\ref{dec-w}) into the iterative method (\ref{T}) and using the translational
invariance in $x$, we obtain the equivalent iterative scheme:
\begin{equation}
\label{T-equiv}
\omega_{n+1} = \phi(\cdot + \Delta b_n) - \phi + T'(\phi(\cdot + \Delta b_n) \omega_n(\cdot + \Delta b_n)
+ N(\omega_n(\cdot + \Delta b_n)),
\end{equation}
where $\Delta b_n := b_{n+1} - b_n$, $T'(\phi) \omega_n$ denotes the linearized iterative operator given by
the right-hand side in (\ref{lin-T}), and $N(\omega_n)$ is the nonlinear terms satisfying
\begin{equation}
\label{estimate-2}
\| N(\omega_n) \|_{H^{\alpha}_{\rm per}} \leq C \| \omega_n \|^2_{H^{\alpha}_{\rm per}},
\end{equation}
for every $\omega_n \in B_{\rho}(0) := \left\{ \omega \in H^{\alpha}_{\rm per}(-\pi,\pi) :
\;\; \| \omega \|_{H^{\alpha}_{\rm per}} \leq \rho \right\}$, where the constant $C > 0$ does not
depend on $\rho$ provided the radius $\rho$ of the ball $B_{\rho}(0)$ is small. Thanks
to (\ref{estimate-0}) and (\ref{estimate-1}), we work with $\rho = C \epsilon$ for some
positive $\epsilon$-independent constant $C$.

By using the constraint (\ref{constraint-w}) both for $\omega_n$ and $\omega_{n+1}$, we derive
the following equation for $\Delta b_n$:
\begin{equation}
\label{T-b}
0 = \langle \phi \phi', \phi(\cdot + \Delta b_n) - \phi \rangle +
\langle \phi \phi', T'(\phi(\cdot + \Delta b_n) \omega_n(\cdot + \Delta b_n) \rangle
+ \langle \phi \phi', N(\omega_n(\cdot + \Delta b_n)) \rangle.
\end{equation}
This equation can be treated as the root-finding problem $F(\Delta b_n,\omega_n) = 0$,
where
$$
F : \mathbb{R} \times H^{\alpha}_{\rm per}(-\pi,\pi) \mapsto \mathbb{R}
$$
is a smooth function in its variables satisfying $F(0,0) = 0$ and $\partial_{\Delta b_n} F(0,0) \neq 0$
thanks to smoothness of $\phi \in H^{\infty}_{\rm per}(-\pi,\pi)$ and $N(\omega_n)$ as well as
the assumption $\int_{-\pi}^{\pi} \phi (\phi')^2 dx \neq 0$. By the implicit function theorem,
the root-finding problem (\ref{T-b}) is uniquely solvable in $\Delta b_n$
for every $\omega_n \in B_{\rho}(0)$ with small $\rho > 0$. Moreover, thanks to
$\langle \phi \phi', T'(\phi) \omega_n \rangle = \langle \phi \phi', \omega_n \rangle = 0$ and (\ref{estimate-2}),
the uniquely found $\Delta b_n$ satisfies the bound
\begin{equation}
\label{estimate-3}
|\Delta b_n| \leq C \| \omega_n \|^2_{H^{\alpha}_{\rm per}},
\end{equation}
for some constant $C > 0$ that does not depend on the small radius $\rho$.

Substituting $\Delta b_n$ satisfying (\ref{estimate-3}) into (\ref{T-equiv})
and decomposing $\omega_n = a_n \phi + \beta_n$ with $a_n \in \mathbb{R}$ and
$\beta_n \in H^{\alpha}_{\rm per}(-\pi,\pi) \cap L^2_c$, we obtain
the linearized problem
\begin{equation}
\label{T-a}
a_{n+1} = 0, \quad \beta_{n+1} = \mathcal{L}_T \beta_n.
\end{equation}
Since $\mathcal{L}_T$ is a strict contraction in $L^2_c$, convergence $a_n \to 0$,
$\Delta b_n \to 0$, and $\beta_n \to 0$ as $n \to \infty$ follows by 
Banach's fixed-point theorem (Theorem 1.A in \cite{Zeidler}). 
Moreover, these sequences converge exponentially fast so that
the sequence $\{ b_n \}_{n \in \mathbb{N}}$ converges to a limit denoted by $b_*$.
Since $|b_* - b_0| \leq C \epsilon^2$ thanks to (\ref{estimate-1}) and (\ref{estimate-3}), 
whereas $|b_0| \leq C \epsilon$ thanks to (\ref{estimate-0}), (\ref{arg-w}), and triangle inequality,
we also have $|b_*| \leq C \epsilon$ for some $\epsilon$-independent $C > 0$.
The assertion is proven thanks
to the decomposition (\ref{dec-w}) with $\omega_n = a_n \phi + \beta_n$.
\end{proof}

\begin{remark}
Compared to Section 6 in \cite{GalPel2015}, where standard orthogonality condition
$\langle \phi', w \rangle = 0$ was used together with the energy conservation,
we have to use the modified orthogonality condition $\langle \phi \phi', w \rangle = 0$
in order to comply with the iterative scheme (\ref{T-equiv}) which results in the non-self-adjoint
linearized operator $T'(\phi) \omega_n$ given by the right-hand side of (\ref{lin-T}).
\end{remark}

In order to compute $\sigma(\mathcal{L}_T)$ in $L^2_c$ 
used in Theorems \ref{theorem-divergence-classic} and
\ref{theorem-convergence-classic}, we study
the spectrum of $\mathcal{L}_{c,\alpha}^{-1} \mathcal{H}_{c,\alpha}$
in $L^2_{\rm per}(-\pi,\pi)$. Analytical results on convergence of the method
for $c \gtrsim 1$ and divergence for $c > 2^{\alpha}$ are obtained in Sections \ref{sec-3-1}
and \ref{sec-3-2} respectively. These results give the proof of Theorem \ref{theorem-main-1}.
Numerical results showing convergence or divergence of the method for $c$ in $(1,2^{\alpha})$
are obtained in Section \ref{sec-3-3} for $\alpha = 2$ and $\alpha = 1$.

\subsection{Case $c \gtrsim 1$}
\label{sec-3-1}

Here we prove that the iterative method converges near the single-lobe periodic wave $\phi$
in the small-amplitude limit for $c \gtrsim 1$ if $\alpha > \alpha_1$
and diverges if $\alpha \in (\alpha_0,\alpha_1)$, where $\alpha_0$ and $\alpha_1$ are
given by (\ref{alpha}). Note that $\alpha_0 \approx 0.585$
and $\alpha_1 \approx 1.322$ so that $1/2 < \alpha_0 < 1 < \alpha_1 < 2$.

The following lemma characterizes the spectrum of $\mathcal{L}_{c,\alpha}^{-1} \mathcal{H}_{c,\alpha}$ in
$L^2_{\rm per}(-\pi,\pi)$ for $c \gtrsim 1$ and $\alpha > \alpha_0$.

\begin{lemma}
\label{lemma-3}
For every $c \gtrsim 1$ and $\alpha > \alpha_0$, $\sigma(\mathcal{L}_{c,\alpha}^{-1} \mathcal{H}_{c,\alpha})$
in $L^2_{\rm per}(-\pi,\pi)$ consists of a countable sequence of eigenvalues in a neighborhood of $1$
and simple eigenvalues $\{ -1, 0, \lambda_1, \lambda_2 \}$ with
$$
\lambda_1 \to \frac{2^{\alpha+1}-5}{2^{\alpha+1}-3} \quad \mbox{\rm and} \quad
\lambda_2 \to 2 \quad \mbox{\rm as} \quad c \to 1.
$$
Moreover, $\lambda_2 < 2$ for $c \gtrsim 1$, whereas
$\lambda_1 < 0$ if $\alpha \in (\alpha_0,\alpha_1)$ and $\lambda_1 \in (0,1)$ if $\alpha > \alpha_1$.
\end{lemma}

\begin{proof}
It follows from (\ref{spectrum-L}) that for every $c \gtrsim 1$, the operator $\mathcal{L}_{c,\alpha}$ in $L^2_{\rm per}(-\pi,\pi)$
is invertible and
$$
\sigma(\mathcal{L}_{c,\alpha}^{-1})=\{ (-c + |n|^\alpha)^{-1}, \;\; n\in\mathbb{Z}\}.
$$
Since the sequence of eigenvalues is squared summable if $\alpha > 1/2$,
the linear bounded operator $\mathcal{L}_{c,\alpha}^{-1}$ is of the Hilbert-Schmidt class (see Example 2 in Section 5.16 of \cite{Zeidler}),
hence it is compact. The linear operator $\mathcal{L}_T$ in $L^2_{\rm per}(-\pi,\pi)$ is a composition
of a bounded operator $2 \phi \cdot$ and a compact (Hilbert--Schmidt) operator $\mathcal{L}_{c,\alpha}^{-1}$,
hence $\mathcal{L}_T$ is a compact operator and $\sigma(\mathcal{L}_T)$ in $L^2_{\rm per}(-\pi,\pi)$
consists of a sequence of eigenvalues converging to $0$. Thanks to the representation (\ref{operator-L-T}),
$\sigma(\mathcal{L}_{c,\alpha}^{-1} \mathcal{H}_{c,\alpha})$ in $L^2_{\rm per}(-\pi,\pi)$
consists of a sequence of eigenvalues converging to $1$.

Eigenvalues $\{-1,0\}$ of $\mathcal{L}_{c,\alpha}^{-1} \mathcal{H}_{c,\alpha}$ in $L^2_{\rm per}(-\pi,\pi)$
follow from exact computations for every $c > 1$:
\begin{equation}
\label{eigenvalues-exact}
\mathcal{L}_{c,\alpha}^{-1} \mathcal{H}_{c,\alpha} \phi = - \phi \quad \mbox{\rm and} \quad \mathcal{L}_{c,\alpha}^{-1} \mathcal{H}_{c,\alpha} \phi' = 0.
\end{equation}
In order to identify other eigenvalues of $\mathcal{L}_{c,\alpha}^{-1} \mathcal{H}_{c,\alpha}$ in $L^2_{\rm per}(-\pi,\pi)$,
we consider the generalized eigenvalue problem (\ref{gEp}) defined by linear operators
$\mathcal{L}_{c,\alpha}$ and $\mathcal{H}_{c,\alpha}$ in $L^2_{\rm per}(-\pi,\pi)$ with the domains 
in $H^{\alpha}_{\rm per}(-\pi,\pi)$.

Since $\mathcal{H}_{c=1,\alpha}$ coincides with $\mathcal{L}_{c=1,\alpha}$,
the generalized eigenvalue problem (\ref{gEp}) for $c = 1$
admits only one solution $\lambda = 1$ for every $v \in H^{\alpha}_{\rm per}(-\pi,\pi) \backslash \{ e^{i x}, e^{-i x} \}$.
Since $(\phi,c)$ depend analytically on $a$ in Theorem \ref{lemma-small},
by the analytic perturbation theory (Theorem VII.1.7 in \cite{Kato}),
the eigenvalues of $\mathcal{L}_{c,\alpha}^{-1} \mathcal{H}_{c,\alpha}$ in $L^2_{\rm per}(-\pi,\pi)$
for every $c \gtrsim 1$ are divided into two sets: a countable sequence of eigenvalues near $1$ and converging to $1$ related to the
subspace $L^2_{\rm per}(-\pi,\pi) \backslash \{ e^{i x}, e^{-i x} \}$ and a finite number of eigenvalues
related to the subspace $\{ e^{i x}, e^{-i x} \}$. The second set includes eigenvalues $\{-1,0\}$ due to the exact solutions
(\ref{eigenvalues-exact}) for every $c > 1$. The subspace
$\{ e^{i x}, e^{-i x} \}$ may be related to more than two simple eigenvalues in the generalized
eigenvalue problem (\ref{gEp}) because both $\mathcal{H}_{c = 1,\alpha}$ and $\mathcal{L}_{c = 1, \alpha}$ vanish
on the subspace.

In order to study all possible eigenvalues of  $\mathcal{L}_{c,\alpha}^{-1} \mathcal{H}_{c,\alpha}$ in $L^2_{\rm per}(-\pi,\pi)$
related to the subspace $\{ e^{i x}, e^{-i x} \}$, we perform perturbation expansions.
Since $\mathcal{L}_{c,\alpha}$ and $\mathcal{H}_{c,\alpha}$ are closed in the subspaces of
even and odd functions in $L^2_{\rm per}(-\pi,\pi)$, the generalized eigenvalue problem (\ref{gEp})
can be uncoupled in these subspaces. By using (\ref{speed-expansion}) and (\ref{H-expansions}),
we rewrite the generalized eigenvalue problem (\ref{gEp}) in the perturbed form:
\begin{eqnarray}
\nonumber
(\lambda - 1) \left[ 1 + D_{\alpha} + c_2 a^2 + c_4 a^4 + \mathcal{O}(a^6) \right] v & \phantom{t} & \\
- 2 \left[ a \cos(x) + a^2 \phi_2(x) + a^3 \phi_3(x) + a^4 \phi_4(x) + \mathcal{O}(a^5) \right] v & = & 0.
\label{gen-eig-exp}
\end{eqnarray}
Assuming $\lambda \neq 1$, we are looking for perturbative expansions of the eigenvalues related to
the even and odd subspace of $\{ e^{i x}, e^{-i x} \}$ separately from each other. For the even subspace, we set
\begin{equation}
\label{expansion-even}
v(x) = \cos(x) + a v_1(x) + a^2 v_2(x) + \mathcal{O}(a^3)
\end{equation}
and obtain recursively
\begin{eqnarray*}
\left\{ \begin{array}{l}
\mathcal{O}(a) \; : \quad (\lambda - 1)\left( 1 + D_{\alpha} \right) v_1 =  1 + \cos(2x), \\
\mathcal{O}(a^2) : \;\;\;(\lambda - 1)\left( 1 + D_{\alpha} \right) v_2 + (\lambda - 1) c_2 \cos(x)
=  2 \cos(x) (v_1 + \phi_2).
\end{array} \right.
\end{eqnarray*}
At $\mathcal{O}(a)$, we obtain the exact solution in $H^{\alpha}_{\rm per}(-\pi,\pi)$:
\begin{eqnarray}
\label{corr-v-1}
v_1(x) = \frac{1}{\lambda - 1} \left[ 1 -\frac{\cos(2x)}{2^\alpha-1} \right].
\end{eqnarray}
The linear inhomogeneous equation at $\mathcal{O}(a^2)$ admits a solution
$v_2 \in H^{\alpha}_{\rm per}(-\pi,\pi)$ if and only if  $\lambda$ satisfies
$$
\left[ \lambda - \frac{2}{\lambda - 1} \right] c_2 = 0.
$$
If $\alpha > \alpha_0$, then $c_2 \neq 0$ and $\lambda$ satisfies the quadratic equation $\lambda ( \lambda - 1) = 2$ with two roots $\{ -1,2\}$.
For each of the two roots, we obtain the exact solution in $H^{\alpha}_{\rm per}(-\pi,\pi)$:
\begin{eqnarray}
\label{corr-v-2}
v_2(x) = \frac{(3-\lambda) \cos(3x)}{2 (\lambda - 1)^2 (2^\alpha-1)(3^\alpha-1)}.
\end{eqnarray}
For the odd subspace, we set
\begin{equation}
\label{expansion-odd}
v(x) = \sin(x) + a v_1(x) + a^2 v_2(x) + \mathcal{O}(a^3)
\end{equation}
and obtain recursively
\begin{eqnarray*}
\left\{ \begin{array}{l}
\mathcal{O}(a) \; : \quad (\lambda - 1)\left( 1 + D_{\alpha} \right) v_1 =  \sin(2x), \\
\mathcal{O}(a^2) : \quad (\lambda - 1) \left( 1 + D_{\alpha} \right) v_2 + (\lambda - 1) c_2 \sin(x) =
2 (\cos(x) v_1 + \sin(x) \phi_2).
\end{array} \right.
\end{eqnarray*}
At $\mathcal{O}(a)$, we obtain the exact solution in $H^{\alpha}_{\rm per}(-\pi,\pi)$:
\begin{eqnarray}
\label{corr-v-3}
v_1(x) = -\frac{\sin(2x)}{(\lambda - 1) (2^\alpha-1)}.
\end{eqnarray}
The linear inhomogeneous equation at $\mathcal{O}(a^2)$ admits a solution
$v_2 \in H^{\alpha}_{\rm per}(-\pi,\pi)$ if and only if  $\lambda$ satisfies
$$
\lambda c_2 + \frac{\lambda}{(\lambda - 1) (2^{\alpha} -1)} = 0.
$$
If $\alpha > \alpha_0$, then $c_2 \neq 0$ and $\lambda$ satisfies the quadratic equation
$\lambda \left[ (2^{\alpha+1} - 3) \lambda - (2^{\alpha+1} - 5) \right] = 0$
with two roots $\{ 0, \frac{2^{\alpha+1}-5}{2^{\alpha+1}-3}\}$.
For each of the two roots, we obtain the exact solution in $H^{\alpha}_{\rm per}(-\pi,\pi)$:
\begin{eqnarray}
\label{corr-v-4}
v_2(x) & = & \frac{(3-\lambda)  \sin(3x)}{2  (\lambda - 1)^2 (2^{\alpha}-1)(3^{\alpha}-1)}.
\end{eqnarray}
Summarizing, we have obtained four eigenvalues related to the subspace $\{ e^{i x}, e^{-i x} \}$,
which are located as $c \to 1$ at the points $\{ -1,0,\frac{2^{\alpha+1}-5}{2^{\alpha+1}-3},2\}$.

The eigenvalues $\{-1,0\}$ are preserved for every $c > 1$ thanks to the exact solution (\ref{eigenvalues-exact}).
However, the eigenvalues $\{ \lambda_1,\lambda_2\}$ near $\{ \frac{2^{\alpha+1}-5}{2^{\alpha+1}-3},2\}$
depend generally on $c$. It follows by the perturbation theory that $\lambda_1 < 0$ for $c \gtrsim 1$ if
$\alpha \in (\alpha_0,\alpha_1)$ and $\lambda_1 \in (0,1)$ for $c \gtrsim 1$ if $\alpha > \alpha_1$.
We now claim that $\lambda_2 < 2$ for $c \gtrsim 1$ if $\alpha > \alpha_0$.
To prove this claim, we use the extended spectral problem (\ref{gen-eig-exp}) up to the order $\mathcal{O}(a^4)$.
Hence, instead of the expansion (\ref{expansion-even}) with (\ref{corr-v-1}) and (\ref{corr-v-2}),
we use the expansions
\begin{equation}
\label{Lambda1}
\left\{ \begin{array}{l}
v(x) = \cos(x) + a v_1(x) + a^2 v_2(x) + a^3 v_3(x) + a^4 v_4(x) + \mathcal{O}(a^5), \\
\lambda = 2 + \Lambda_2 a^2 + \mathcal{O}(a^4), \end{array} \right.
\end{equation}
where
$$
v_1(x) = 1 -\frac{\cos(2x)}{2^\alpha-1}, \quad
v_2(x) = \frac{\cos(3x)}{2 (2^\alpha-1)(3^\alpha-1)}.
$$
We obtain from the extended spectral problem (\ref{gen-eig-exp})
the linear inhomogeneous equations:
\begin{eqnarray*}
\left\{ \begin{array}{l}
\mathcal{O}(a^3) : \;\;\; (1 + D_{\alpha}) v_3 + \Lambda_2 (1 + D_{\alpha}) v_1 + c_2 v_1 =
2 \left[ \cos(x) (v_2 + \phi_3) + \phi_2 v_1\right], \\
\mathcal{O}(a^4) : \;\;\; (1 + D_{\alpha}) v_4 + \Lambda_2 (1 + D_{\alpha}) v_2 + c_2 v_2 + (c_4 + c_2 \Lambda_2) \cos(x) =
2 \left[ \cos(x) (v_3 + \phi_4) + \phi_2 v_2 + \phi_3 v_1 \right].
\end{array} \right.
\end{eqnarray*}
The linear inhomogeneous equation at $\mathcal{O}(a^3)$ admits the explicit solution:
\begin{eqnarray*}
v_3(x) & = & \frac{3^{\alpha} - 2^{\alpha + 1} + 1}{2 (2^{\alpha} - 1)^2 (3^{\alpha} - 1) (4^{\alpha} - 1)} \cos(4x)
+ \left[ \frac{\Lambda_2}{2^{\alpha} - 1} - \frac{1 + (2+c_2) (3^{\alpha} -1)}{(2^{\alpha} - 1)^{2} (3^{\alpha} - 1)} \right] \cos(2x) \\
& \phantom{t} & - \left( \Lambda_2 + c_2 + 1+ \frac{1}{2(2^\alpha-1)^2}\right).
\end{eqnarray*}
The linear inhomogeneous equation at $\mathcal{O}(a^4)$ admits a solution
$v_4 \in H^{\alpha}_{\rm per}(-\pi,\pi)$ if and only if  $\Lambda_2$ is given by
\begin{equation}
\label{Lambda2}
\Lambda_2 = -1 +\frac{3}{2^\alpha -1} - \frac{7}{ 2^{\alpha+1}-3}.
\end{equation}
It is easy to see that $\Lambda_2$ has a vertical asymptote at $\alpha= \alpha_0$.
By plotting $\Lambda_2$ versus $\alpha$ on Figure \ref{fig:Lambda2_alpha},
we verify that $\Lambda_2 < 0$ for every $\alpha > \alpha_0$. Hence the eigenvalue
$\lambda = 2 + \Lambda_2 a^2 + \mathcal{O}(a^4)$ satisfies $\lambda < 2$ for every
$c \gtrsim 1$ and $\alpha > \alpha_0$.
\end{proof}

\begin{figure}[th]
\centering
\includegraphics[width=0.6\linewidth]{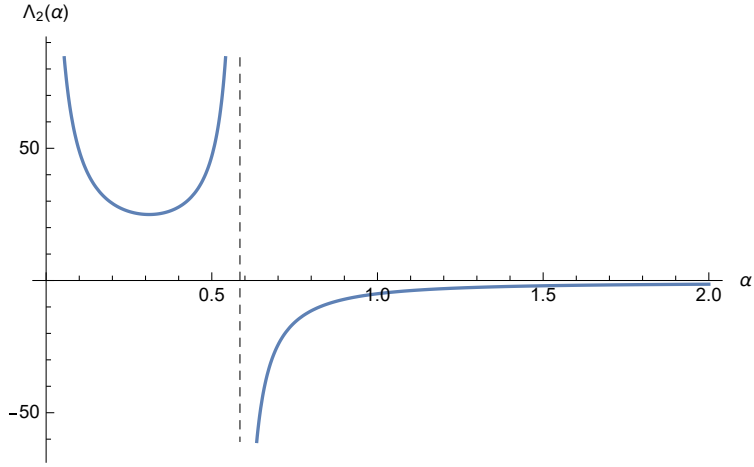}
\caption{Plot of $\Lambda_2$ versus $\alpha$.}
\label{fig:Lambda2_alpha}
\end{figure}

\begin{corollary}
\label{corollary-1}
For every $c \gtrsim 1$, the iterative method (\ref{T})
converges to $\phi$ in $H^{\alpha}_{\rm per}(-\pi,\pi)$ if $\alpha > \alpha_1$
and diverges from $\phi$ if $\alpha \in (\alpha_0,\alpha_1)$.
\end{corollary}

\begin{proof}
Assumptions (\ref{assumption-on-phi}) used in Theorems
\ref{theorem-divergence-classic} and \ref{theorem-convergence-classic}
have been verified for $c \gtrsim 1$ in Lemma \ref{proposition-negative}.

If $\alpha > \alpha_1$, then $\lambda_1 \in (0,1)$ for $c \gtrsim 1$ by Lemma \ref{lemma-3}.
By using the representation (\ref{operator-L-T}) and the count of eigenvalues
of the generalized eigenvalue problem (\ref{gEp}) in Lemma \ref{lemma-3},
we can see that $\sigma(\mathcal{L}_T)$ in $L^2_{\rm per}(-\pi,\pi)$ consists of a countable sequence
of eigenvalues in a neighborhood of $0$ and converging to $0$
for every $c \gtrsim 1$, two simple eigenvalues inside the interval $(-1,1)$,
and two additional simple eigenvalues: $1$ related to the eigenfunction $\phi'$
and $2$ related to the eigenfunction $\phi$. The two constraints in (\ref{constrained-space})
remove the latter two eigenvalues so that the operator
$\mathcal{L}_T$ is a strict contraction in $L^2_c$ for every $c \gtrsim 1$ if $\alpha > \alpha_1$.
Convergence of the iterative method (\ref{T}) for $\alpha > \alpha_1$ follows by
Theorem \ref{theorem-convergence-classic}.

If $\alpha \in (\alpha_0,\alpha_1)$, then $\lambda_1 < 0$ for $c \gtrsim 1$ by Lemma \ref{lemma-3}.
Then, $\sigma(\mathcal{L}_T)$ in $L^2_{\rm per}(-\pi,\pi)$ consists of a countable sequence
of eigenvalues in a neighborhood of $0$ and converging to $0$
for every $c \gtrsim 1$, one simple eigenvalue inside the interval $(-1,1)$,
simple eigenvalue $1$ related to the eigenfunction $\phi'$,
simple eigenvalue $2$ related to the eigenfunction $\phi$,
and an additional simple eigenvalue bigger than $1$ with an odd eigenfunction denoted by $v_*$.
Because of the orthogonality conditions
$$
\langle  \mathcal{L}_{c,\alpha} v_j, v_k \rangle = 0, \quad j \neq k,
$$
between eigenfunctions $v_j$ and $v_k$ of the generalized eigenvalue problem (\ref{gEp}) for distinct eigenvalues,
we verify that $\langle \phi^2, v_* \rangle = \langle \phi \phi', v_* \rangle = 0$, which implies that $v_* \in L^2_c$. Therefore,
$\sigma(\mathcal{L}_T)$ in $L^2_c$ contains exactly one eigenvalue outside the unit disk
for every $c \gtrsim 1$ if $\alpha \in (\alpha_0,\alpha_1)$.
Divergence of the iterative method (\ref{T}) for $\alpha \in (\alpha_0,\alpha_1)$ follows by Theorem \ref{theorem-divergence-classic}.
\end{proof}

\begin{remark}
Since the unstable eigenfunction $v_*$ is odd, divergence of the iterative method (\ref{T})
for $\alpha \in (\alpha_0,\alpha_1)$ is only observed if the initial guess $w_0 \in H^{\alpha}_{\rm per}(-\pi,\pi)$
is not even but of a general form.
\end{remark}

Although Theorem \ref{theorem-main-1} follows already from Corollary \ref{corollary-1},
we would like to add few more details on the eigenfunctions
of the generalized eigenvalue problem (\ref{gEp}).

The numbers of negative eigenvalues of operators $\mathcal{L}_{c,\alpha}$ and $\mathcal{H}_{c,\alpha}$
are affected by the constraint $\langle \phi^2, \alpha \rangle = 0$ in (\ref{constrained-space}).
As is well-known (see, e.g., Theorem 4.1 in \cite{pel-book}), if $n$ is the number of negative
eigenvalues of a self-adjoint invertible operator $\mathcal{L}$ in a Hilbert space and if
$\langle \mathcal{L}^{-1} \phi^2, \phi^2 \rangle < 0$, then the restriction of the self-adjoint
operator to the constraint $\langle \phi^2, \alpha \rangle = 0$ has one less negative eigenvalues
$n-1$. We compute:
\begin{equation}
\label{cube-constraint}
\langle \mathcal{L}_{c,\alpha}^{-1} \phi^2, \phi^2 \rangle = \langle \phi, \phi^2 \rangle, \quad
\langle \mathcal{H}_{c,\alpha}^{-1} \phi^2, \phi^2 \rangle = -\langle \phi, \phi^2 \rangle.
\end{equation}
By Lemma \ref{proposition-negative}, we have $\langle \phi, \phi^2 \rangle = \int_{-\pi}^{\pi} \phi(x)^3 dx < 0$
for every $c \gtrsim 1$ if $\alpha > \alpha_0$. Therefore,
$\mathcal{L}_{c,\alpha}$ restricted to $L^2_c$ has one less negative eigenvalue compared to $\mathcal{L}_{c,\alpha}$
in $L^2_{\rm per}(-\pi,\pi)$, whereas $\mathcal{H}_{c,\alpha}$ restricted to $L^2_c$ has still one simple negative eigenvalue.
In the space of even functions, $\mathcal{L}_{c,\alpha}$ and $\mathcal{H}_{c,\alpha}$ restricted to $L^2_c$ have only
one simple negative eigenvalue. By Theorem 1 in \cite{ChugPel}, the generalized eigenvalue problem (\ref{gEp})
admits one of the following in the subspace of even functions in $L^2_{\rm per}(-\pi,\pi)$:
\begin{itemize}
\item two simple negative eigenvalues $\lambda$ with the two eigenfunctions $v$
for which the sign of $\langle \mathcal{L}_{c,\alpha} v, v \rangle$ is opposite to the sign
of $\langle \mathcal{H}_{c,\alpha} v, v \rangle$;
\item a simple positive eigenvalue $\lambda$ with the eigenfunction $v$
for which the sign of $\langle \mathcal{L}_{c,\alpha} v, v \rangle$ and
$\langle \mathcal{H}_{c,\alpha} v, v \rangle$ are negative;
\item a double defective real eigenvalue $\lambda$ with only one eigenfunction $v$
for which $\langle \mathcal{L}_{c,\alpha} v, v \rangle = \langle \mathcal{H}_{c,\alpha} v, v \rangle = 0$;
\item a complex-conjugate pair of eigenvalues $\lambda$.
\end{itemize}
The following result shows that the second option from the list above is true if $\alpha > \alpha_1$ and $c \gtrsim 1$.

\begin{lemma}
\label{lemma-count}
For every $c \gtrsim 1$ and $\alpha > \alpha_1$, the generalized eigenvalue problem
(\ref{gEp}) admits
\begin{itemize}
\item a simple positive eigenvalue $\lambda$ with the eigenfunction $v$
for which the sign of $\langle \mathcal{L}_{c,\alpha} v, v \rangle$ and
$\langle \mathcal{H}_{c,\alpha} v, v \rangle$ are negative;
\item a simple negative eigenvalue $\lambda$ with the eigenfunction $v$
for which the sign of $\langle \mathcal{L}_{c,\alpha} v, v \rangle$ is negative;
\item a simple zero eigenvalue with the eigenfunction $v$
for which the sign of $\langle \mathcal{L}_{c,\alpha} v, v \rangle$ is negative.
\end{itemize}
The eigenfunction $v$ for all other eigenvalues corresponds to
positive values of $\langle \mathcal{L}_{c,\alpha} v, v \rangle$ and
$\langle \mathcal{H}_{c,\alpha} v, v \rangle$.
\end{lemma}

\begin{proof}
We utilize the perturbative expansions in the proof of Lemma \ref{lemma-3}.
For the even expansion (\ref{expansion-even}), we obtain
$$
\mathcal{L}_{c,\alpha} v = -\frac{a}{\lambda - 1} \left[ 1 + \cos(2x) \right] +
\frac{(3-\lambda) a^2}{2(2^\alpha-1) (\lambda - 1)^2} \cos(3x) -\left(1-\frac{1}{2(2^\alpha-1)}\right)a^2 \cos(x) + \mathcal{O}(a^3).
$$
By evaluating elementary integrals, we obtain
$$
\langle \mathcal{L}_{c,\alpha} v, v \rangle = - \frac{\pi a^2(2^{\alpha+1}-3)}{2(\lambda - 1)^2(2^\alpha-1)} \left[ 2 + (\lambda - 1)^2 \right] + \mathcal{O}(a^3).
$$
If $\alpha > \alpha_0$, then $\langle \mathcal{L}_{c,\alpha} v, v \rangle$ is negative
for both roots $\{ -1,2\}$ of the quadratic equation
$\lambda ( \lambda - 1) = 2$. Since $\langle \mathcal{H}_{c,\alpha} v, v \rangle = \lambda
\langle \mathcal{L}_{c,\alpha} v, v \rangle$, the eigenvalue at $\lambda = 2 + \mathcal{O}(a^2)$
corresponds to the simple positive eigenvalue, for which both quadratic forms are negative,
whereas the eigenvalue $\lambda = -1$ corresponds to the negative eigenvalue, for which
only $\langle \mathcal{L}_{c,\alpha} v, v \rangle$ is negative and $\langle \mathcal{H}_{c,\alpha} v, v \rangle$ is positive.

For the odd expansion (\ref{expansion-odd}), we obtain
$$
\mathcal{L}_{c,\alpha} v = -\frac{a}{\lambda - 1} \sin(2x) +
\frac{(3-\lambda) a^2}{2(2^\alpha-1) (\lambda - 1)^2} \sin(3x) - \left(1-\frac{1 }{2(2^\alpha-1)}\right)a^2 \sin(x) + \mathcal{O}(a^3).
$$
By evaluating elementary integrals, we obtain
$$
\langle \mathcal{L}_{c,\alpha} v, v \rangle = \frac{\pi a^2}{2(2^\alpha-1) (\lambda - 1)^2} \left[ 2 - (2^{\alpha+1}-3) (\lambda - 1)^2 \right] + \mathcal{O}(a^3),
$$
The sign of the quadratic forms depends on the value of $\alpha$ for the two roots
$\{ 0, \frac{2^{\alpha+1}-5}{2^{\alpha+1}-3}\}$ of the quadratic equation
$\lambda[(2^{\alpha+1}-3)\lambda + (5-2^{\alpha+1})] = 0$.
If $\alpha > \alpha_1$, then $\langle \mathcal{L}_{c,\alpha} v, v \rangle$ is negative for
the eigenvalue $\lambda = 0$, for which $\langle \mathcal{H}_{c,\alpha} v, v \rangle$ is zero,
and positive for the eigenvalue $\lambda = \frac{2^{\alpha+1}-5}{2^{\alpha+1}-3} + \mathcal{O}(a^2)$,
for which $\langle \mathcal{H}_{c,\alpha} v, v \rangle$ is also positive.

Every other eigenvalue bifurcating from $\lambda = 1$
corresponds to the positive eigenvalues, for which both quadratic forms $\langle \mathcal{L}_{c,\alpha} v, v \rangle$
and $\langle \mathcal{H}_{c,\alpha} v, v \rangle$ are positive.
\end{proof}

\begin{remark}
For $\alpha \in (\alpha_0,\alpha_1)$ the eigenvalue
$\lambda = \frac{2^{\alpha+1}-5}{2^{\alpha+1}-3} + \mathcal{O}(a^2)$ is negative
and the third item of Lemma \ref{lemma-count} changes as follows.
The sign of $\langle \mathcal{L}_{c,\alpha} v, v \rangle$ is now positive for the eigenvalue $\lambda = 0$
and negative for the eigenvalue $\lambda = \frac{2^{\alpha+1}-5}{2^{\alpha+1}-3} + \mathcal{O}(a^2)$.
Nevertheless, we still count three eigenfunctions $v$ of the generalized eigenvalue problem (\ref{gEp})
with negative values of $\langle \mathcal{L}_{c,\alpha} v, v \rangle$ and one eigenfunction
$v$ with positive values of $\langle \mathcal{H}_{c,\alpha} v, v \rangle$, in agreement with
Theorem 1 in \cite{ChugPel}.
\end{remark}

\subsection{Case $c > 2^{\alpha}$}
\label{sec-3-2}

The first resonance occurs at $c = 2^{\alpha}$, when a double eigenvalue of the operator $\mathcal{L}_{c,\alpha}$ crosses zero and become
a negative eigenvalue for $c > 2^{\alpha}$. Some eigenvalues of the operator $\mathcal{L}_{c,\alpha}^{-1} \mathcal{H}_{c,\alpha}$ may diverge
as $c \to 2^{\alpha}$ and the conclusion on convergence of the iterative method (\ref{T}) may change after the resonance.
Here we prove that the iterative method (\ref{T}) diverges for every $c > 2^{\alpha}$ and $\alpha \in (\alpha_0,2]$,
for which $\mathcal{L}_{c,\alpha}^{-1}$ exists. Compared to the perturbative results in Section \ref{sec-3-1}, 
the restriction $\alpha \leq 2$ is necessary to apply the results of \cite{JH1} in the proof of Lemma \ref{lemma-11}. 
The following lemma specifies the number of negative eigenvalues of
$\mathcal{L}_{c,\alpha}^{-1} \mathcal{H}_{c,\alpha}$ in $L^2_{\rm per}(-\pi,\pi)$.

\begin{lemma}
\label{lemma-4}
For every $c > 2^{\alpha}$ and $\alpha \in (\alpha_0,2]$, for which $\mathcal{L}_{c,\alpha}^{-1}$ exists,
$\sigma(\mathcal{L}_{c,\alpha}^{-1} \mathcal{H}_{c,\alpha})$
in $L^2_{\rm per}(-\pi,\pi)$ includes $N$ negative
eigenvalues (counting with their algebraic multiplicities) with $N \geq 1$
in addition to the simple negative eigenvalue $-1$.
\end{lemma}

\begin{proof}
It follows from (\ref{spectrum-L}) that for $c > 2^{\alpha}$,
$\sigma(\mathcal{L}_{c,\alpha})$ in $L^2_{\rm per}(-\pi,\pi)$ admits $n$ negative eigenvalues
(counting with their algebraic multiplicities) with $n \geq 5$. By Lemma \ref{lemma-11},
$\sigma(\mathcal{H}_{c,\alpha})$ in $L^2_{\rm per}(-\pi,\pi)$ admits only one simple negative eigenvalue
and the simple zero eigenvalue with an odd eigenfunction $\phi'$. In the space of even
functions, $\mathcal{H}_{c,\alpha}$ has only one simple negative eigenvalue and is invertible,
whereas $\mathcal{L}_{c,\alpha}$ has $n_{ev}$ negative eigenvalues with $n_{ev} \geq 3$.
Both $\mathcal{H}_{c,\alpha}$ and $\mathcal{L}_{c,\alpha}$ are self-adjoint in $L^2_{\rm per}(-\pi,\pi)$
with the domain in $H^{\alpha}_{\rm per}(-\pi,\pi)$, as well as in the corresponding
subspaces of even functions. By Theorem 4.1 in \cite{pel-book},
the constraints in $L^2_c$ may only reduce one negative eigenvalue in either
$\mathcal{L}_{c,\alpha}$ or $\mathcal{H}_{c,\alpha}$ (the choice between the two operators depends
on the sign of $\int_{-\pi}^{\pi} \phi^3(x) dx$). In either case,
by Theorem 1 in \cite{ChugPel}, there exist at least $N \geq 1$ negative eigenvalues $\lambda$
of the generalized eigenvalue problem (\ref{gEp}) in $L^2_c$.
\end{proof}

\begin{corollary}
\label{corollary-2}
Assume $\int_{-\pi}^{\pi} \phi^3 dx \neq 0$.
The iterative method (\ref{T}) diverges from $\phi$ for every $c > 2^{\alpha}$ 
and $\alpha \in (\alpha_0,2]$.
\end{corollary}

\begin{proof}
It follows from Lemma \ref{lemma-4} and the representation (\ref{operator-L-T})
that $\sigma(\mathcal{L}_T)$ in $L^2_c$ includes $N \geq 1$ positive eigenvalues
larger than $1$. These eigenvalues of $\mathcal{L}_T$ outside the unit disk
correspond to the eigenfunctions in the constrained subspace (\ref{constrained-space}),
which satisfy the orthogonality conditions
$$
\langle \mathcal{L}_{c,\alpha} \phi, \alpha \rangle = \langle \mathcal{L}_{c,\alpha} \phi', \alpha \rangle = 0.
$$
Divergence of the iterative method (\ref{T}) follows by Theorem \ref{theorem-divergence-classic}.
\end{proof}

\begin{remark}
It follows from the proof of Lemma \ref{lemma-4} that the divergence of the iterative method (\ref{T})
for $c > 2^{\alpha}$ and $\alpha \in (\alpha_0,2]$ is observed if the initial guess $w_0 \in H^{\alpha}_{\rm per}(-\pi,\pi)$
is even.
\end{remark}

\subsection{Case $c \in (1,2^{\alpha})$}
\label{sec-3-3}

Here we address numerically convergence of the iterative method (\ref{T}) near the single-lobe periodic
wave for $c \in (1,2^{\alpha})$. For simplicity of computations, we only consider the classical KdV and BO equations.

For the KdV equation with $\alpha = 2$, we show that the method converges for
$c \gtrsim 1$ in agreement with Corollary \ref{corollary-1}. On the other hand,
we illustrate transition to instability at $c \approx 2.3$ and divergence
of the method for $c \gtrsim 2.3$, which persists until $c = 2^2 = 4$.

Figure \ref{fig:Eigenvalues_against_c1_37} shows eigenvalues of
the generalized eigenvalue problem (\ref{gEp})
computed numerically with the Fourier method for $c \in (1,4)$. Five largest and
five smallest eigenvalues of the operator $\mathcal{L}_{c,\alpha}^{-1}\mathcal{H}_{c,\alpha}$ are shown on
the left panel. In agreement with the result of Lemma \ref{lemma-3}, we observe eigenvalues $\lambda$
near points $\{-1,0, \frac{3}{5}, 2\}$ in addition to a countable sequence of eigenvalues
near $1$. The right panel zooms in eigenvalues near $c = 1$ and shows the asymptotic approximation
of the eigenvalue near $2$ given by (\ref{Lambda1}) and (\ref{Lambda2}) with $\alpha = 2$.

For  $c_* \approx 1.2$, two real eigenvalues coalesce to create a pair of complex eigenvalues
that exist for every $c > c_*$. This transformation of eigenvalues compared to the result of
Lemma \ref{lemma-count} for $c \gtrsim 1$ does not contradict to the count of eigenvalues in Theorem 1 of \cite{ChugPel}.
Figure \ref{ModulusComplexEvals} shows that $|1- \lambda|$ for the eigenvalues of
$\mathcal{L}_T$ remains inside the unit disk for $c\in (c_*, 4)$. Therefore,
the complex eigenvalue pair does not introduce additional instability to the iterative method.

\begin{figure}[h]
	\centering
	\includegraphics[width=0.45\linewidth]{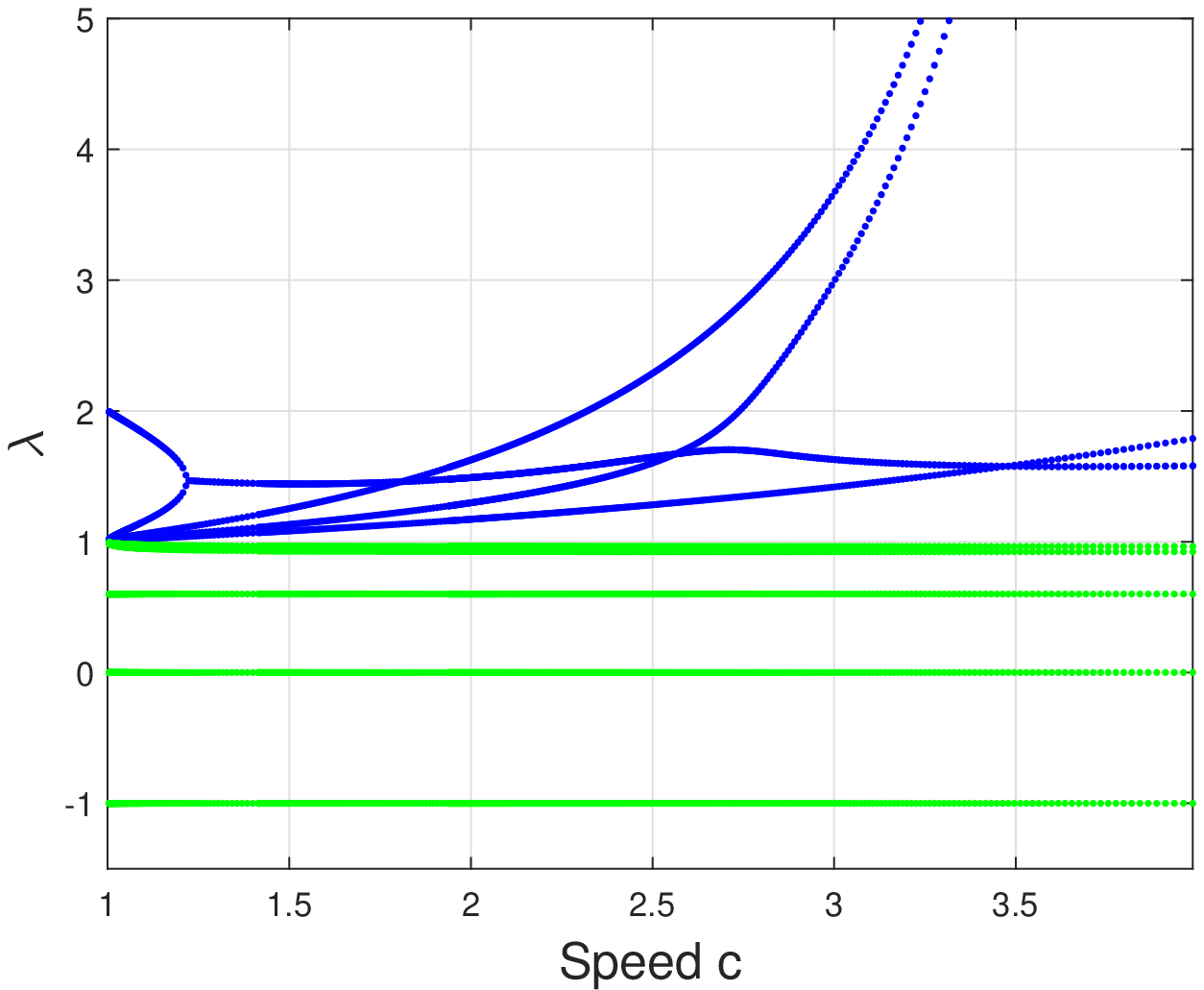}
	\includegraphics[width=0.45\linewidth]{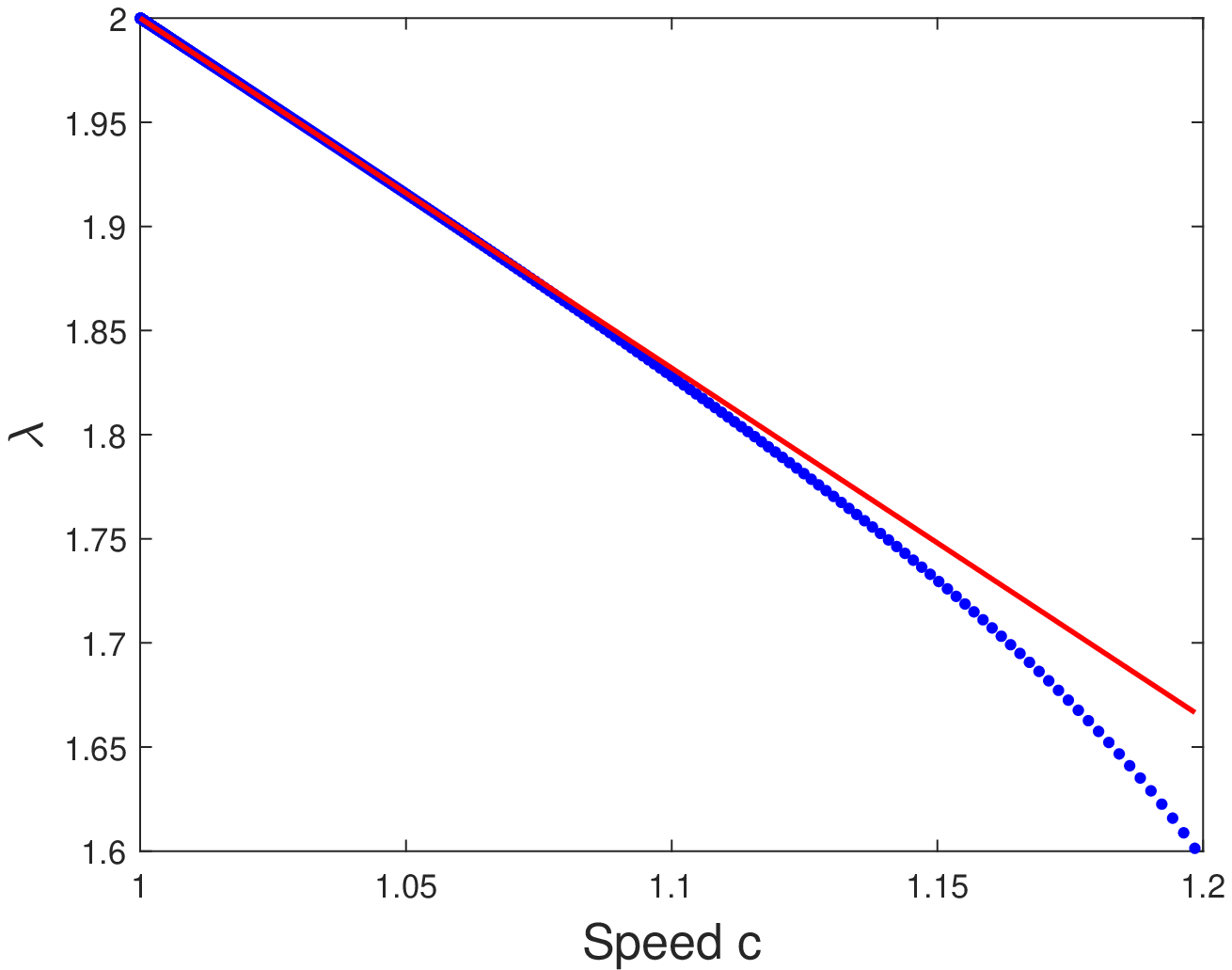}
	\caption{Left: Eigenvalues of the operator $\mathcal{L}_{c,\alpha}^{-1}\mathcal{H}_{c,\alpha}$ for $\alpha = 2$.
Right: Zoom in with the asymptotic dependence given by (\ref{Lambda1}) and (\ref{Lambda2}).}
	\label{fig:Eigenvalues_against_c1_37}
\end{figure}

\begin{figure}[h]
\centering
\includegraphics[width=0.5\linewidth]{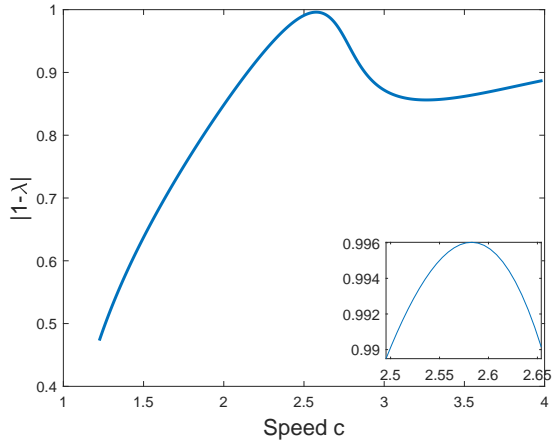}
\caption{The plot of $|1 - \lambda|$ for the complex eigenvalues $\lambda$.
The insert shows that the complex eigenvalues do not reach the boundary of the unit disk.}
\label{ModulusComplexEvals}
\end{figure}

For $c \in (1,c_{**})$ with $c_{**} \approx 2.3$, the spectrum of $\mathcal{L}_T$ in $L^2_c$ 
remain inside the unit disk for $c\in (1,c_{**})$.
However, the largest eigenvalue of $\mathcal{L}^{-1}_{c,\alpha}\mathcal{H}_{c,\alpha}$ crosses the level $2$ for $c = c_{**}$
and the corresponding eigenvalue of $\mathcal{L}_T$ is smaller than $-1$ for $c \in (c_{**},4)$.
This numerical result suggests that the iterative method (\ref{T}) converges for $c\in (1,c_{**})$
and diverge for $c \in (c_{**},4)$. Moreover, for $c_{***} \approx 2.7$, the second largest eigenvalue of
$\mathcal{L}^{-1}_{c,\alpha}\mathcal{H}_{c,\alpha}$ crosses the level $2$, hence the iterative method (\ref{T})
diverges with two unstable eigenvalues for $c \in (c_{***},4)$.

To illustrate convergence of the iterative method (\ref{T}) for $\alpha =2$, we
use the initial function
\begin{equation}
\label{initial-guess}
u_0(x)= a \cos(x)+ \frac{1}{2} a^2 \left(\cos(2x)-3\right) + \varepsilon \sin(x),
\end{equation}
where $a > 0$ and $\varepsilon \in \mathbb{R}$ are small parameters to our disposal.
Notice that we include the $\mathcal{O}(a^2)$ correction term of
the Stokes expansion (\ref{wave-expansion}) in the initial function (\ref{initial-guess}) to avoid
vanishing denominator in the Petviashvili quotient $M$ defined by \eqref{rule}.
Indeed, $\int_{\pi}^{\pi} \cos(x)^3 dx = 0$, whereas
$\int_{-\pi}^{\pi} \phi^3 dx < 0$ for every $c > 1$ and $\alpha = 2$
by Lemma \ref{lemma-kdv}. Computations reported below correspond to $a = 0.4$ and $\varepsilon = 0$;
we have checked that computations for other small values of $a$ and $\varepsilon$ return similar results.

We measure the computational errors in three ways:
the quantity $|1-M_n|$, where $M_n = M(u_n)$,
the distance between two successive approximations $\|u_{n+1}-u_{n}\|_{L^\infty}$,
and the residual error $\| c u_n + u_n'' + u_n^2 \|_{L^{\infty}}$.
If iterations do not converge, we stop the algorithm after 500 iterations.

Figure \ref{fig:Kdv-c2} shows the profile of the last iteration and the three
computational errors versus the number of iterations in the case $c = 2$.
It is seen that the iterative method (\ref{T}) converges to the single-lobe periodic wave,
in agreement with Corollary \ref{corollary-1}. Since
the exact periodic wave is known in \eqref{wave-explicit}--(\ref{speed-explicit}),
we can also compute the distance between the last iteration and the exact solution,
in which case we find $\|u-\phi\|_{L^{\infty}} \approx 2 \cdot 10^{-11}$.
If $\varepsilon \neq 0$ in the initial function (\ref{initial-guess}),
the convergence to the periodic wave is still observed but
the last iteration is shifted from $x = 0$, in agreement with Theorem \ref{theorem-convergence-classic}.

\begin{figure}[h]
	\centering
	\begin{subfigure}[t]{0.5\textwidth}
		\centering
		\includegraphics[width=0.8\linewidth]{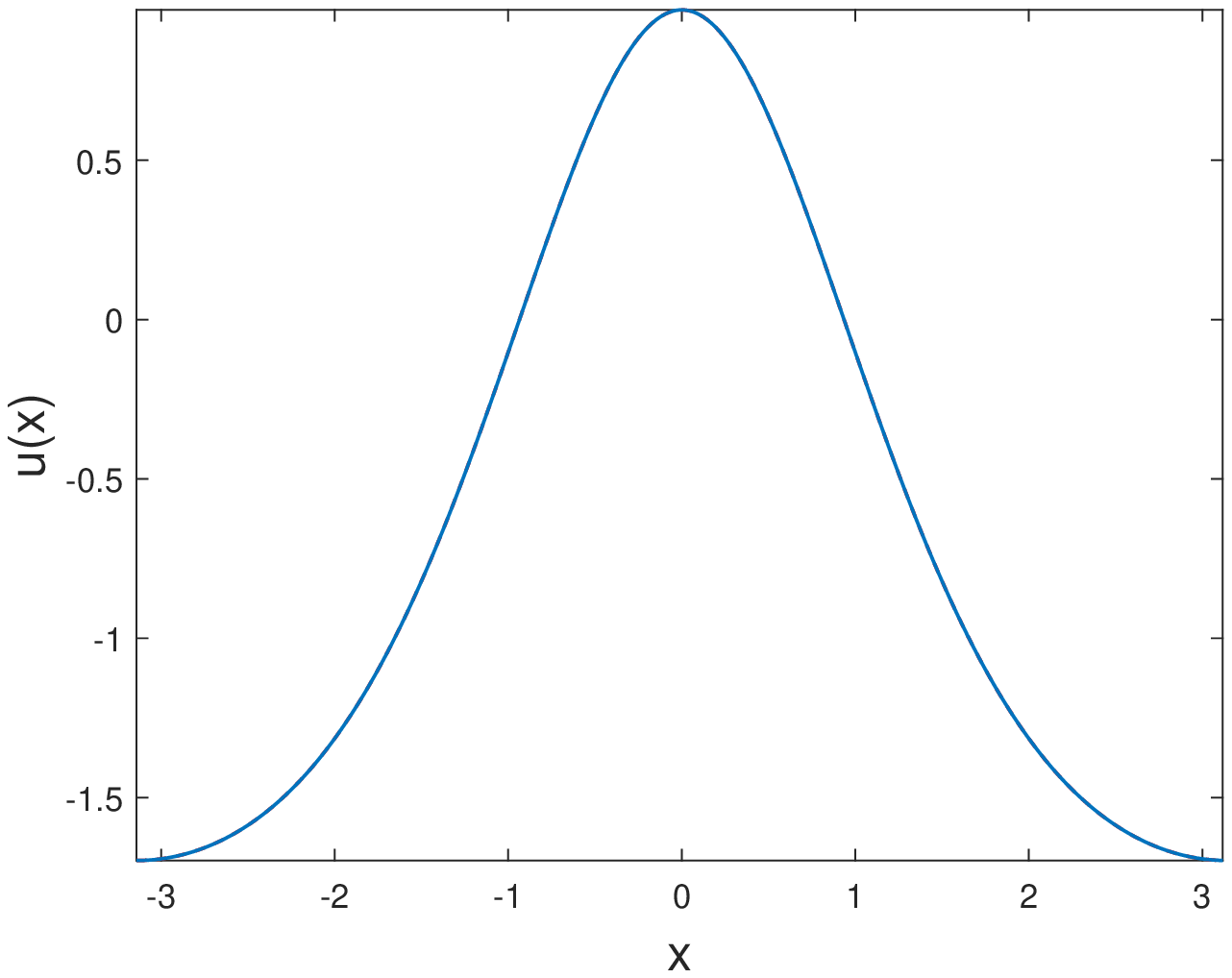}
	\end{subfigure}%
	\begin{subfigure}[t]{0.5\textwidth}
		\centering
		\includegraphics[width=0.8\linewidth]{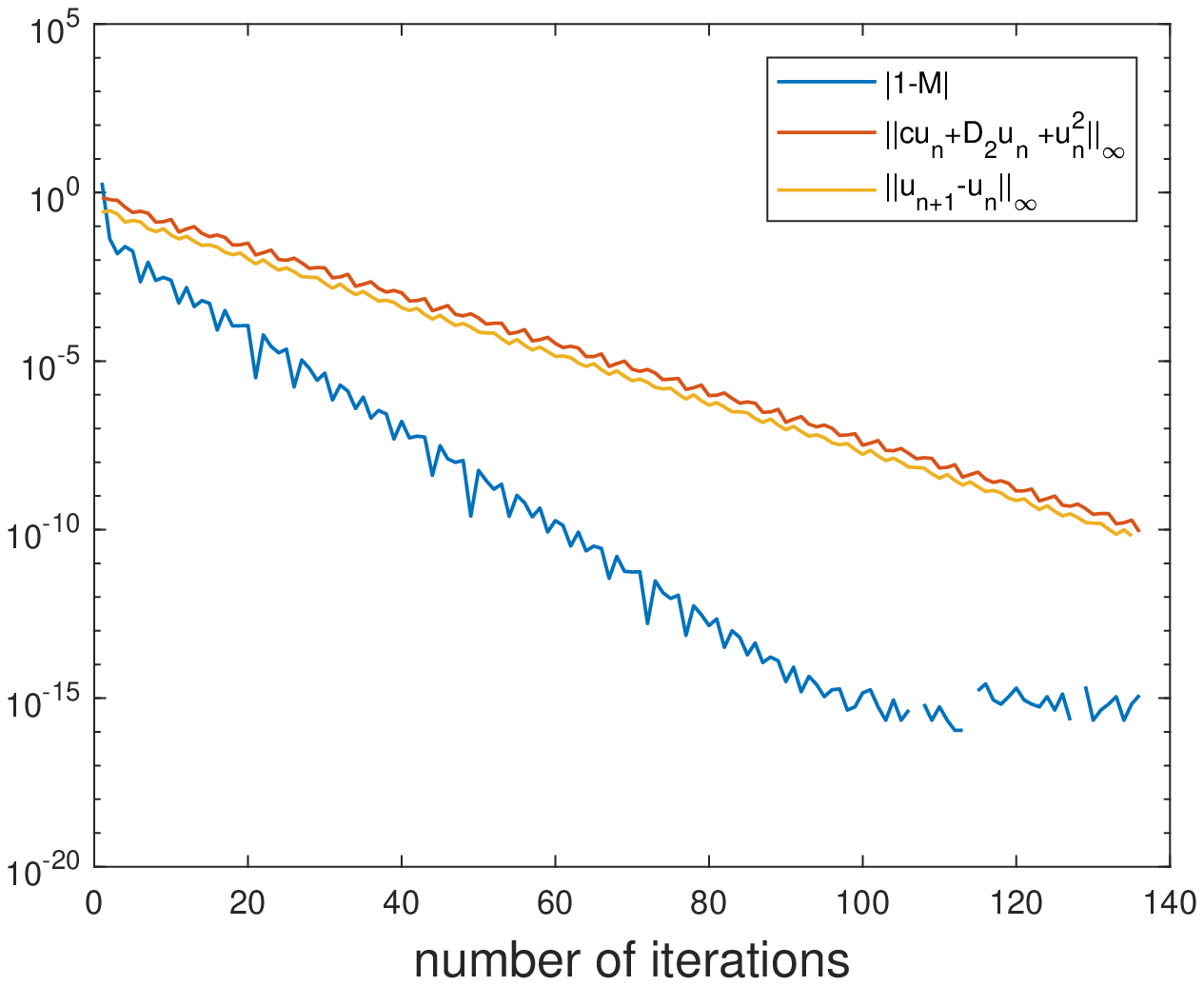}
	\end{subfigure}
		\caption{Iterations for $c = 2$ and $\alpha = 2$. (a) The last iteration versus $x$. (b) Computational errors versus $n$. }
		\label{fig:Kdv-c2}
\end{figure}
\begin{figure}[h]
	\begin{subfigure}[t]{0.5\textwidth}
		\centering
		\includegraphics[width=0.8\linewidth]{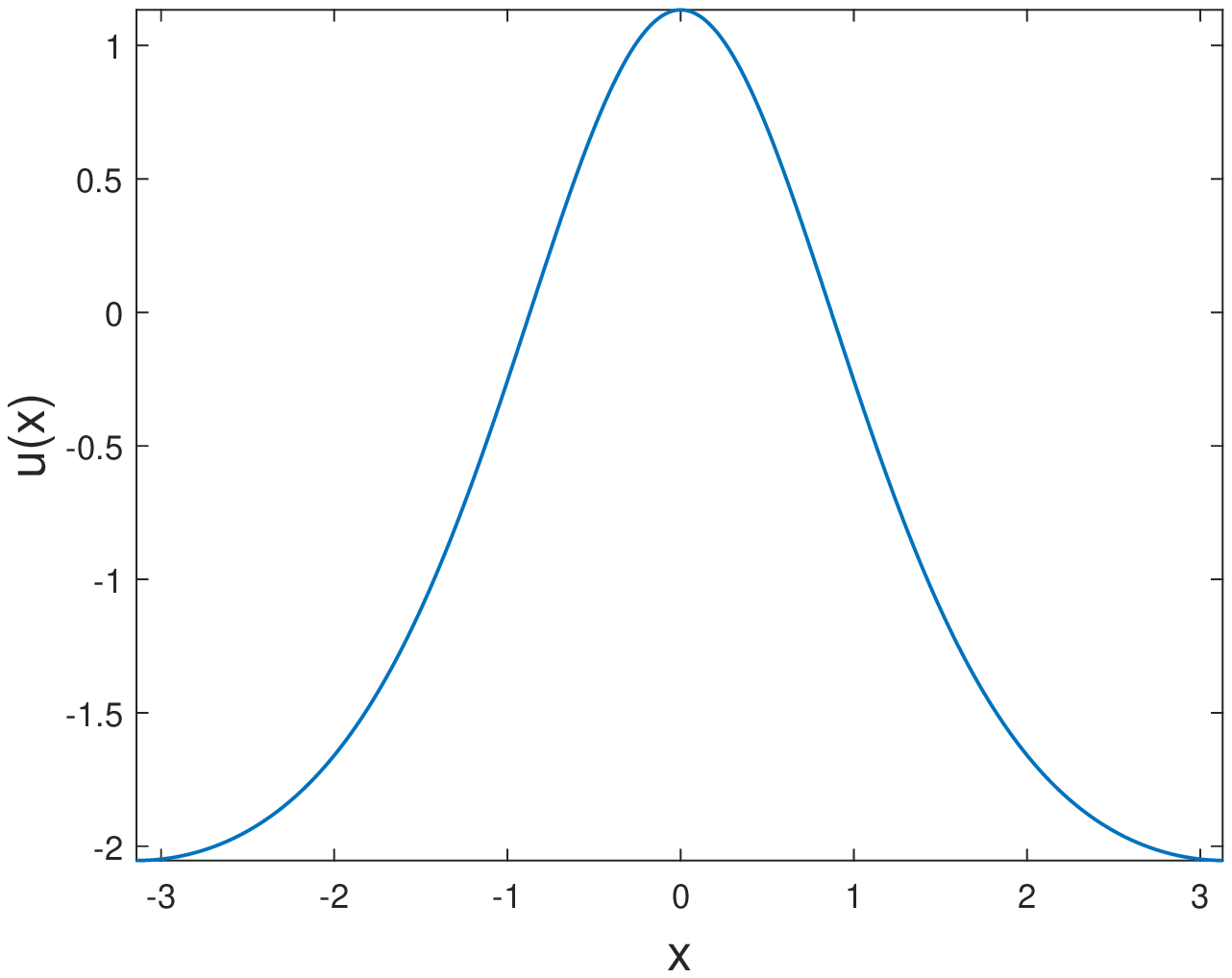}
	\end{subfigure}%
	\begin{subfigure}[t]{0.5\textwidth}
		\centering
		\includegraphics[width=0.8\linewidth]{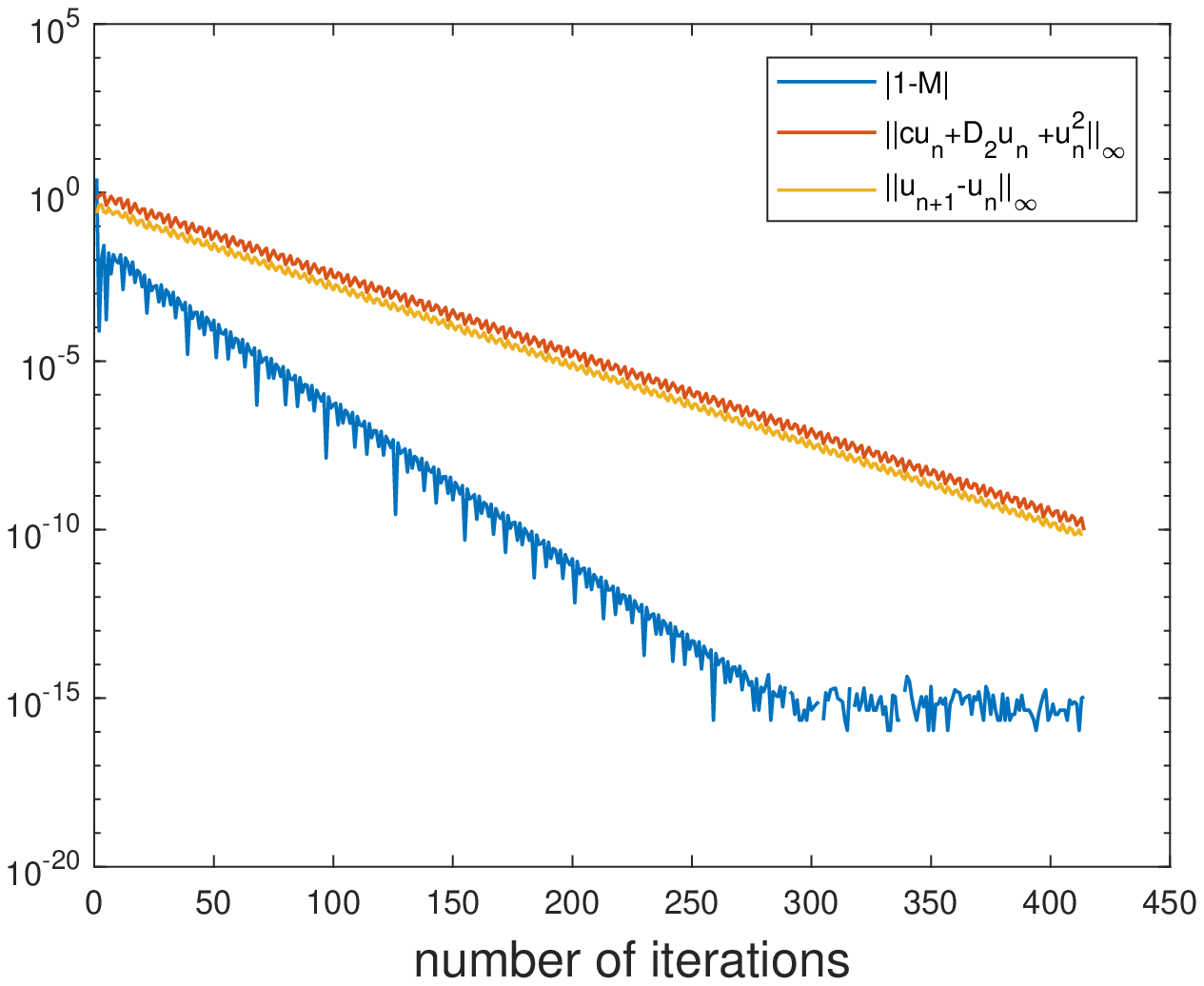}
	\end{subfigure}
\caption{Iterations for $c = 2.3$ and $\alpha = 2$. (a) The last iteration versus $x$. (b) Computational errors versus $n$. }
		\label{fig:Kdv-c23}
\end{figure}
\begin{figure}[h]
	\centering
	\begin{subfigure}[t]{0.5\textwidth}
		\centering
		\includegraphics[width=0.7\linewidth]{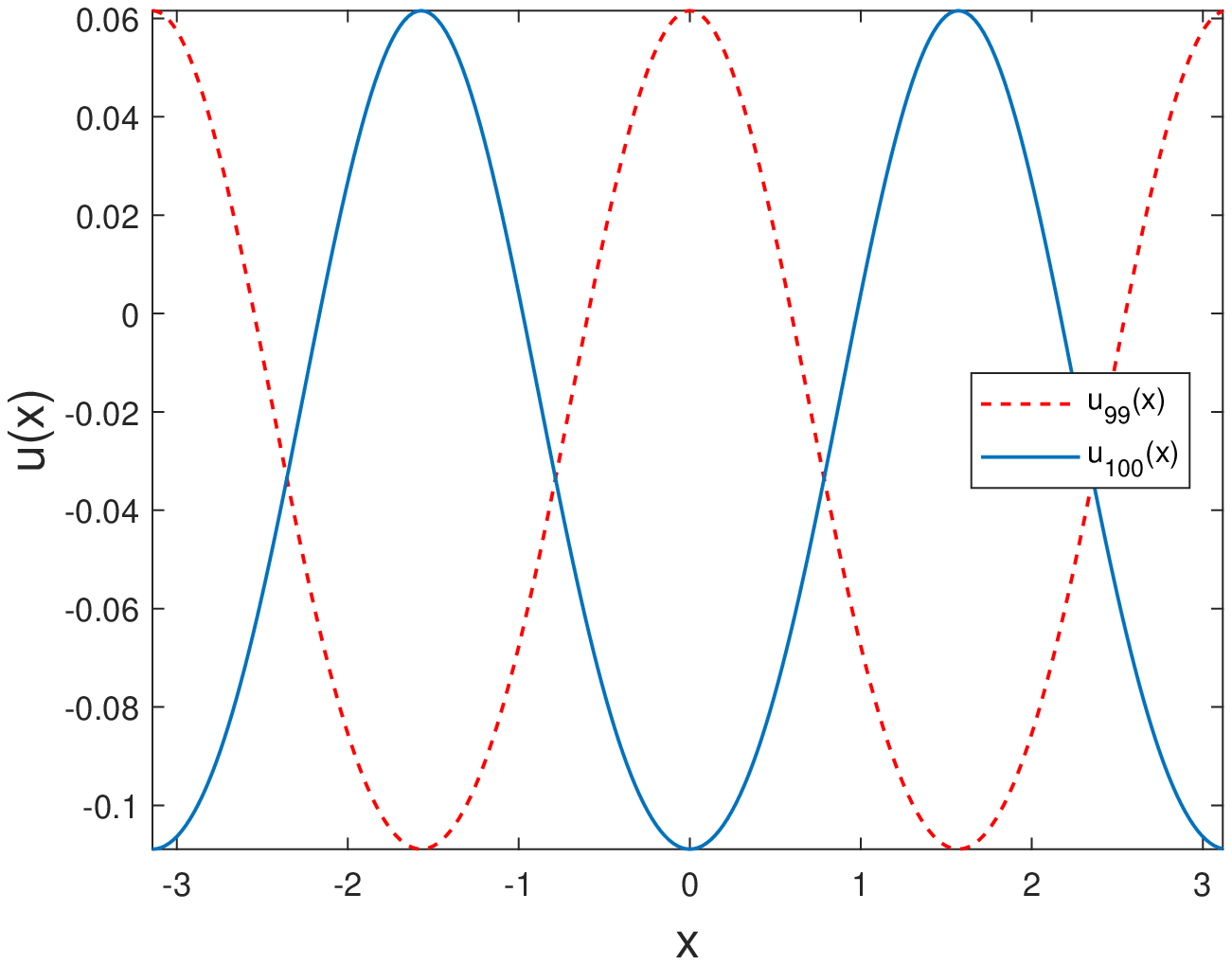}
	\end{subfigure}%
	\begin{subfigure}[t]{0.5\textwidth}
		\centering
		\includegraphics[width=0.7\linewidth]{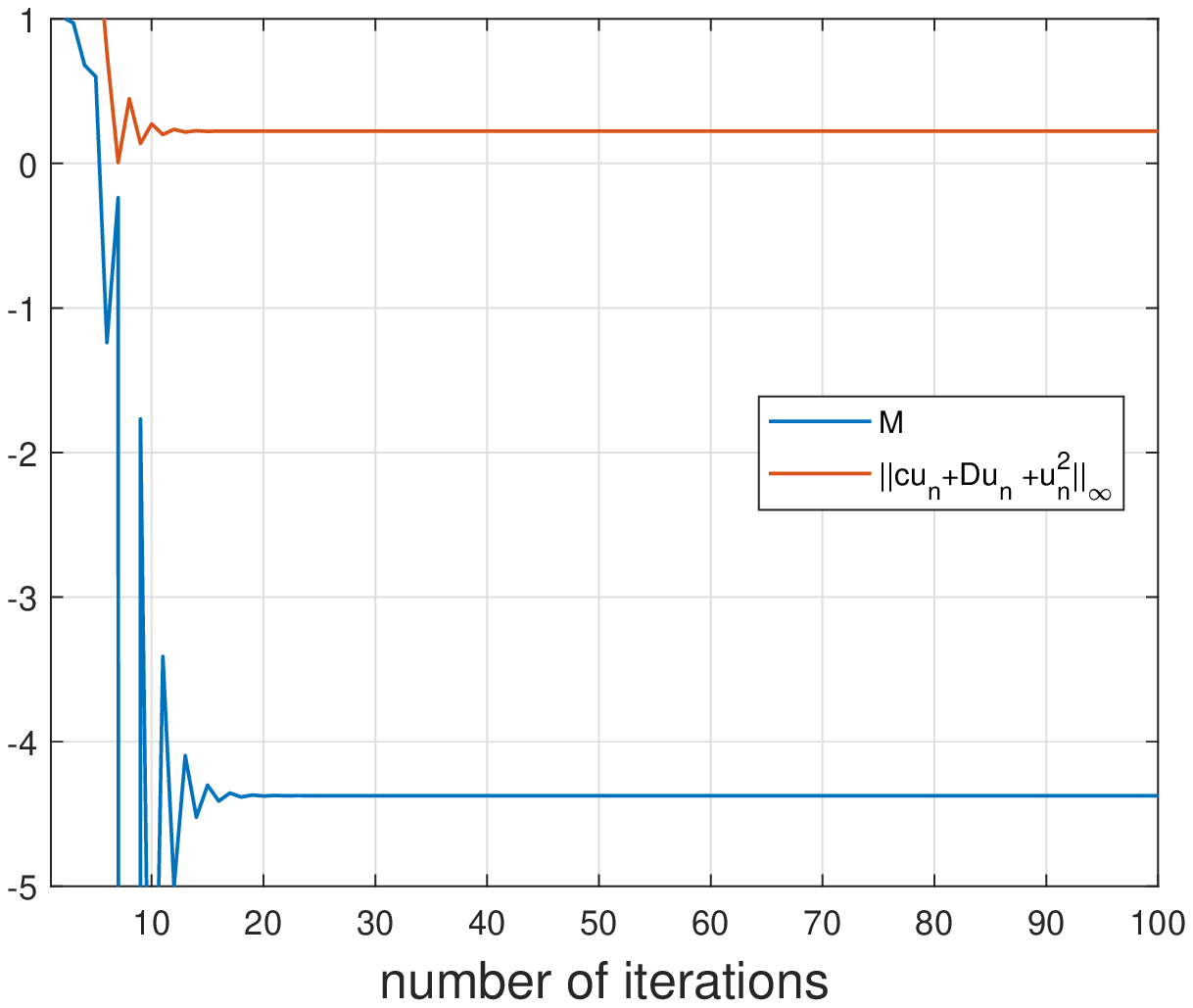}
	\end{subfigure}
\caption{Iterations for $c = 3$ and $\alpha = 2$. (a) The last two iterations versus $x$.
(b) Computational errors versus $n$. }	\label{fig:Kdv-c3}
\end{figure}

Figure \ref{fig:Kdv-c23} illustrates the case $c = 2.3$. Since
the largest eigenvalue of $\mathcal{L}^{-1}_{c,\alpha}\mathcal{H}_{c,\alpha}$ crosses the level $2$ at this value of $c$,
see Figure \ref{fig:Eigenvalues_against_c1_37}, this case is marginal for convergence of iterations. As we can see from Figure \ref{fig:Kdv-c23},
iterations still converge to a single-lobe periodic wave but the convergence is slow.

Figure \ref{fig:Kdv-c3} illustrates the case $c = 3$. The iterative method (\ref{T}) diverges from the single-lobe periodic wave.
The instability is related to the eigenvalue of $\mathcal{L}_T$ which is smaller than $-1$, hence
the period-doubling instability leads to an alternating sequence which oscillates between two double-lobe profile
shown on the left panel. The right panel shows that the factor $M$ no longer converges to 1 but to -4.3737 and the residual
errors does not converge to 0 but remains strictly positive with the number of iterations.
Therefore, the two limiting states of the iterative method (\ref{T}) in the $2$-periodic orbit are not
a periodic wave of the boundary-value problem (\ref{ode}).

For the BO equation with $\alpha = 1$, we show that the method diverges for
$c \gtrsim 1$ in agreement with Corollary \ref{corollary-1}.
Figure \ref{fig:Eigenvalues_against_c1_2_BO} shows the eigenvalues
of the generalized eigenvalue problem (\ref{gEp}) for $\alpha=1$.
The eigenvalue $\lambda_1 = \frac{2^{\alpha+1}-5}{2^{\alpha+1}-3}$ in Lemma \ref{lemma-3}
yields $\lambda_1 = -1$ for $\alpha = 1$ in addition to the other eigenvalue $-1$
in $\{-1,0,\lambda_1,\lambda_2\}$. Hence, $\lambda=-1$ is a double eigenvalue
and the left panel shows that this double eigenvalue
is preserved in $c$.
The right panel zooms in eigenvalues near $c = 1$ and shows the asymptotic approximation
of the eigenvalue near $2$ given by (\ref{Lambda1}) and (\ref{Lambda2}) with $\alpha = 1$.

\begin{figure}[h]
\centering
\includegraphics[width=0.45\linewidth]{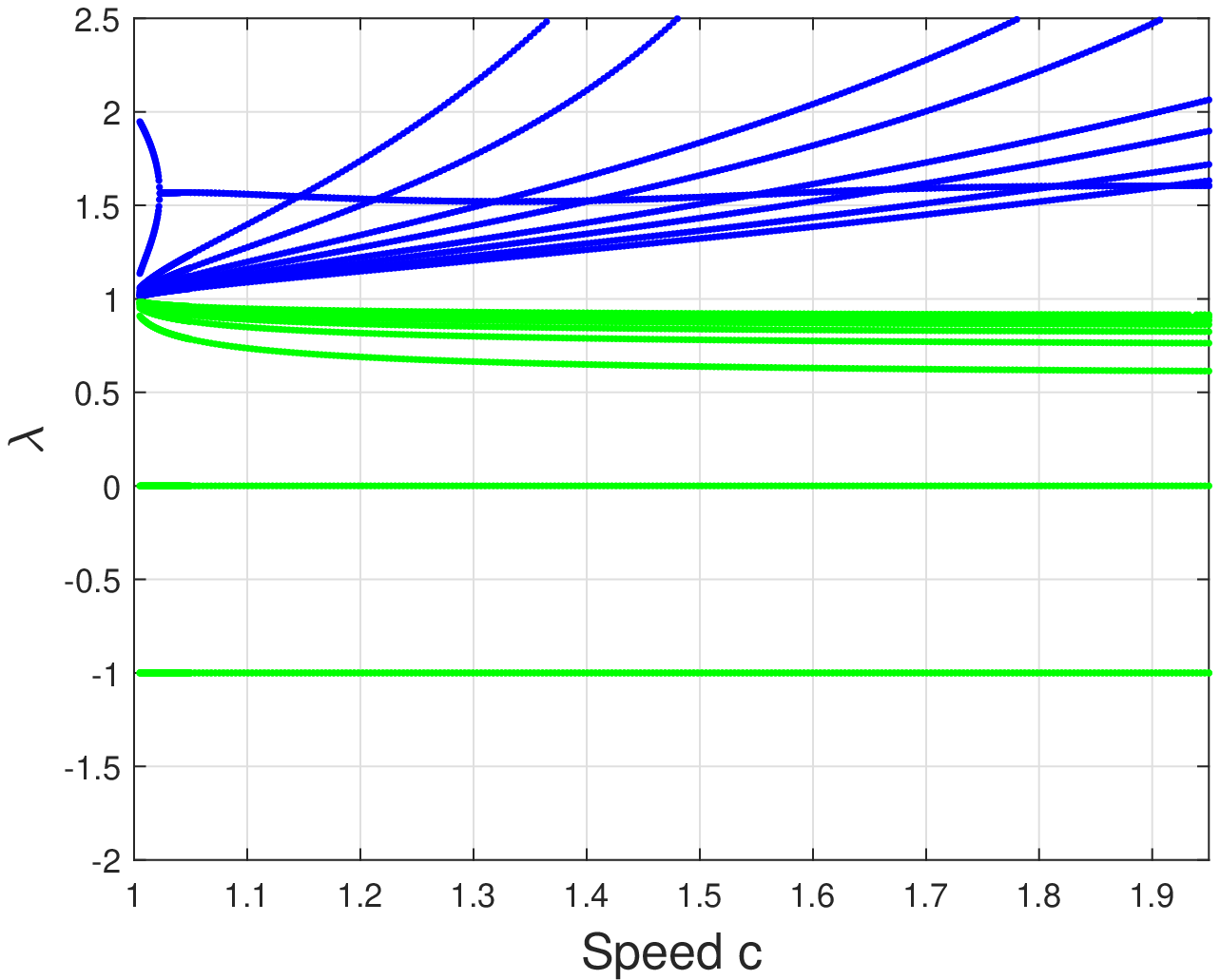}
\includegraphics[width=0.45\linewidth]{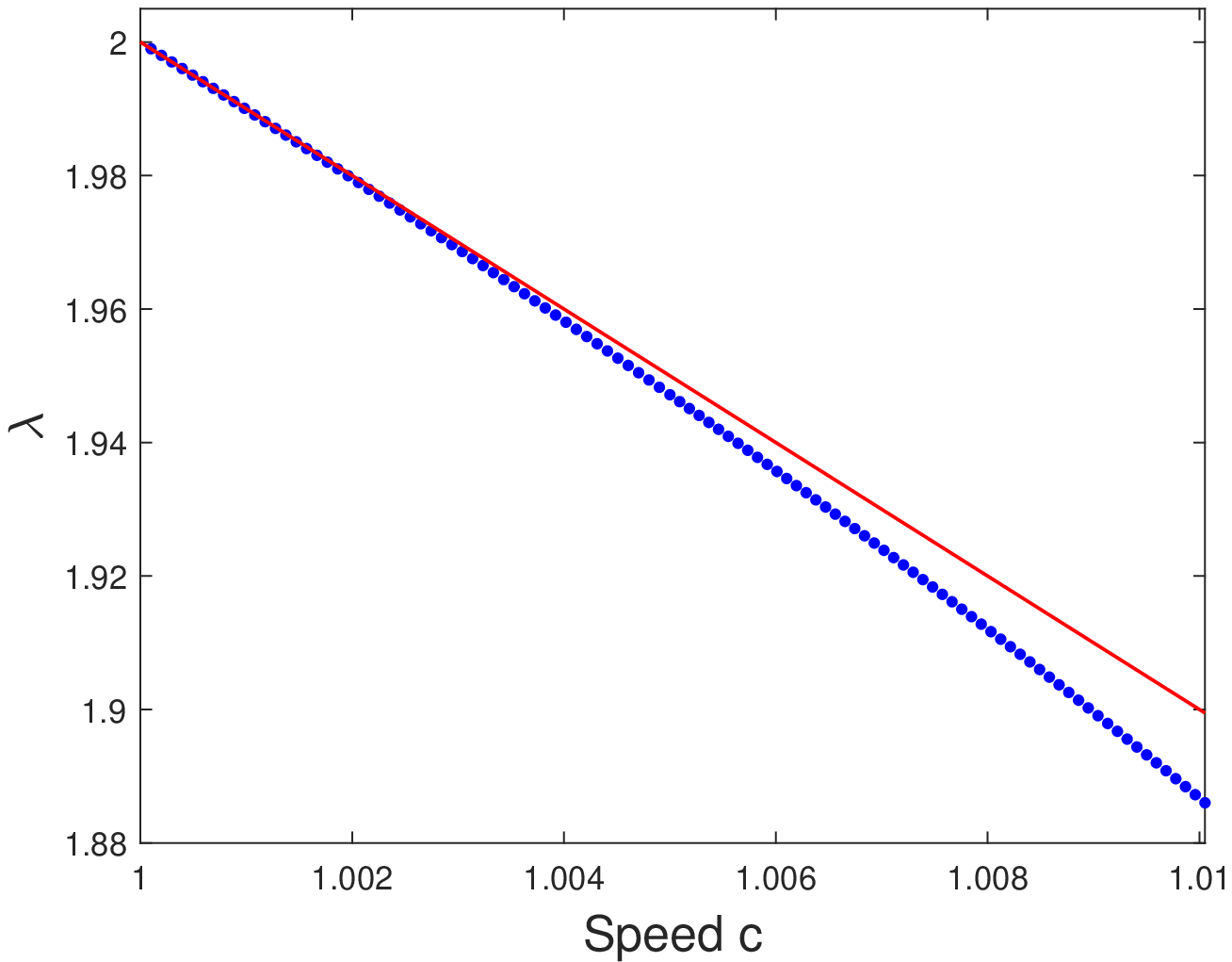}
\caption{Left: Eigenvalues of the operator $\mathcal{L}_{c,\alpha}^{-1}\mathcal{H}_{c,\alpha}$ for $\alpha = 1$. Right: Zoom in with the asymptotic dependence given by (\ref{Lambda1}) and (\ref{Lambda2}).}
\label{fig:Eigenvalues_against_c1_2_BO}
\end{figure}

\begin{figure}[h]
	\centering
	\begin{subfigure}[t]{0.5\textwidth}
		\centering
		\includegraphics[width=0.65\linewidth]{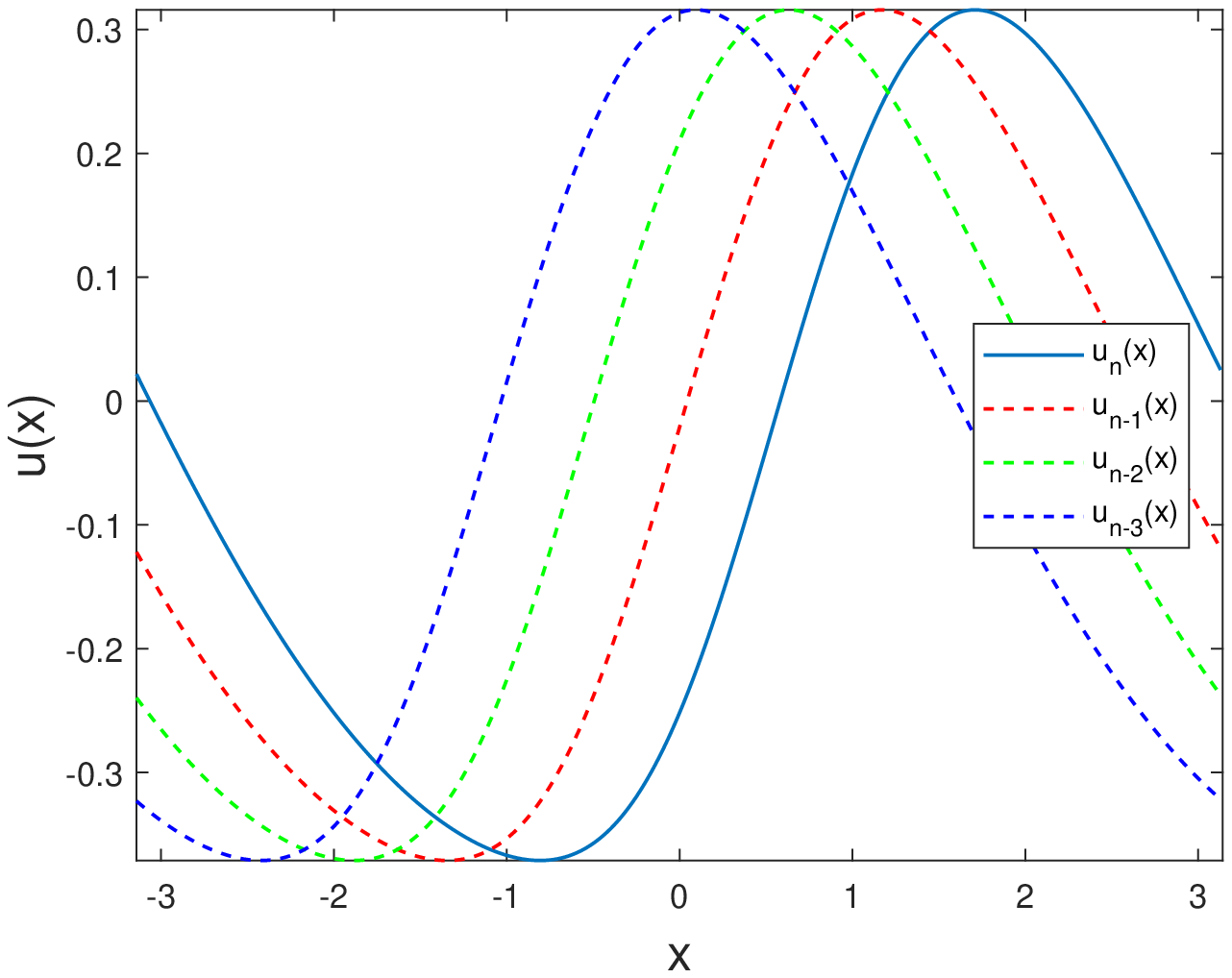}
	\end{subfigure}%
	\begin{subfigure}[t]{0.5\textwidth}
		\centering
		\includegraphics[width=0.65\linewidth]{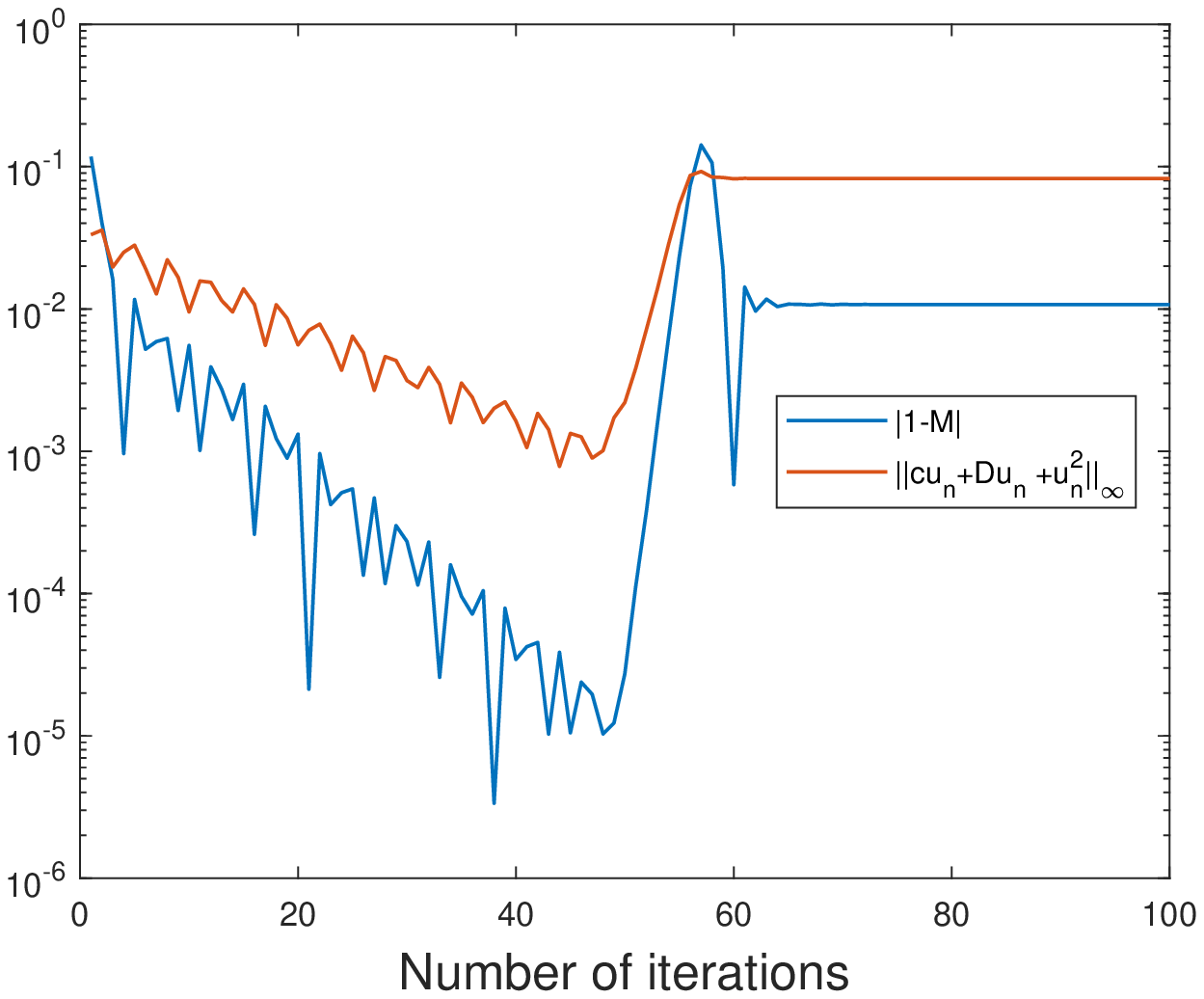}
	\end{subfigure}
\caption{Iterations for $c = 1.1$ and $\alpha = 1$. (a) The last four iterations versus $x$.
(b) Computational errors versus $n$. }	
\label{fig:BO-c11-2ndorder}
\end{figure}
\begin{figure}[h]
	\centering
	\begin{subfigure}[t]{0.5\textwidth}
		\centering
		\includegraphics[width=0.65\linewidth]{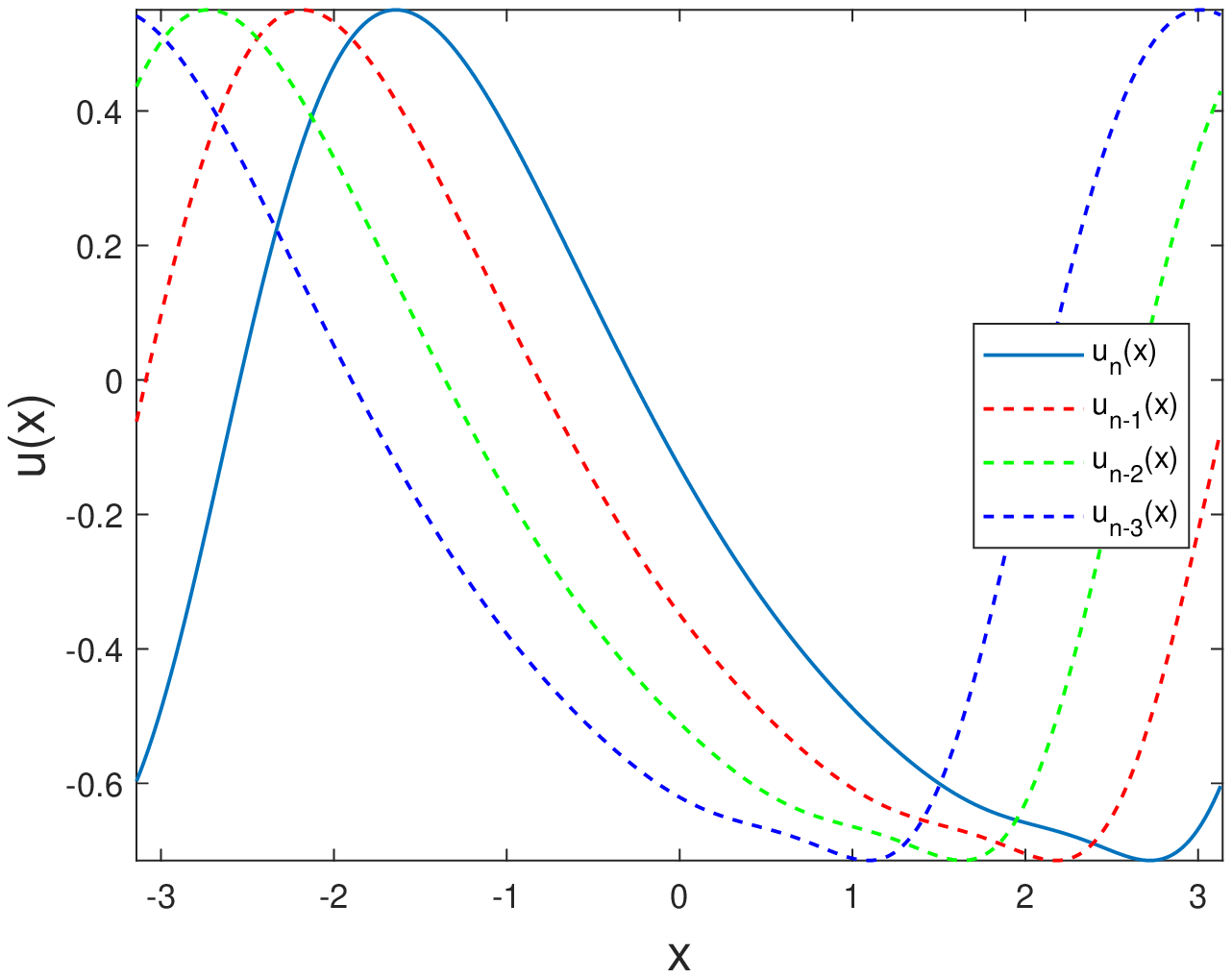}
	\end{subfigure}%
	\begin{subfigure}[t]{0.5\textwidth}
		\centering
		\includegraphics[width=0.65\linewidth]{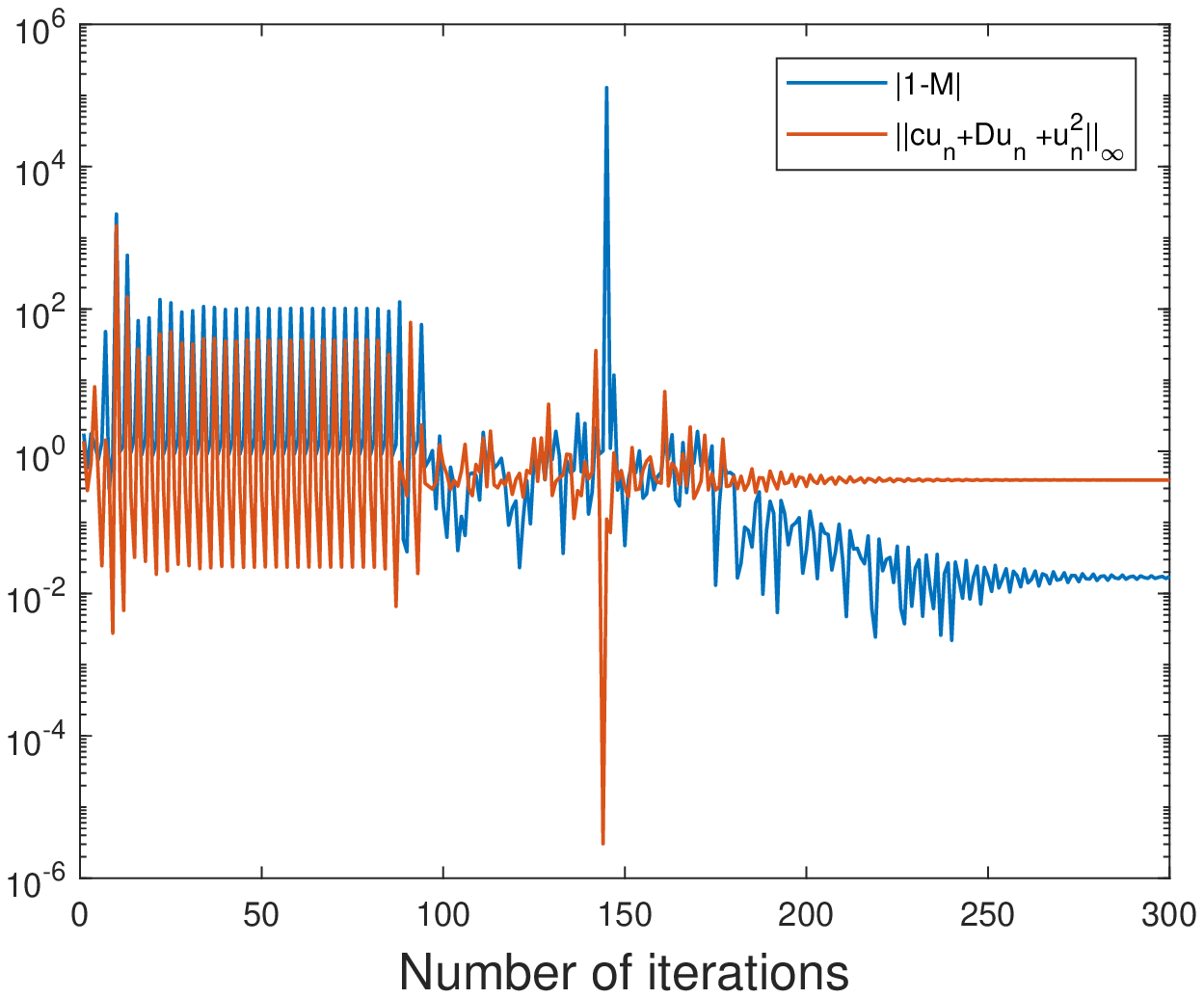}
	\end{subfigure}
\caption{Iterations for $c = 1.3$ and $\alpha = 1$. (a) The last four iterations versus $x$.
(b) Computational errors versus $n$. }	\label{fig:BO-c13}
\end{figure}
\begin{figure}[h]
	\centering
	\begin{subfigure}[t]{0.5\textwidth}
		\centering
		\includegraphics[width=0.65\linewidth]{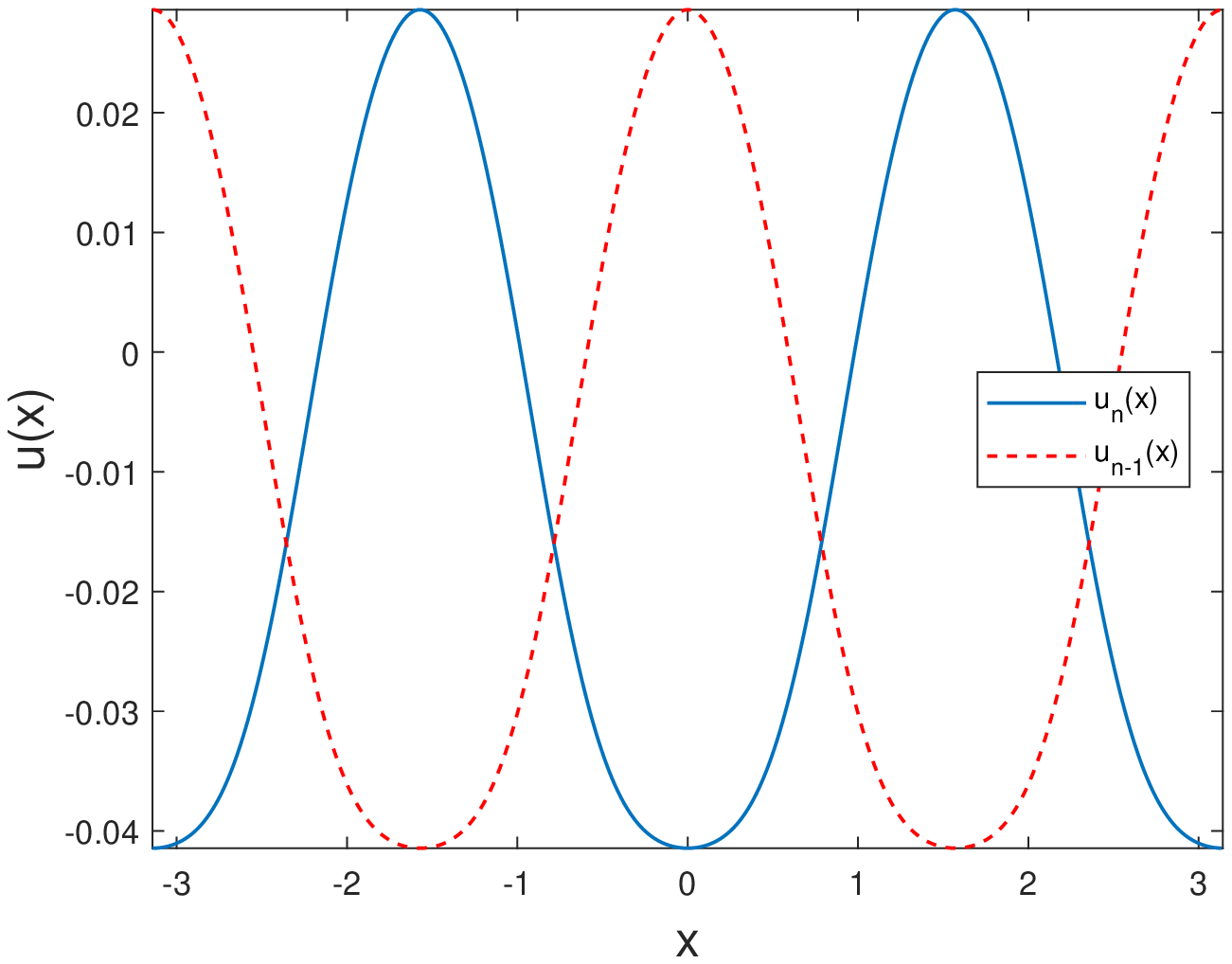}
	\end{subfigure}%
	\begin{subfigure}[t]{0.5\textwidth}
		\centering
		\includegraphics[width=0.65\linewidth]{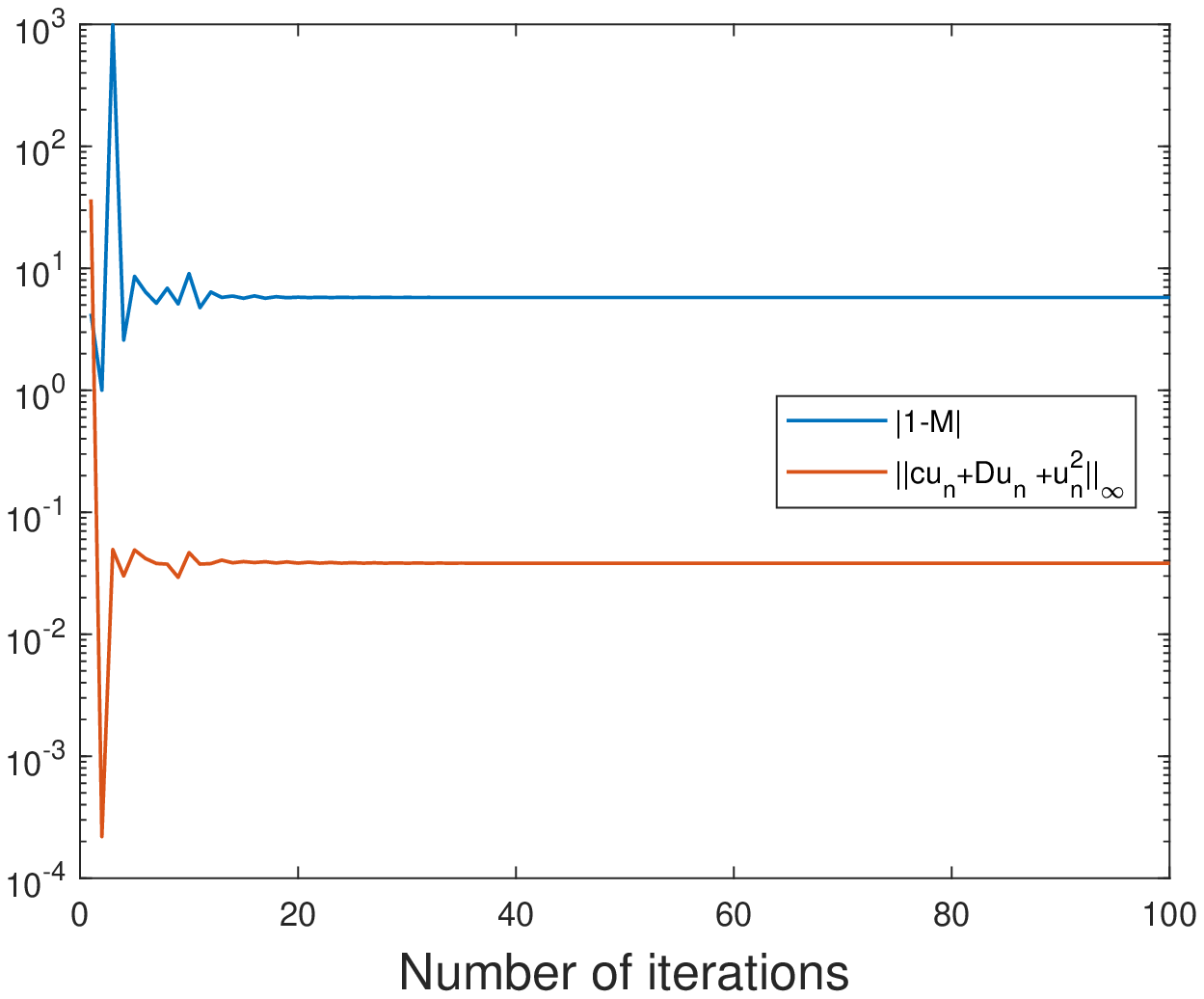}
	\end{subfigure}
\caption{Iterations for $c = 1.6$ and $\alpha = 1$. (a) The last two iterations versus $x$.
(b) Computational errors versus $n$. }			\label{fig:BO-c153}
\end{figure}

To illustrate convergence of the iterative method (\ref{T}) for $\alpha =1$, we
use the initial function
\begin{equation}
\label{initial-guess-BO}
u_0(x)= a \cos(x)+ \frac{1}{2} a^2 \left(\cos(2x)-1\right) + \varepsilon \sin(x),
\end{equation}
where $a > 0$ and $\varepsilon \in \mathbb{R}$.
Again, we have verified that $\int_{-\pi}^{\pi} \phi^3 dx < 0$ for every
$c > 1$ and $\alpha = 1$ by the explicit computations in (\ref{wave-negative-bo}),
therefore, we included the second term of the Stokes expansion (\ref{wave-assumption}) in
the initial function (\ref{initial-guess-BO}). In computations below, we take $a = 0.4$.

As predicted by Corollary \ref{corollary-1} for $\alpha = 1$,
the iterative method (\ref{T}) diverges for the BO equation and
this divergence for $c \gtrsim 1$ is due to an odd eigenfunction
of the generalized eigenvalue problem (\ref{gEp}) for the eigenvalue $\lambda_1 = -1$.

Figure \ref{fig:BO-c11-2ndorder} illustrates the case $c=1.1$ showing
the last four iterations in the left panel and the factor $M$ converging to $1.0107$
and the residual error converges to $0.0826$ in the right panel.
In this computation, we take $\varepsilon = 0$.
Although the residual error starts to decrease initially due to contracting
properties of $\mathcal{L}_T$ on the even subspace of $L^2_{\rm per}(-\pi,\pi)$,
round-off errors induce odd perturbations which result in slow instability.
As a result, the periodic wave of amplitude $0.458$ is not captured by
the iterative method (\ref{T}), instead iterations converge to the periodic profile of amplitude $0.344$
which is drifted by every iteration to the right. This drifted periodic profile
of the iterative method (\ref{T}) is not a solution
to the boundary-value problem (\ref{ode}). If $\varepsilon \neq 0$,
the instability develops much faster and the drifted periodic profile is drifted to the right
if $\varepsilon > 0$ and to the left if $\varepsilon < 0$.

Figure \ref{fig:BO-c13} shows the marginal case $c=1.3$ where another unstable eigenvalue of $\mathcal{L}_T$
related to the even eigenfunction crosses the level $-1$. Although the instability pattern of Figure \ref{fig:BO-c11-2ndorder}
is repeated on Figure \ref{fig:BO-c13}, the periodic profile becomes more complicated and
the instability process is accompanied by many intermediate oscillations.
Here again we set $\varepsilon = 0$, if $\varepsilon \neq 0$, the drifted periodic profile
is formed much faster and intermediate oscillations are reduced.

Figure \ref{fig:BO-c153} illustrates the case $c=1.6$ when several eigenvalues
of $\mathcal{L}_T$ are located below $-1$.
After short intermediate iterations, the iterative method starts to
oscillate between two iterations, similarly to the pattern of Figure \ref{fig:Kdv-c3}.
The right panel of Figure \ref{fig:BO-c153} shows that the factor $M$ converges to $-5.1447$
and the residual error remains strictly positive. The two limiting states of the iterative method (\ref{T})
in the $2$-periodic orbit are not a periodic wave of the boundary-value problem (\ref{ode}).

\section{Proof of Theorem \ref{theorem-main-2}}
\label{sec-4}

By linearizing $\tilde{T}_{c,\alpha}$ at $\psi$ with $w_n = \psi + a_n \psi + b_n \psi' + \beta_n$,
where $\beta_n \in H^{\alpha}_{\rm per}(-\pi,\pi) \cap L^2_c$ satisfies the two constraints in
\begin{equation}
\label{constrained-space-new}
L^2_c := \left\{ \omega \in L^2_{\rm per}(-\pi,\pi) : \quad \langle \psi^2, \omega \rangle = \langle \psi \psi', \omega \rangle = 0 \right\},
\end{equation}
we obtain the linearized iterative rule:
\begin{equation}
\label{lin-T-new}
a_{n+1} = 0, \quad b_{n+1} = b_n, \quad \beta_{n+1} = \tilde{\mathcal{L}}_T \beta_n,
\end{equation}
where
\begin{equation}
\label{operator-L-T-new}
\tilde{\mathcal{L}}_T := \tilde{\mathcal{L}}_{c,\alpha}^{-1} (2 \psi \cdot) =
{\rm Id} - \tilde{\mathcal{L}}_{c,\alpha}^{-1} \tilde{\mathcal{H}}_{c,\alpha}  : \quad
H^{\alpha}_{\rm per}(-\pi,\pi) \cap L^2_c \mapsto H^{\alpha}_{\rm per}(-\pi,\pi) \cap L^2_c
\end{equation}
with $\tilde{\mathcal{L}}_{c,\alpha} = c - D_{\alpha}$ and
$\tilde{\mathcal{H}}_{c,\alpha} = \mathcal{H}_{c,\alpha}$ given by (\ref{Jacobian-new}).
Hence Lemmas \ref{lemma-2} and \ref{lemma-11} apply to $\tilde{\mathcal{H}}_{c,\alpha} = \mathcal{H}_{c,\alpha}$.
In addition, Theorem \ref{lemma-positive} ensures positivity of $\psi(x) > 0$ for every $x \in [-\pi,\pi]$.

The following lemma characterizes the spectrum of $\tilde{\mathcal{L}}_{c,\alpha}^{-1} \tilde{\mathcal{H}}_{c,\alpha}$
in $L^2_{\rm per}(-\pi,\pi)$ for every $c > 1$ and $\alpha \in (\alpha_0,2]$. Convergence of
the iterative method (\ref{T-new}) to the positive periodic wave $\psi \in H^{\alpha}_{\rm per}(-\pi,\pi)$
of the boundary-value problem (\ref{ode-new}) follows from this lemma.
This construction yields the proof of Theorem \ref{theorem-main-2}.

\begin{lemma}
\label{lemma-modified}
For every $c > 1$ and $\alpha \in (\alpha_0,2]$, $\sigma(\tilde{\mathcal{L}}_{c,\alpha}^{-1} \tilde{\mathcal{H}}_{c,\alpha}) \in (0,1)$
in $L^2_c$.
\end{lemma}

\begin{proof}
We note that $\tilde{\mathcal{L}}_{c,\alpha}$ is positive for every $c > 1$ and $\alpha > 0$,
whereas $\tilde{\mathcal{H}}_{c,\alpha}$ has one simple negative eigenvalue and a simple zero eigenvalue 
for every $c > 1$ and $\alpha \in (\alpha_0,2]$ by Lemma \ref{lemma-11}. 

By Theorem 1 in \cite{ChugPel}, $\sigma(\tilde{\mathcal{L}}_{c,\alpha}^{-1} \tilde{\mathcal{H}}_{c,\alpha})$
in $L^2_{\rm per}(-\pi,\pi)$ is real and contains one simple negative eigenvalue and a simple zero eigenvalue, the rest
of the spectrum is positive and bounded away from zero.
The negative and zero eigenvalues correspond to the exact solutions:
\begin{equation}
\label{eigenvalues-exact-new}
\tilde{\mathcal{L}}_{c,\alpha}^{-1} \tilde{\mathcal{H}}_{c,\alpha} \psi = - \psi \quad \mbox{\rm and} \quad
\tilde{\mathcal{L}}_{c,\alpha}^{-1} \tilde{\mathcal{H}}_{c,\alpha} \psi' = 0.
\end{equation}
These eigenvalues are removed by adding two constraints in the definition of $L^2_c$ in (\ref{constrained-space-new}).
The positive eigenvalues are bounded from above by $1$ because the operator
$$
\tilde{\mathcal{L}}_{\tilde{T}} = \tilde{\mathcal{L}}_{c,\alpha}^{-1} (2 \psi \cdot) = Id - \tilde{\mathcal{L}}_{c,\alpha}^{-1} \tilde{\mathcal{H}}_{c,\alpha}
$$
is strictly positive due to positivity of $\tilde{\mathcal{L}}_{c,\alpha}$ and $\psi$.
Hence, $\sigma(\tilde{\mathcal{L}}_{c,\alpha}^{-1} \tilde{\mathcal{H}}_{c,\alpha}) \in (0,1)$ in $L^2_c$.
\end{proof}

\begin{corollary}
\label{corollary-modified}
For every $c > 1$ and $\alpha \in (\alpha_0,2]$,
the iterative method (\ref{T-new}) converges to $\psi$ in $H^{\alpha}_{\rm per}(-\pi,\pi)$.
\end{corollary}

\begin{proof}
Conditions $\int_{-\pi}^{\pi} \psi^3 dx > 0$ and $\int_{-\pi}^{\pi} \psi (\psi')^2 dx > 0$
follow by positivity of $\psi$ in Theorem \ref{lemma-positive}.
By Lemma \ref{lemma-modified}, the operator $\tilde{\mathcal{L}}_{\tilde{T}}$ is a strict contraction in $L^2_c$ 
for every $c > 1$ and $\alpha \in (\alpha_0,2]$. 
Convergence of the iterative method (\ref{T-new}) follows by Theorem \ref{theorem-convergence-classic}.
\end{proof}

\begin{figure}[h]
	\centering
	\begin{subfigure}[t]{0.5\textwidth}
		\centering
		\includegraphics[width=0.8\linewidth]{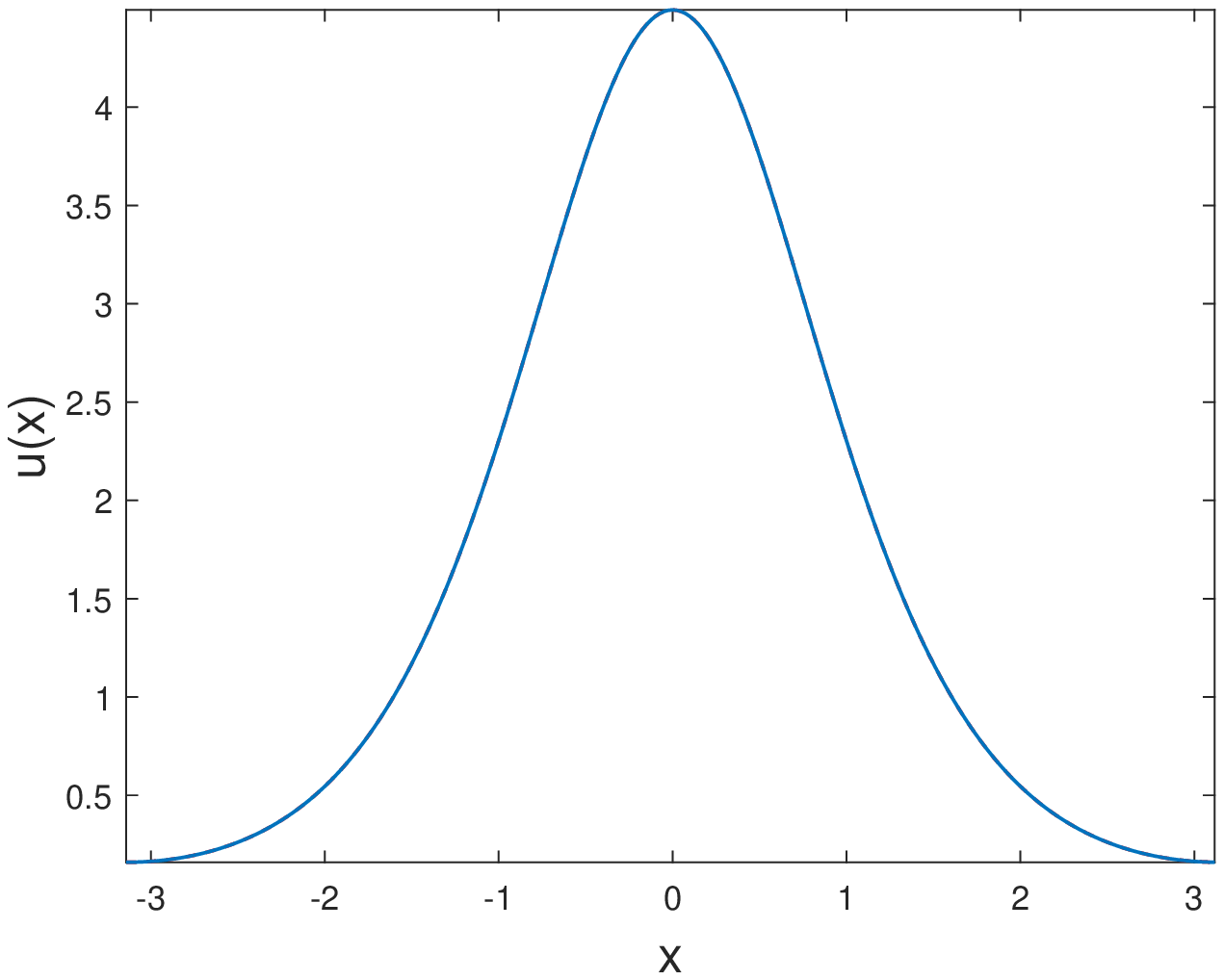}
	\end{subfigure}%
	\begin{subfigure}[t]{0.5\textwidth}
		\centering
		\includegraphics[width=0.8\linewidth]{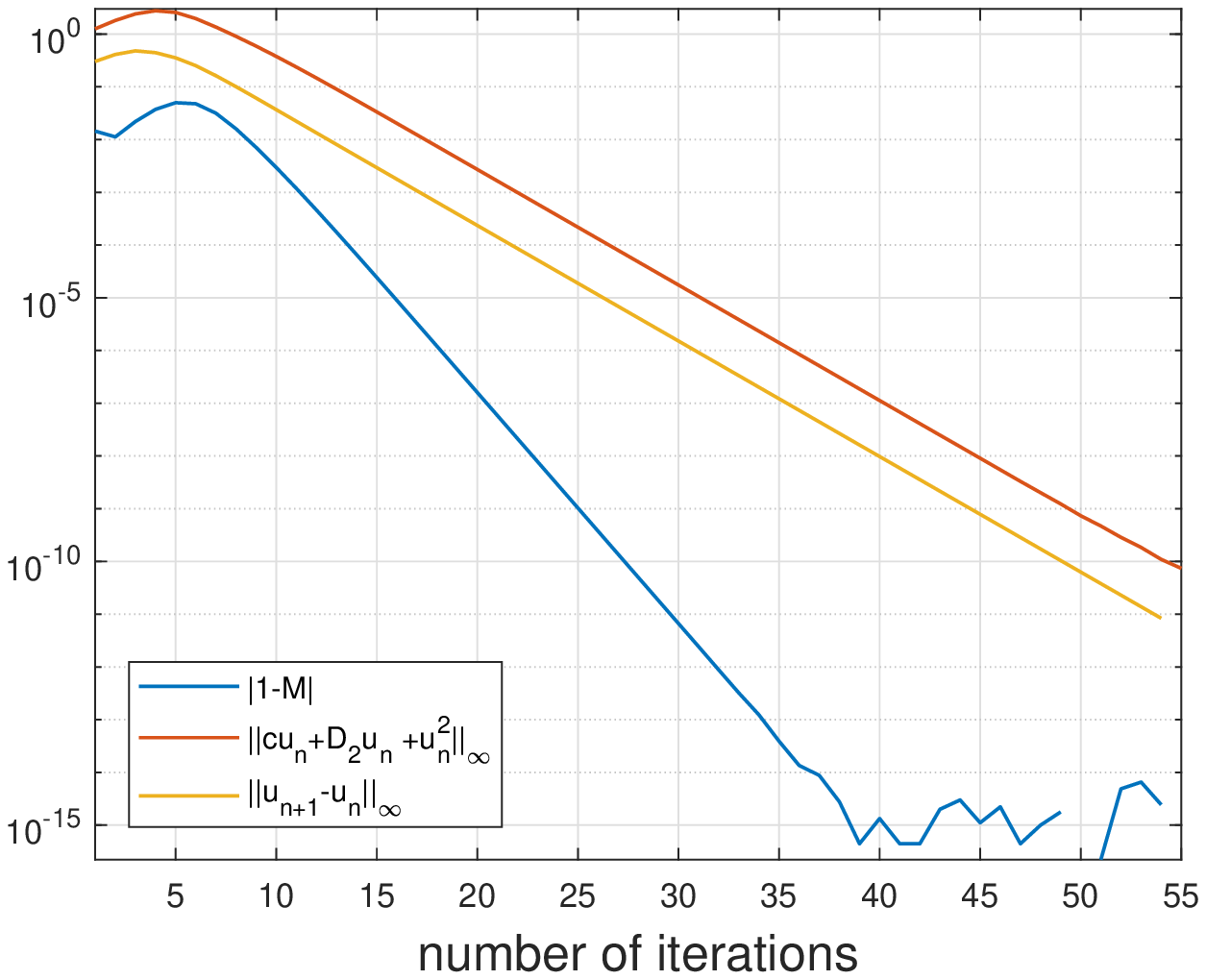}
	\end{subfigure}
		\caption{Iterations for $c = 3$ and $\alpha = 2$. (a) The last iteration versus $x$. (b) Computational errors versus $n$. }
\label{fig:Kdv-shift}
\end{figure}

\begin{figure}[h]
	\centering
	\begin{subfigure}[t]{0.5\textwidth}
		\centering
		\includegraphics[width=0.8\linewidth]{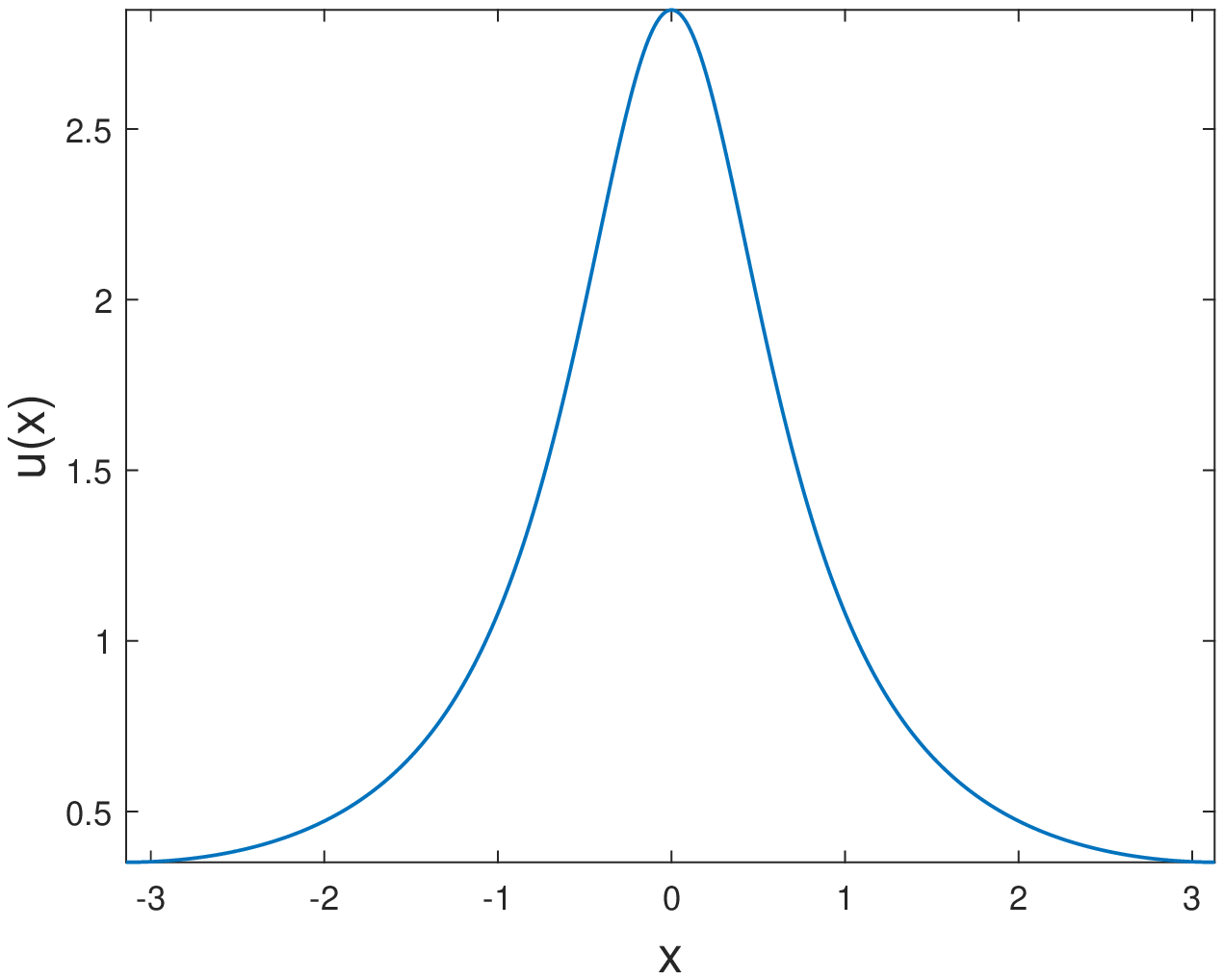}
	\end{subfigure}%
	\begin{subfigure}[t]{0.5\textwidth}
		\centering
		\includegraphics[width=0.8\linewidth]{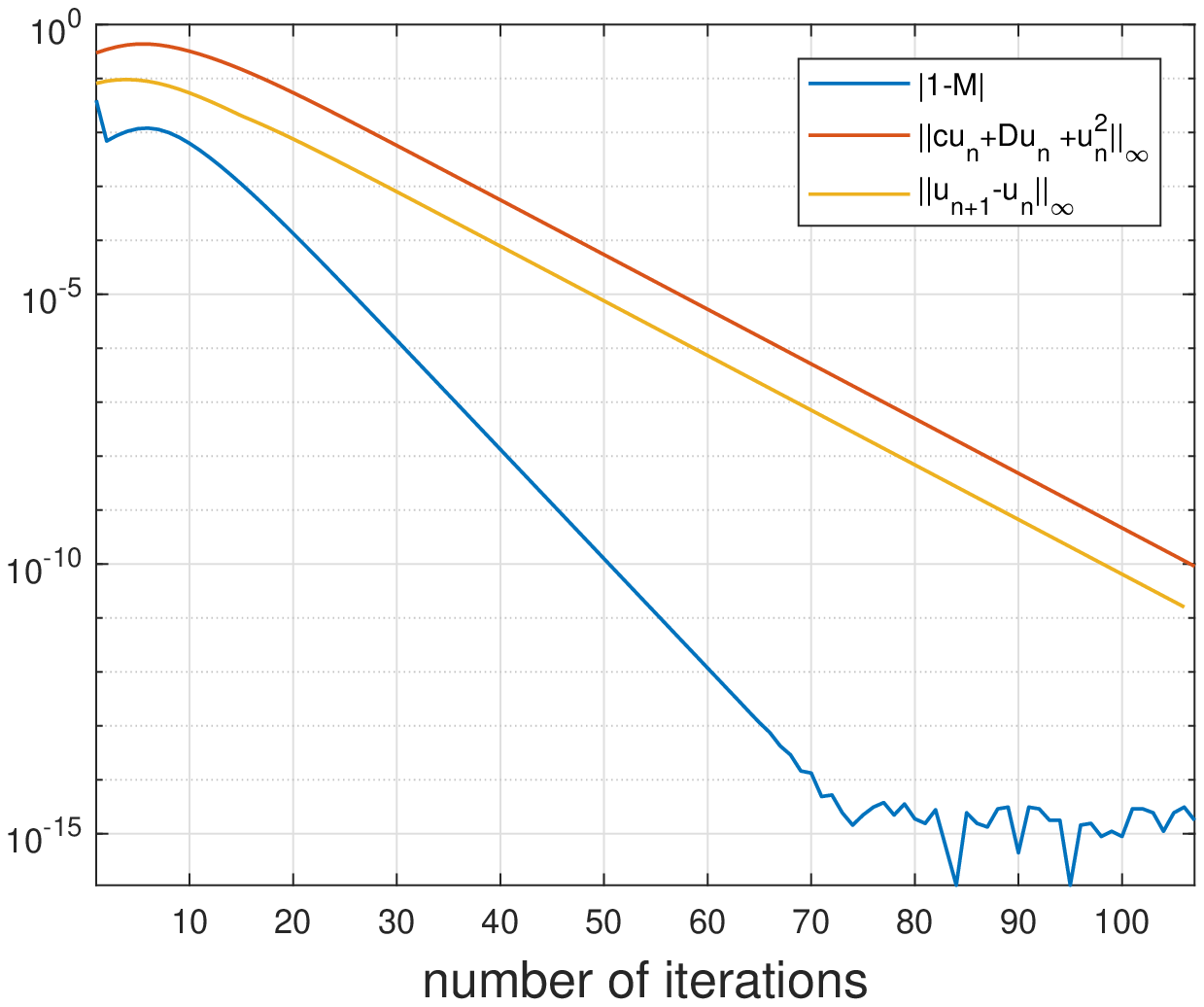}
	\end{subfigure}
		\caption{Iterations for $c = 1.6$ and $\alpha = 1$. (a) The last iteration versus $x$. (b) Computational errors versus $n$. 	}
	\label{fig:BO-shift}
\end{figure}

To demonstrate the convergence of the iterative method (\ref{T-new}), we use the initial condition
	\begin{equation*}
		u_{0}(x) = c + a\cos(x)
	\end{equation*}
with $a = 0.4$. This initial guess corresponds
to the first two terms of the Stokes expansion (\ref{wave-expansion}) for $\psi(x) = c + \phi(x)$.
We do not need to include the $\mathcal{O}(a^2)$ to the initial guess because
$\int_{-\pi}^{\pi} u_0^3 dx > 0$ and the denominator of the Petviashvili quotient (\ref{rule-new})
does not vanish at $u_0$.

Figure \ref{fig:Kdv-shift} shows the result of iterations for $c = 3$ and $\alpha = 2$.
It is seen that iterations converge quickly to a positive, single-lobe periodic wave
$\psi$ in agreement with Corollary \ref{corollary-modified}. Note that the iterative method (\ref{T})
diverges for $c = 3$ and $\alpha = 2$, as is seen from Figure \ref{fig:Kdv-c3}.
We can also compute the distance between the last iteration and the exact solution,
in which case we find $\|u-\phi\|_{L^{\infty}} \approx 1.3 \cdot 10^{-11}$.

Figure \ref{fig:BO-shift} reports similar results for $c = 1.6$ and $\alpha = 1$.
Again, the iterative method (\ref{T}) diverges for these values of $c$ and $\alpha$, as is
seen from Figure \ref{fig:BO-c153}.  We can also compute the distance between the last iteration and the exact solution,
in which case we find $\|u-\phi\|_{L^{\infty}} \approx 5.9 \cdot 10^{-11}$.

\vspace{0.25cm}

{\bf Acknowledgements:} DEP thanks D. Clamond for posing a problem on the lack of convergence of Petviashvili's method
for periodic waves in the KdV equation. The authors also thank H. Chen, A. Dur\'{a}n, M. Johnson, and P. Torres
for discussion of various technical aspects of this manuscript.


\begin{thebibliography}{99}

%\bibitem{AD11} J. \'Alvarez and A. Dur\'an, ``A numerical scheme for periodic travelling-wave simulations
%in some nonlinear dispersive wave models", J. Comput. Appl. Math. {\bf 235}  (2011) 1790--1797.

\bibitem{AD14} J. \'Alvarez and A. Dur\'an, ``An extended Petviashvili method for the numerical generation
of traveling and localized waves", Commun. Nonlinear Sci. Numer. Simul. {\bf 19}  (2014) 2272--2283.

\bibitem{AD15} J. \'Alvarez and A. Dur\'an, ``Petviashvili type methods for traveling wave computations: I. Analysis of convergence",
J. Comput. Appl. Math. {\bf 266}  (2014) 39--51.

\bibitem{AD16} J. \'Alvarez and A. Dur\'an, ``Petviashvili type methods for traveling wave computations:
II. Acceleration with vector extrapolation methods", Math. Comput. Simulation {\bf 123}  (2016) 19--36.

\bibitem{AD17} J. \'Alvarez and A. Dur\'an, ``Numerical generation of periodic traveling wave solutions
of some nonlinear dispersive wave systems", J. Comput. Appl. Math. {\bf 316}  (2017) 29--39.

\bibitem{AT} C.J. Amick and J. F. Toland, ``Uniqueness and related analytic properties for the Benjamin--Ono equation:
a nonlinear Neumann problem in the plane", Acta Math. {\bf 167} (1991), 107--126.

\bibitem{NataliPava} J. Angulo Pava and F. Natali, ``Positivity properties of the Fourier transform and the stability of
periodic travelling-wave solutions", SIAM J. Math. Anal. {\bf 40} (2008), 1123--1151.

\bibitem{Bai} Z. Bai, H. L\"{u}, ``Positive solutions for boundary value problem
of nonlinear fractional differential equation", J. Math. Anal. Appl. {\bf 311} (2005) 495--505.

\bibitem{BBM} T.B. Benjamin, J.L. Bona and J.J. Mahony, ``Model equations for long waves in nonlinear
dispersive systems", Phil. Trans. Royal Soc. London, Ser. A {\bf 272} (1972), 47--78.

\bibitem{C} H. Chen, ``Existence of periodic traveling-wave solutions of nonlinear, dispersive wave equa-
tions", Nonlinearity {\bf 17} (2004), 2041--2056.

\bibitem{BC} H. Chen and J. Bona, ``Periodic travelling wave solutions of nonlinear dispersive evolution
equations", Discr. Cont. Dynam. Syst. {\bf 33} (2013), 4841--4873.

\bibitem{ChugPel-KdV} M. Chugunova and D. Pelinovsky, ``Two-pulse solutions in the fifth-order KdV equation:
rigorous theory and numerical approximations", DCDS B {\bf 8} (2007), 773--800.

\bibitem{ChugPel} M. Chugunova and D. Pelinovsky, ``Count of eigenvalues in the generalized eigenvalue problem",
J. Math. Phys. {\bf 51} (2010), 052901 (19 pages).

\bibitem{DC1} D. Clamond and D. Dutykh, ``Fast accurate computation of the fully nonlinear solitary surface
gravity waves", Comput. \& Fluids {\bf 84} (2013), 35--38.

\bibitem{DC3} D. Clamond and D. Dutykh, ``Accurate fast computation of steady two-dimensional surface
gravity waves in arbitrary depth", J. Fluid Mech. {\bf 844} (2018), 491--518.

%\bibitem{DK}  B. Deconinck and T. Kapitula, ``The orbital stability of the cnoidal waves
%of the Korteweg de Vries equation'', Phys. Lett. A {\bf 374} (2010), 4018--4022.

\bibitem{Var-Book} K. Deimling, {\em Nonlinear Functional Analysis} (Springer-Verlag, Heidelberg, 1985)

\bibitem{DS} L. Demanet and W. Schlag, ``Numerical verification of a gap
condition for a linearized NLS equation", Nonlinearity {\bf 19}
(2006), 829--852.

\bibitem{D18} A. Dur\'an, ``An efficient method to compute solitary wave
solutions of fractional Korteweg–de Vries equations", Int.J. Comp. Math. {\bf 95} (2018), 1362--1374.

\bibitem{DC2} D. Dutykh and D. Clamond, ``Efficient computation of steady solitary gravity waves",
Wave Motion {\bf 51} (2014), 86--99.

\bibitem{DLK} S.A. Dyachenko, P.M. Lushnikov, and A.O. Korotkevich, ``Complex singularity of a Stokes wave",
JETP Letters {\bf 98} (2014), 675--679.

\bibitem{GalPel2015} T. Gallay and D.E. Pelinovsky, ``Orbital stability in the cubic defocusing NLS equation.
Part I: Cnoidal periodic waves",  J. Diff. Eqs. {\bf 258} (2015), 3607--3638.

\bibitem{GV} A. Garijo and J. Villadelprat, ``Algebraic and analytical tools for the study of the period function",
J. Diff. Eqs. {\bf 257} (2014), 2464--2484.

\bibitem{HLP} M. Haragus, J. Li, and D.E. Pelinovsky,  ``Counting unstable eigenvalues in Hamiltonian spectral
problems via commuting operators", {\em Comm. Math. Math.} {\bf 354} (2017), 247--268.

\bibitem{JH2} V. Hur and M.A. Johnson, ``Modulational instability in the Whitham equation for
water waves", Stud. Appl. Math. {\bf 134} (2014), 120--143.

\bibitem{JH1} V.M. Hur and M.A. Johnson, ``Stability of periodic traveling waves for nonlinear dispersive equations",
SIAM J. Math. Anal. {\bf 47} (2015), 3528--3554.

\bibitem{J0} M.A. Johnson, ``Nonlinear stability of periodic traveling wave solutions of the generalized
Korteweg-de Vries equation", SIAM J. Math. Anal. {\bf 41} (2009), 1921--1947.

\bibitem{J} M.A. Johnson, ``Stability of small periodic waves in fractional KdV-type equations",
SIAM J. Math. Anal. {\bf 45} (2013), 3168--3293.

\bibitem{Kato} T. Kato, {\em Perturbation Theory for Linear Operators} (Springer–Verlag: Berlin/Heidelberg, Germany, 1995).

\bibitem{LY07} T.I. Lakoba and J. Yang, ``A generalized Petviashvili iteration method for scalar and vector
Hamiltonian equations with arbitrary form of nonlinearity", J. Comput. Phys. {\bf 226} (2007), 1668--1692.

\bibitem{LY08} T.I. Lakoba and J. Yang, ``A mode elimination technique to improve convergence of iteration
methods for finding solitary waves", J. Comput. Phys.  {\bf 226}  (2007), 1693--1709.

\bibitem{SautPilod} F. Linares, D. Pilod, and J.C. Saut, ``Dispersive perturbations of Burgers and hyperbolic
equations I: Local theory", SIAM J. Math. Anal. {\bf 46} (2014), 1505--1537.

\bibitem{Matsuno} Y. Matsuno, {\it Bilinear Transformation Method}. Mathematics in Science and Engineering. Vol 174. Academic Press. 1984

\bibitem{Molinet} L. Molinet, D. Pilod, and S. Vento, ``On well-posedness for some dispersive perturbations
of Burgers equation",  Ann. I.H.Poincare. (2018), in press.

\bibitem{Nieto} J.J. Nieto, ``Maximum principles for fractional differential equations derived from
Mittag--Leffler functions", Appl. Math. Lett. {\bf 23} (2010) 1248--1251.

\bibitem{pel-book} D.E. Pelinovsky, \textit{Localization in periodic potentials: from Schr\"{o}dinger operators to the
Gross-Pitaevskii equation}, LMS Lecture Note Series \textbf{390}. Cambridge University Press, Cambridge, 2011.

\bibitem{PelSt} D.E. Pelinovsky and Yu. A. Stepanyants, ``Convergence of Petviashvili's iteration
method for numerical approximation of stationary solutions of nonlinear wave equations",
SIAM J. Numer. Anal. {\bf 42} (2004), 1110--1127.

\bibitem{Pet} V.I. Petviashvili, ``Equation of an extraordinary soliton", Plasma Physics {\bf 2} 469 (1976).

\bibitem{Schaaf} R. Schaaf, ``A class of Hamiltonian systems with increasing periods", Journal f\"{u}r
die reine und angewandte Mathematik {\bf363} (1985), 96--109.

\bibitem{Torres} P.J. Torres, ``Existence of one-signed periodic solutions of some second-order 
differential equations via a Krasnoselskii fixed point theorem", J. Diff. Eqs. {\bf 190} (2003), 643--662.

\bibitem{yang} J. Yang, {\em Nonlinear Waves in Integrable and Nonintegrable Systems}
(SIAM, Philadelphia, 2010).

\bibitem{Zeidler} E. Zeidler,  {\em Applied Functional Analysis-Applications to Mathematical Physics} (Springer–Verlag: New York), Vol 108.

\end{thebibliography}
\end{document}